\documentclass[oneside, a4paper,11pt,reqno]{amsart}
\usepackage{yfonts}

\RequirePackage[colorlinks,citecolor=blue,urlcolor=blue]{hyperref}

\newcommand{\Ber}{\mathrm{Ber}}

\makeatletter
\def\timenow{\@tempcnta\time
  \@tempcntb\@tempcnta
  \divide\@tempcntb60
  \ifnum10>\@tempcntb0\fi\number\@tempcntb
  \multiply\@tempcntb60
  \advance\@tempcnta-\@tempcntb
  :\ifnum10>\@tempcnta0\fi\number\@tempcnta}
\makeatother

\usepackage[ddmmyyyy,hhmmss]{datetime}
\usepackage{color}

\usepackage{euscript}

\usepackage[applemac]{inputenc}
\usepackage{amssymb,amsmath,amsthm,dsfont}
\usepackage{graphics,graphicx}
\usepackage[T1]{fontenc}
\usepackage{color}
\usepackage{epsfig,psfrag}

\usepackage{enumerate}  

\newtheorem{theo}{Theorem}[section]
\newtheorem{prop}[theo]{Proposition}

\newtheorem{lemma}[theo]{Lemma}

\newtheorem{cor}[theo]{Corollary}

\newtheorem{remark}[theo]{Remark}

\newtheorem{ex}{Example}

\def\dd{\textnormal{d}}
\newcommand{\mr}{{\mathrm{m}}}

\newcommand*{\affmark}[1][*]{\textsuperscript{#1}}

\title[Stochastic flows and the $\Lambda$-coalescent with dust]{Stochastic flows and Poisson representations for the block masses of the $\Lambda$-coalescent with dust : moment asymptotics and large deviation estimates}

\author{Gr\'egoire V\'echambre\affmark[1]}

\address{\affmark[1]State Key Laboratory of Mathematical Sciences, Academy of Mathematics and Systems Science, Chinese Academy of Sciences, Beijing 100190, China}
\email{vechambre@amss.ac.cn}

\begin{document} 

\maketitle

\begin{abstract}
We develop a new methodology for the study of the $\Lambda$-coalescent with dust, based on the construction of a stochastic flow of inverses, introduced by Bertoin and Le Gall \cite{BERTOIN2005307}, in which the coalescent is naturally embedded as a nested interval-partition. This framework yields Poisson representations for the ordered block masses $(W_k(t))_{k \geq 1}$ as stochastic integrals with respect to the Poisson random measure governing the flow, enabling the use of stochastic calculus in a setting where it was not previously available. We believe this methodology to be of independent interest and applicable beyond the specific results of this paper. As a first application, we derive precise logarithmic asymptotics for the moments of $W_k(t)$ as $t \to \infty$, which reveal an interesting cutoff phenomenon related to the presence of dust. We then establish a law of large numbers and a large deviation principle for $\log(1-W_1(t))/t$ and a one-sided weak large deviation principle for $\log(W_k(t))/t$, for $k \geq 2$, with explicit rate functions. 
\end{abstract}

{\footnotesize
	
	}

\section{Introduction} \label{intro}

For a finite and non-zero measure $\Lambda$ on $[0,1]$, the $\Lambda$-coalescent is a Markovian process $(\Pi_t)_{t\geq 0}$ on partitions of $\mathbb{N}=\{1,2,\ldots\}$ introduced by \cite{pitman1999} and \cite{sagitov1999}. For any $n\geq 1$ let $\Pi^n_t$ denote the restriction of the partition $\Pi_t$ to $\{1,\ldots,n\}$. Then $(\Pi^n_t)_{t\geq 0}$ is a Markov process on a finite state space which has the following dynamics: if there are currently $p$ blocks in the partition, for $k \in \{2,\ldots,p\}$, any $k$ of the blocks merge into one with rate 
\begin{align}
\lambda_{p,k}(\Lambda):=\int_{[0,1]} r^{k-2}(1-r)^{p-k} \Lambda(dr). \label{coalrates}
\end{align}
This and the starting condition $\Pi_0=\{\{1\},\{2\},\ldots\}$ characterize the law of the partition process $(\Pi_t)_{t\geq 0}$. The $\Lambda$-coalescent is an exchangeable coalescent that generalizes the classical Kingman coalescent (which corresponds to the case $\Lambda=\delta_0$) by allowing multiple mergers instead of only binary mergers. It can also be seen as a particular case of coalescents with simultaneous multiple collisions \cite{10.1214/EJP.v5-68}. Background on the $\Lambda$-coalescent can be found in \cite{Berestycki2009,lecturebertoin2010,surveylambdacoal}. 

The $\Lambda$-coalescent is related to the genealogy of several population models \cite{10.1214/aop/1015345761,schweinsberg2003,DURRETT20051628,D08,Huillet2014,10.1214/22-EJP739}, to the genealogy of Continuous State Branching Processes (CSBPs) \cite{ref29surveyb,ref36surveyb,BLGIII,ref19surveyb,ref18surveyb}, to stable Continuous Random Trees (CRTs) \cite{ref20surveyb}, or also to pruning of trees \cite{10.1214/EJP.v10-265,romain2013,romain2015}. One of its most fundamental connections is with the $\Lambda$-Fleming-Viot flow, of which it provides the genealogy \cite{BLGI}. The latter process is valued in probability measures on $[0,1]$ and models an infinite population with constant size, with genotypes indexed by $[0,1]$, and that is subject to random neutral reproductions determined by the measure $\Lambda$. These flows have been introduced in \cite{BLGI} and implicitly in the \textit{lookdown construction} of \cite{donnellykurtz1999} which, in \cite{labbe2014a}, is unified with the construction of \cite{BLGI}. They have been studied by many authors \cite{BERTOIN2005307,ref36surveyb,BLGIII,10.1214/10-AOP629,HENARD20132054,Griffiths2014,10.3150/17-BEJ971} and can be seen as multi-type versions of $\Lambda$-Wright-Fisher diffusions, as the frequency of any hereditary subpopulation of a $\Lambda$-Fleming-Viot flow follows a $\Lambda$-Wright-Fisher diffusion \cite{BERTOIN2005307}. 

For a block $\mathcal{B} \in \Pi_t$, its \textit{mass} (also called \textit{asymptotic frequency}) is defined as $|\mathcal{B}|:=\lim_{n \rightarrow \infty}\sharp (\mathcal{B} \cap \{1,\ldots,n\})/n$. It is shown in \cite[Prop. 2.13]{labbe2014a} that, almost surely, this limit exists for all blocks in $\Pi_t$ at all $t$. A block $\mathcal{B} \in \Pi_t$ is called \textit{massive} if $|\mathcal{B}|>0$. From a $\Lambda$-coalescent process $(\Pi_t)_{t\geq 0}$, one can define the process $(\{ |\mathcal{B}|, \mathcal{B} \in \Pi_t \ \text{s.t.} \ |\mathcal{B}|>0 \})_{t\geq 0}$ of the collection of masses of its massive blocks, which is a process on mass partitions. Looking at one or the other process is equivalent thanks to Kingman's correspondence \cite{Berestycki2009,lecturebertoin2010,lecturepitman2006}. Therefore, understanding the behavior of the collection of masses of blocks is essential to understand the $\Lambda$-coalescent. Unfortunately, the distribution of masses of blocks at a given time $t$ (the \textit{entrance law from dust}) has no known explicit expression, except in the remarkable case of the Bolthausen-Sznitman coalescent \cite{pitman1999,lecturepitman2006,Berestycki2009}, and little is known about the precise behavior of masses of blocks as $t$ goes to infinity. This problem is a focus of the present paper. 

We say that a $\Lambda$-coalescent process $(\Pi_t)_{t\geq 0}$ has \textit{dust} if the cumulative mass of the collection of singletons in $\Pi_t$ is positive. There are four parameter regimes for the $\Lambda$-coalescent \cite{pitman1999,schweinsberg2000} (see also \cite{Berestycki2009,surveylambdacoal}) and, equivalently, for the $\Lambda$-Fleming-Viot flow \cite[Prop. 1.3]{labbe2014a}, namely (i) the case with finitely many massive blocks and dust, (ii) the case with infinitely many massive blocks and dust, (iii) the case with infinitely many massive blocks and no dust, (iv) the case with finitely many massive blocks and no dust. The processes $(\Pi_t)_{t\geq 0}$ and $(\Pi^n)_{n\geq 1}$ have been intensely studied in cases (iii) and (iv) \cite{10.1214/EJP.v10-265,BLGIII,10.1214/ECP.v12-1253,DRMOTA20071404,10.1214/07-AAP476,ref20surveyb,ref19surveyb,ref17surveyb,10.1214/EJP.v17-2378,10.1214/11-AAP827,betacoalgimm,ref18surveyb,10.1214/14-AAP1077,10.1214/13-AOP902,10.1214/19-AAP1462,10.1214/19-EJP354}. We are here interested in the case referred in the literature as \textit{the $\Lambda$-coalescent with dust}, corresponding to the union of cases (i) and (ii) of the above classification. As discussed in Section \ref{assumptions} below, this is equivalent to the measure $\Lambda$ satisfying the assumption \eqref{integassumption}. In the case with dust, interesting properties of the sequence of partition processes $(\Pi^n)_{n\geq 1}$ have been established as $n$ goes to infinity by many authors \cite{mohle2010,ref27ofsurvey,10.3150/10-BEJ312,10.1214/14-AAP1077,romain2013,romain2015,mohle2016,kerstingschweinsbergwakolbinger2018,kerstingwakolbinger2018,mohle2021} (see also \cite{surveylambdacoal,kerstingwakolbinger2020}) but, unlike cases (iii) and (iv), little attention has been given to the process $(\Pi_t)_{t\geq 0}$ of infinite partitions. This is perhaps surprising, as the long time behavior of the masses of blocks turns out to be more subtle in this case: in case (iv), there is a (random) finite time $T>0$ after which a single block occupies all the mass, whereas in cases (i) and (ii) the convergence to this configuration occurs only asymptotically, at a rate that is far from obvious. The case with dust allows one to define models in which massive blocks emerge (in a discontinuous way) from dispersed matter. Moreover, it has been observed recently that this case may be used in models displaying mathematically and biologically interesting behaviors. For example, in \cite{cordhumvech2022}, the authors study a family of $\Lambda$-Wright-Fisher diffusions with frequency dependent selection and environmental effects, where the measure $\Lambda$ satisfies the assumption \eqref{integassumption}, leading to four possible regimes that include in particular a regime of coexistence. The latter had been observed empirically by biologists but could not have been captured by simple mathematical models before. Motivated by this, we believe that a thorough study of the case with dust is in order, even in the classical models of the $\Lambda$-coalescent and $\Lambda$-Fleming-Viot flow. 

In the present paper we study the entrance law from dust at time $t$ of the masses of blocks of a $\Lambda$-coalescent with dust and its asymptotic behavior as $t$ is large. We denote by $W_k(t)$ the mass of the $k^{th}$ largest block at time $t$ (note that $W_k(t_1)$ and $W_k(t_2)$ may correspond to completely unrelated blocks since, at any time, blocks are ordered by non-increasing masses). It is intuitively clear that, as $t$ goes to infinity, one large block occupies a proportion of the mass increasing to $1$ as other massive blocks progressively merge with it, thus $W_1(t) \rightarrow 1$. Meanwhile, as $t$ increases, the $k^{th}$ largest block (for $k\geq 2$) is found among the smaller and smaller remaining blocks of "rare genotypes" whose total mass is less than $1-W_1(t)$, thus $W_k(t) \rightarrow 0$ for $k \geq 2$. One of our goals is to determine how fast these convergences occur in the case with dust (i.e. cases (i) and (ii)), by providing in particular a law of large numbers and a large deviation principle for $\log (1-W_1(t))/t$. As alluded to above, the case with dust leads to specific challenges that require the development of a methodology specific to this case, which we believe to be of independent interest. We note that similar challenges arise in case (iii), but that this case falls outside the scope of the present paper.

A useful point of view on the $\Lambda$-coalescent is as follows. Consider a population subject to a $\Lambda$-Fleming-Viot flow dynamic on $(-\infty,0]$ and divide the population at time $0$ into \textit{$t$-families}, where two individuals belong to the same $t$-family if and only if their ancestor from time $-t$ is common. It can be seen from Bertoin and Le Gall's correspondence \cite{BLGI} that the sizes of $t$-families are the masses of blocks in an associated $\Lambda$-coalescent. The process of the $t$-families is a \textit{nested interval-partition} in the framework of \cite{10.1214/20-AAP1641}, where such an object is constructed and shown to be associated to the $\Lambda$-coalescent. 
In other words, it is perfectly equivalent to study the masses of blocks in a $\Lambda$-coalescent or to study the sizes of $t$-families in a population subject to a $\Lambda$-Fleming-Viot flow dynamic on $(-\infty,0]$ (then, $W_k(t)$ is the size of the $k^{th}$ largest $t$-family). In the present paper, we take the latter point of view and associate the $\Lambda$-coalescent with dust to a nested interval-partition arising naturally from the flow of inverses introduced in \cite{BERTOIN2005307}. Our construction method is different from the one in \cite{10.1214/20-AAP1641} but the underlying objects are the same. This allows us to prove Poisson representations for the masses of blocks in terms of the flow of inverses. The methodology alluded to above consists precisely in exploiting these representations to harness stochastic calculus in the study of the $\Lambda$-coalescent with dust -- an approach that, to the best of our knowledge, has not been pursued before in this setting. The representations enable us to use the power of stochastic calculus to study aspects of the entrance law from dust of the masses of blocks, yielding in particular the asymptotic results for the moments of $W_k(t)$ from which the large deviation principles are derived, and bypassing the classical approximation of $(\Pi_t)_{t\geq 0}$ by $(\Pi^n_t)_{t\geq 0}$ that is standard in the literature. 

Our results display an interesting cutoff phenomenon that is related to the presence of dust. More precisely, the decay rate of $\mathbb{E}[W_k(t)^{\eta}]$ (as $t \to \infty$) is increasing in $\eta$ for fixed $k$, and increasing in $k$ for fixed $\eta$, but in both cases remains constant beyond a certain threshold. These two cutoffs have the following interpretation. New blocks whose masses are proportional to the total mass of the dust regularly emerge from the dust; therefore, the mass of the $k^{th}$ largest block cannot decay faster than the mass of the dust. As a consequence, when $\eta$ is fixed and $k$ increases, the decay rate stabilizes once it reaches the decay rate of the moments of the dust mass. When $k$ is fixed and $\eta$ increases, the decay rate is initially governed by the extreme values of the dust mass distribution, but once $\eta$ exceeds a certain threshold, $\mathbb{E}[W_k(t)^{\eta}]$ becomes dominated by the rare event where the $k$ largest blocks undergo no mergers for a long time. 

Let us mention that a companion paper \cite{vechldp} studies, for the Beta-coalescent, large deviations for the absorption time $\tau_n$ of the coalescent started with $n$ blocks, as $n \to \infty$, via the analysis of Laplace transforms of integral functionals. The present paper addresses a complementary question -- the long-time behavior of the block masses $(W_k(t))_{k \geq 1}$ of the $\Lambda$-coalescent with dust -- via an entirely different framework based on stochastic flows and Poisson representations. 

\subsection{Related models and perspectives}

We now briefly discuss some related models and perspectives.

\textit{Flows of subordinators.} While $\Lambda$-Fleming-Viot flows model constant size populations, flows of subordinators (the flow version of CSBPs) model populations with similar dynamics but non-constant size. Their genealogy has also been studied \cite{ref29surveyb} and, in that context, a similar problem to ours is studied in \cite{foucartmamallein2019,foucartmohle2022} via the inverse flow. 
In contrast to our setting, the independence enjoyed by subpopulations in the non-constant size case leads to a Poisson structure for the $t$-families, making the analysis more tractable. In our case, the ordered block masses arise from an exchangeable structure that results in strong dependencies, and developing tools that can efficiently address this feature is a motivation in itself.

\textit{Evolving coalescent.} A realization of the $\Lambda$-coalescent can be seen as the genealogical structure of a population sampled at a given time. By letting the sampling time evolve forward in time, one obtains a Markov process of genealogical structures called the \textit{evolving coalescent}, see for example \cite{10.1214/EJP.v17-2378,10.1214/EJP.v19-3332,kerstingwakolbinger2020}. Finding the appropriate representation and state space for this Markov process is non-trivial and several approaches have been proposed \cite{treevalued2013,10.1214/18-EJP166,10.1214/20-AAP1641}. In the present paper, the $\Lambda$-coalescent (in the case with dust) is seen as a function of the flow of inverses starting at $0$. Provided that one consistently constructs the flow of inverses starting at all times (along with its Poisson background), this would yield another construction of the evolving coalescent in the case with dust, and the Poisson representation from Theorem \ref{represwktviamut} would naturally extend to the masses of blocks of the evolving coalescent. 

\textit{Infinite allele model.} Consider the $\Lambda$-coalescent restricted to $\{1,\ldots,n\}$, let every block freeze at some rate $\theta>0$, and only allow mergings between unfrozen blocks. This is interpreted biologically as looking at the allelic types of $n$ individuals, with merging events corresponding to coalescents of ancestral lineages and freezing events corresponding to appearances of mutations leading to new allelic types. After letting time go to infinity, the resulting partition is called the \textit{allelic partition} and the sizes of its blocks are called the \textit{allele frequency spectrum}. These objects have attracted a lot of interest \cite{bj/1141136647,ref20surveyb,10.1214/EJP.v13-494,mohle2009,mohle2010,10.1214/13-AIHP546} (see also \cite{Berestycki2009,surveylambdacoal}). We believe that the construction and representation of the present paper could be adapted to include the effect of mutations and study properties of the allele frequency spectrum in the case with dust. 

\subsection{Assumptions and notations} \label{assumptions}

Let $\Lambda$ be a non-zero finite measure on $[0,1]$. In all this paper we assume that 
\begin{align}
\Lambda(\{1\})=0, \ \int_{[0,1)} r^{-1} \Lambda(dr) < \infty, \label{integassumption}
\end{align}
with the convention $1/0=\infty$. If $\Lambda(\{1\})>0$, all blocks merge into a single one at rate $\Lambda(\{1\})$. The condition $\Lambda(\{1\})=0$ is meant to eliminate this degenerate case. Assuming $\Lambda(\{1\})=0$, the condition $\int_{[0,1)} r^{-1} \Lambda(dr)<\infty$ is equivalent to the almost sure presence of dust, i.e. cases (i) and (ii) mentioned above, as was shown by \cite{pitman1999} (see also \cite{Berestycki2009,surveylambdacoal}, and see \cite[Prop. 1.3]{labbe2014a} for the $\Lambda$-Fleming-Viot flow). We also note that \eqref{integassumption} implies $\Lambda(\{0\})=0$ so there is no Kingman component in the coalescent process. 
The case (i) "finitely many massive blocks and dust" is a sub-case of the case with dust. It corresponds to the condition 
\begin{align}
\Lambda(\{1\})=0, \ \int_{[0,1)} r^{-2} \Lambda(dr) < \infty. \label{strongintegassumption}
\end{align}

We define 
\begin{align}
\theta(\Lambda):=\sup \{ \eta \geq 0, \int_{(1/2,1)} (1-r)^{-\eta} \Lambda(dr)<\infty \}. \label{domainboundary}
\end{align}
In some of our results, we will additionally assume that 
\begin{align}
\text{Either} \ \theta(\Lambda)=\infty, \ \text{or} \ 0<\theta(\Lambda)<\infty \ \text{and} \ \int_{(1/2,1)} (1-r)^{-\eta} \Lambda(dr) \underset{\eta \to \theta(\Lambda)}{\longrightarrow} \infty. \label{continuitycondition}
\end{align}
Let us define 
\begin{align}
\phi_S(\lambda) := \int_{(0,1)} ( 1-(1-r)^\lambda ) r^{-2}\Lambda(dr). \label{laplaceexponentlk}
\end{align}
Note that, under assumption \eqref{integassumption}, $\phi_S(\lambda)\in(-\infty,\infty)$ if $\lambda>-\theta(\Lambda)$ and $\phi_S(\lambda)=-\infty$ if $\lambda<-\theta(\Lambda)$. Under assumption \eqref{continuitycondition} this transition is continuous (i.e. $\phi_S(\lambda) \to -\infty$ as $\lambda \to -\theta(\Lambda)$). The function $\phi_S(\cdot)$ is the Laplace exponent of a subordinator $S$, rigorously defined in Section \ref{model3}, which is related to the dust mass and plays a central role in the statement of our results.

\begin{ex}[Beta-coalescent] \label{betacoal}
For any $a,b>0$ we set $\Lambda_{a,b}(dr):=r^{a-1}(1-r)^{b-1}dr$. Then the transitions rates $\lambda_{p,k}(\Lambda_{a,b})$ from \eqref{coalrates} can be expressed as $\lambda_{p,k}(\Lambda_{a,b})=B(a+k-2,b+p-k)$ where $B(\cdot,\cdot)$ is the beta function. The condition \eqref{integassumption} holds true if and only if $a>1$ and the condition \eqref{strongintegassumption} holds true if and only if $a>2$. We note that $\theta(\Lambda_{a,b})=b$ and that $\Lambda_{a,b}$ satisfies the condition \eqref{continuitycondition}. An important case is when $a=2-\alpha, b=\alpha$ for some $\alpha \in (0,2)$. That case is called the $\mathrm{Beta}(2-\alpha,\alpha)$-coalescent. 
\end{ex}

\subsection{Main results and applications} \label{mainresults}

In this section, we state our main results. We work with the process $(W_k(t))_{k \geq 1, t \geq 0}$ of ordered block masses of the $\Lambda$-coalescent with dust as constructed in Section \ref{model2} below. We set $M_k(t):=\sum_{1 \leq j \leq k} W_j(t)$. Our first main result establishes a stochastic integral representation for the masses of blocks $W_k(t)$ in terms of a Poisson random measure. This representation is a cornerstone of our approach and underpins all subsequent results.
\begin{theo} [Stochastic integral representation for $W_k(t)$] \label{intsto}
Assume that \eqref{integassumption} holds true. There exists a Poisson random measure $N(ds,dr,du)$ on $(0,\infty) \times (0,1)^2$ with intensity measure $ds \otimes r^{-2} \Lambda(dr) \otimes du$ and, for any $k \geq 1$, processes $K^k_{t}(r,u)$ and $H^k_{t}(r,u)$ indexed by $(t,r,u) \in [0,\infty) \times (0,1)^2$, progressively measurable in $t$ and constructed in Section \ref{model3} below, such that almost surely, for all $t\geq 0$,
\begin{align}
W_k(t) & = \int_{(0,t] \times (0,1)^2} K^k_{s-}(r,u) N(\dd s, \dd r, \dd u), \label{itoforwk0} \\
M_k(t) & = \int_{(0,t] \times (0,1)^2} H^k_{s-}(r,u) N(\dd s, \dd r, \dd u). \label{itoformk0}
\end{align}
\end{theo}
Thanks to the above result, we are now able to use stochastic calculus to study $W_k(t)$. 
\begin{theo} [Pseudo-generator formula for $W_k(t)$] \label{ito}
Assume that \eqref{integassumption} holds true. For any Lipschitz function $f:[0,1]\rightarrow\mathbb{R}$ and $k \geq 1$, $(t \mapsto \mathbb{E}[f(W_k(t))])$ and $(t \mapsto \mathbb{E}[f(M_k(t))])$ are of class $\mathcal{C}^1$ and for any $t \geq 0$, 
\begin{align}
\frac{d}{dt} \mathbb{E}[f(W_k(t))] & = \int_{(0,1)^2} \mathbb{E} \left [ f \big (W_k(t)+K^k_{t}(r,u) \big ) - f (W_k(t) ) \right ] r^{-2} \Lambda(dr) du. \label{derfwkt} \\
\frac{d}{dt} \mathbb{E}[f(M_k(t))] & = \int_{(0,1)^2} \mathbb{E} \left [ f \big (M_k(t)+H^k_{t}(r,u) \big ) - f (M_k(t) ) \right ] r^{-2} \Lambda(dr) du. \label{derfmkt}
\end{align}
\end{theo}
Theorems \ref{intsto} and \ref{ito} provide insights on the entrance law from dust of the masses of blocks and allow us to study their long time behavior. We now present such results. For $k \geq 1$ let 
\begin{align}
\lambda_k(\Lambda) := \int_{(0,1)} (1-(1-r)^k - kr(1-r)^{k-1}) r^{-2} \Lambda(dr). \label{pushingrates}
\end{align}
Note from \eqref{coalrates} that $\lambda_k(\Lambda)=\sum_{\ell =2}^k \binom{k}{\ell}\lambda_{k,\ell}(\Lambda)$, that is, for any $n\geq k$, $\lambda_k(\Lambda)$ is the total transition rate of $(\Pi^n_t)_{t\geq 0}$ when the partition currently contains $k$ blocks. Also, the sequence $(\lambda_k(\Lambda))_{k\geq 1}$ is increasing, since $(1-(1-r)^k - kr(1-r)^{k-1})$ is the probability for a binomial random variable with parameter $(k,r)$ to be larger or equal to $2$. Finally we note that $\lambda_k(\Lambda) \rightarrow \int_{(0,1)} r^{-2} \Lambda(dr) \in (0,\infty]$ as $k \rightarrow \infty$. Let us define 
\begin{align}
N(\Lambda) := \inf \left \{ k \geq 1 \ \text{s.t.} \ \lambda_k(\Lambda) \geq \phi_S(1) \right \}. \label{defnlambda}
\end{align}
Since $\int_{(0,1)} r^{-2} \Lambda(dr)>\phi_S(1)$, the above shows that $N(\Lambda)$ is finite and, since $\lambda_2(\Lambda)=\Lambda((0,1))<\int_{(0,1)} r^{-1} \Lambda(dr)=\phi_S(1)$, we necessarily have $N(\Lambda) \geq 3$. 
\begin{remark} \label{propnlambd}
Assume that \eqref{integassumption} holds true, then we have $\lambda_k(\Lambda)<(k-1)^2 \lambda_2(\Lambda)$ for $k\geq 3$ and $N(\Lambda) > 2 \vee \sqrt{\phi_S(1)/\lambda_2(\Lambda)}$. This shows in particular that, the more $\Lambda$ has mass around $0$, the larger $N(\Lambda)$ is. Table \ref{tab:Nlambda} illustrates this behavior in the case of the $\mathrm{Beta}(2-\alpha,\alpha)$-coalescent. The proofs of the bounds stated in this remark are given in Appendix \ref{proofestnlbd}. 
\end{remark}

\begin{table}[h]
\centering
\begin{tabular}{|c|c|c|c|c|c|c|c|c|c|}
\hline
$\alpha$ & $0.1$ & $0.2$ & $0.3$ & $0.4$ & $0.5$ & $0.6$ & $0.7$ & $0.8$ & $0.9$ \\
\hline
$N(\Lambda_{2-\alpha,\alpha})$ & $4$ & $4$ & $4$ & $4$ & $5$ & $6$ & $7$ & $8$ & $14$ \\
\hline
\end{tabular}
\caption{Values of $N(\Lambda_{2-\alpha,\alpha})$ for the $\mathrm{Beta}(2-\alpha,\alpha)$-coalescent, for several values of $\alpha \in (0,1)$.}
\label{tab:Nlambda}
\end{table}

The following result is derived from Theorem \ref{ito} and provides the asymptotic behaviors of the moments \footnote{Here and throughout, we use the term 'moment of order $\eta$ of a random variable $Z$' for $\mathbb{E}[Z^\eta]$ with $\eta \in \mathbb{R}$.} of the ordered masses of blocks. The cutoff phenomenon described in the Introduction appears explicitly in the form of the exponential decay rate $\phi_S(\eta) \wedge \lambda_k(\Lambda)$.
\begin{theo} [Long time behavior of moments] \label{longtimebehavth}
Assume that \eqref{integassumption} holds true. 
\begin{itemize}
\item[(i)] For any $\eta \in (-\theta(\Lambda),\infty)$, we have 
\begin{align} 
\frac1{t}\log \mathbb{E}[(1-W_1(t))^{\eta}] & \underset{t \to \infty}{\longrightarrow} -(\phi_S(\eta) \wedge \lambda_2(\Lambda)). \label{lgtimeboundw1t}
\end{align}
If, additionally, \eqref{continuitycondition} holds true, then \eqref{lgtimeboundw1t} extends to the case $\eta \in (-\infty,-\theta(\Lambda)]$ where the right-hand side equals $\infty$. 
\item[(ii)] For $k \in \{2,..., N(\Lambda)-1\}$ and $\eta \in [0,\infty)$, or $k \geq N(\Lambda)$ and $\eta \in [0,1]$, we have 
\begin{align} 
\frac1{t}\log \mathbb{E}[W_k(t)^{\eta}] & \underset{t \to \infty}{\longrightarrow} -(\phi_S(\eta) \wedge \lambda_k(\Lambda)). \label{lgtimeboundwkt}
\end{align}
\end{itemize}
\end{theo}
At first glance, one might worry that Theorem \ref{longtimebehavth}(ii) contradicts the fact that $\sum_{k\geq 1} \mathbb{E}[W_k(t)]\leq 1$. Indeed, the cutoff phenomenon implies that the decay rate of $\mathbb{E}[W_k(t)]$ stabilizes for large $k$, so that infinitely many terms in the sum $\sum_{k \geq 1} \mathbb{E}[W_k(t)]$ share the same exponential decay rate. This is not a contradiction, however, since terms sharing the same exponential decay rate may still differ in their subexponential behavior, leaving room for $\sum_{k\geq 1} \mathbb{E}[W_k(t)]$ to be finite for all $t$. 

When $\eta<0$, the moments $\mathbb{E}[W_k(t)^{\eta}]$ are not always finite under our usual assumptions. For example, if \eqref{strongintegassumption} holds true, we have $\mathbb{P}(W_k(t)=0)>0$ and therefore $\mathbb{E}[W_k(t)^{\eta}]=\infty$. This is why, in Theorem \ref{longtimebehavth}(ii), we only study $\mathbb{E}[W_k(t)^{\eta}]$ when $\eta\geq0$, and this restriction is not merely technical. However, it is a consequence of Proposition \ref{minoyklem2} that, when $\eta \in (-\theta(\Lambda),0)$, 
\begin{align} 
\liminf_{t \to \infty} \frac1{t}\log \mathbb{E}[W_k(t)^{\eta}] \geq -\phi_S(\eta) = -(\phi_S(\eta) \wedge \lambda_k(\Lambda)). \label{caseetaneg}
\end{align}

For Theorem \ref{longtimebehavth}(ii), our argument is based on Theorem \ref{ito} and a conditional Jensen inequality that requires the concavity of $x \mapsto x^{\eta}$, hence the restriction $\eta \in [0,1]$. The case "$\eta>1$ and $k<N(\Lambda)$" is then obtained by monotonicity from the case $\eta=1$. We leave the case "$\eta > 1$ and $k \geq N(\Lambda)$" as an open question. 

\begin{remark}[Bolthausen-Sznitman coalescent] \label{rmkbs}
The Bolthausen-Sznitman coalescent is the $\mathrm{Beta}(2-\alpha,\alpha)$-coalescent with $\alpha=1$. It falls under case (iii) of the classification (see the Introduction). In this case, $(W_k(t))_{k \geq 1}$ follows the Poisson-Dirichlet distribution with parameters $(e^{-t},0)$, see \cite[Thm.~6.2]{Berestycki2009}. This distribution admits a stick-breaking representation: if $(V_k)_{k \geq 1}$ are independent random variables with distribution $B(1-e^{-t},ke^{-t})$ respectively, then the sequence $(P_k)_{k \geq 1}$ defined by $P_1 = V_1$ and $P_k = V_k \prod_{j=1}^{k-1}(1-V_j)$ for $k \geq 2$ has the same distribution as $(W_k(t))_{k \geq 1}$ up to a reordering. Using this representation, one can compute the moments of $P_k$ rather than those of $W_k(t)$, which yields analogues of Theorem \ref{longtimebehavth} for the Bolthausen-Sznitman coalescent. Precisely, for $\eta > 0$,
\begin{align*} 
\mathbb{E}[(1-P_1)^{\eta}] & = \frac{\Gamma(\eta+e^{-t})}{\Gamma(e^{-t})\Gamma(\eta+1)} \underset{t\rightarrow \infty}{\sim} \frac{e^{-t}}{\eta}, \\
\forall k \geq 2, \ \mathbb{E}[P_k^{\eta}] & = \frac{\Gamma(1+\eta-e^{-t})\Gamma(1+(k-1)e^{-t})}{\Gamma(1-e^{-t})\Gamma(1+\eta+(k-1)e^{-t})} \prod_{j=1}^{k-1} \frac{\Gamma(\eta+je^{-t})\Gamma(1+(j-1)e^{-t})}{\Gamma(je^{-t})\Gamma(1+\eta+(j-1)e^{-t})} \\& \underset{t\rightarrow \infty}{\sim} \eta^{-(k-1)}(k-1)! e^{-t(k-1)}. 
\end{align*}
Moreover, it is not difficult to see that $\lambda_{k,\ell}(\Lambda_{1,1})=(\ell-2)! (k-\ell)!/(k-1)!$, so $\lambda_k(\Lambda_{1,1})=k-1$. Furthermore, $\phi_S(\eta)=\infty$ for any $\eta>0$. The asymptotic behaviors obtained above are therefore consistent with those that Theorem \ref{longtimebehavth} would predict. 
\end{remark}

We now assume that \eqref{continuitycondition} holds true (additionally to \eqref{integassumption}) and note that $\phi_S(\eta)=-\infty$ when $\eta \in (-\infty,-\theta(\Lambda)]$. For any $k \geq 2$, we define a function $\mathcal{J}_k: \mathbb{R} \to [0,\infty)$ as the Legendre transform of the convex and non-decreasing function $\eta \mapsto -(\phi_S(-\eta) \wedge \lambda_k(\Lambda))$: 
\begin{align}
\mathcal{J}_k(x):=\sup \{ \eta x + (\phi_S(-\eta) \wedge \lambda_k(\Lambda)); \eta \in \mathbb{R} \}. \label{defiaslegendrezeta}
\end{align}
Note that $\mathcal{J}_k(x)=\infty$ when $x<0$ and $\mathcal{J}_k(x)\in[0,\infty)$ when $x \geq 0$. Let $\gamma_k := \inf \{ \eta \geq 0, \phi_S(\eta)>\lambda_k(\Lambda) \}$. If $\gamma_k=\infty$ (which can occurs if $\Lambda$ satisfies \eqref{strongintegassumption}) then $\mathcal{J}_k(\cdot)$ is just the Legendre transform of $\eta \mapsto -\phi_S(-\eta)$ that we denote by $\mathcal{J}(\cdot)$. If $\gamma_k<\infty$, one can see that $\mathcal{J}_k(x)=\lambda_k(\Lambda)-\gamma_k x$ for $x \in [0,\phi_S'(\gamma_k)]$ and $\mathcal{J}_k(\cdot)$ coincides with $\mathcal{J}(\cdot)$ on $[\phi_S'(\gamma_k),\infty)$. We refer to Figure \ref{represgraphicratefct} for a graphical representation of the functions $\mathcal{J}(\cdot)$ and $\mathcal{J}_2(\cdot)$. As justified in the proof of Corollary \ref{corlgn} (see Section \ref{longtimebehavccl}), $\mathcal{J}_k(\cdot)$ has a single zero at $\phi_S'(0)$, it is decreasing on $[0,\phi_S'(0)]$ and increasing on $[\phi_S'(0),\infty)$. Since $0<\phi_S'(\gamma_k)<\phi_S'(0)$ by strict concavity of $\phi_S(\cdot)$, the zero of $\mathcal{J}_k(\cdot)$ is located after the interval on which this function is affine. 
\begin{figure}
\centering
\includegraphics[scale=0.25]{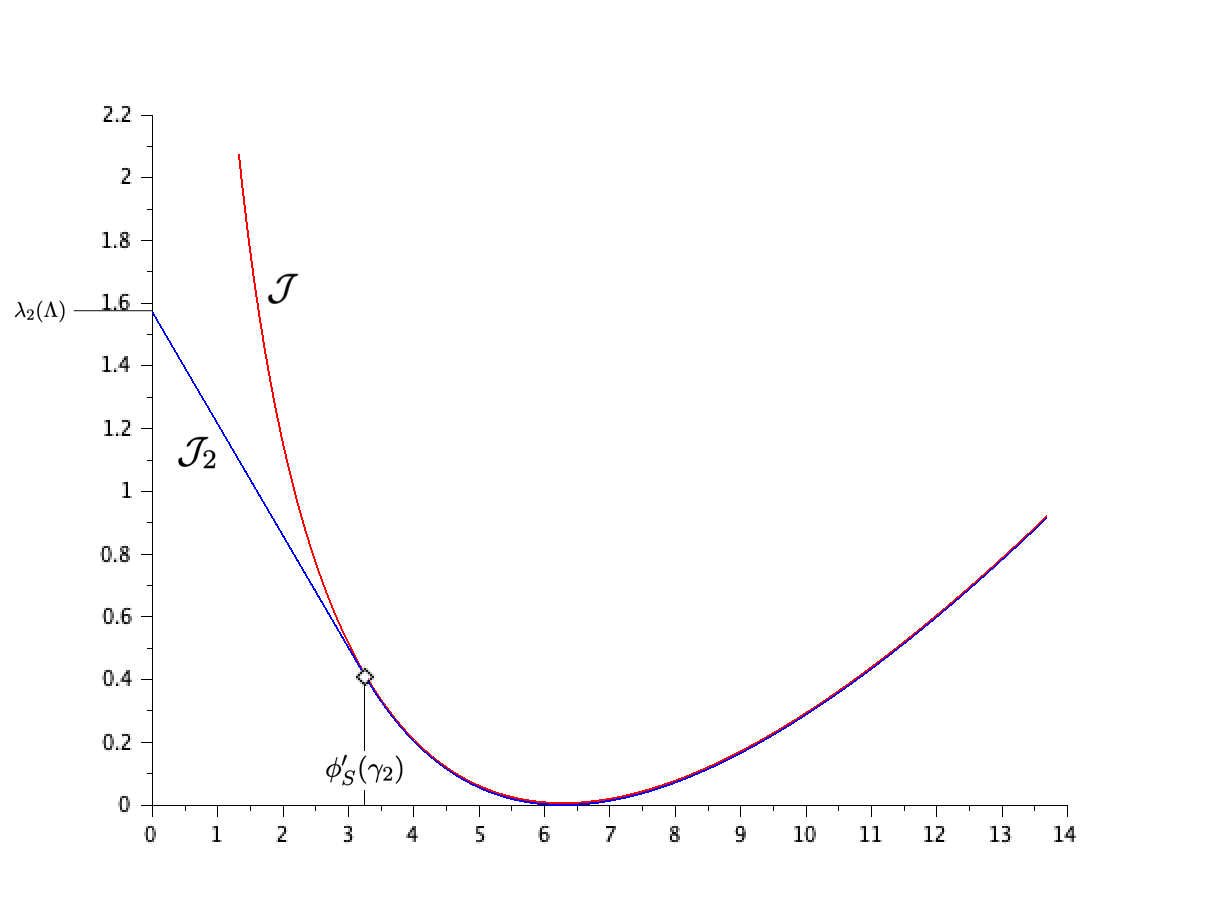} 
\caption{\small{Graphical representation of the functions $\mathcal{J}(\cdot)$ and $\mathcal{J}_2(\cdot)$ when $\Lambda=\Lambda_{3/2,1/2}$. The graph of $\mathcal{J}_2(\cdot)$ is a straight line on $[0,\phi_S'(\gamma_2)]$ and coincides with the graph of $\mathcal{J}(\cdot)$ on $[\phi_S'(\gamma_2),\infty)$.}}
\label{represgraphicratefct}
\end{figure}
Theorem \ref{longtimebehavth} entails the following Large Deviation Principle (LDP) for the mass of the largest block. 
\begin{theo} [LDP for the largest block] \label{thmpgd}
Assume that \eqref{integassumption} and \eqref{continuitycondition} hold true, then $-\log (1-W_1(t))/t$ satisfies a LDP with good rate function $\mathcal{J}_2(\cdot)$ and speed $t$. More precisely, for any closed set $F \subset \mathbb{R}$, 
\begin{align}
\limsup_{t \rightarrow \infty} \frac{1}{t} \log \mathbb{P} \left ( \frac{-\log (1-W_1(t))}{t} \in F \right ) \leq -\inf_{x \in F} \mathcal{J}_2(x), \label{ldpupperbound}
\end{align}
and for any open set $G \subset \mathbb{R}$, 
\begin{align}
\liminf_{t \rightarrow \infty} \frac{1}{t} \log \mathbb{P} \left ( \frac{-\log (1-W_1(t))}{t} \in G \right ) \geq -\inf_{x \in G} \mathcal{J}_2(x). \label{ldplowerbound}
\end{align}
\end{theo}
Note from Cram\'er's theorem (see \cite[Thm. 2.2.3]{dembozeitouni}) that the subordinator $S$ with Laplace exponent $\phi_S(\cdot)$ (see Sections \ref{assumptions} and \ref{model3}) satisfies a LDP with rate function $\mathcal{J}(\cdot)$. Note also that, in a $\Lambda$-coalescent, two given blocks do not merge on a given time interval of length $\ell$ with probability $e^{-\ell \lambda_2(\Lambda)}$. This allows to give the following interpretation of Theorem \ref{thmpgd}. On the one hand, the cheapest strategy to achieve abnormally small or large values of $W_1(t)$ of the form $W_1(t) \approx 1-e^{-xt}$ with $x \geq \phi_S'(\gamma_2)$ is given by large deviations of the subordinator $S$, i.e. by abnormal values for the dust mass. On the other hand, the cheapest strategy to achieve abnormally small values of $W_1(t)$ of the form $W_1(t) \approx 1-e^{-xt}$ with $x \leq \phi_S'(\gamma_2)$ is given by a large deviations type behavior of $S$ during some time, followed with a non-merging of the two largest blocks in the remaining time until $t$. We note that, while these strategies are seen from Theorem \ref{thmpgd} to be the cheapest ones, it is not clear \textit{a priori} that they are and what is the main source of large deviations of the form $W_1(t) \approx 1-e^{-xt}$ for a given value of $x$. While guessing these strategies can allow one to build suitable events that lead to the large deviation lower bound \eqref{ldplowerbound}, one needs another way to prove the large deviation upper bound \eqref{ldpupperbound} that shows the optimality of these strategies. Our proof of Theorem \ref{thmpgd} is based on Theorem \ref{longtimebehavth} (whose proof is mostly analytical and relies on Theorem \ref{ito}) and G\"artner-Ellis Theorem. This yields the upper bound \eqref{ldpupperbound} and a partial version of the lower bound \eqref{ldplowerbound}. In the case of large deviations of the form $W_1(t) \approx 1-e^{-xt}$ with $x \leq \phi_S'(\gamma_2)$, we build events based on the above-mentioned combination strategy to complete the proof of the lower bound \eqref{ldplowerbound}. 

An application of Theorem \ref{thmpgd} is the following law of large numbers for the mass of the largest block. 
\begin{cor} [Law of large numbers] \label{corlgn}
Assume that \eqref{integassumption} and \eqref{continuitycondition} hold true, then we have 
\begin{align}
\frac{-\log (1-W_1(t))}{t} \overset{a.s.}{\underset{t \rightarrow \infty}{\longrightarrow}} \phi_S'(0) =\int_{(0,1)} |\log (1-r)| r^{-2}\Lambda(dr). \label{lgn}
\end{align}
\end{cor}

\begin{remark} \label{corresp}
It is shown in \cite{kerstingwakolbinger2018} (see also \cite{ref27ofsurvey,vechldp}) that the absorption time $\tau_n$ of the process $(\Pi^n_t)_{t\geq 0}$ satisfies $\tau_n / \log n \to 1/\int_{(0,1)} |\log(1-r)| \, r^{-2} \Lambda(dr)$ in probability. The fact that this limit is the reciprocal of the one appearing in \eqref{lgn} is not a coincidence: both results are related to a subordinator $S$ with $\mathbb{E}[S_1]=\int_{(0,1)} |\log(1-r)| \, r^{-2} \Lambda(dr)$ (see Section \ref{model3}), which governs the typical behavior of $-\log(1-W_1(t))$ in our case, and whose inverse evaluated at $\log n$ approximates $\tau_n$ in the case of \cite{kerstingwakolbinger2018}.
\end{remark}

As a consequence of Theorem \ref{longtimebehavth}, we also obtain the following one-sided weak LDP for $W_k(t)$ when $k \in \{2,..., N(\Lambda)-1\}$. 
\begin{theo}[one-sided weak LDP for other blocks] \label{thmpgdwk}
Assume that \eqref{integassumption} and \eqref{continuitycondition} hold true then, for any $k \in \{2,..., N(\Lambda)-1\}$ and any compact set $K \subset (-\infty,\phi_S'(0)]$, 
\begin{align}
\limsup_{t \rightarrow \infty} \frac{1}{t} \log \mathbb{P} \left ( \frac{-\log W_k(t)}{t} \in K \right ) \leq -\inf_{x \in K} \mathcal{J}_k(x), \label{ldpupperboundwk}
\end{align}
and for any open set $G \subset (-\infty,\phi_S'(0))$, 
\begin{align}
\liminf_{t \rightarrow \infty} \frac{1}{t} \log \mathbb{P} \left ( \frac{-\log W_k(t)}{t} \in G \right ) \geq -\inf_{x \in G} \mathcal{J}_k(x). \label{ldplowerboundwk}
\end{align}
\end{theo}
The one-sided and weak nature of the LDP both reflect the same fundamental obstruction: as noted in the discussion above \eqref{caseetaneg}, the moments $\mathbb{E}[W_k(t)^\eta]$ may be infinite for $\eta < 0$ under our assumptions. Removing the one-sided restriction would require finiteness of the moments of order $\eta < 0$ to control large deviations to the left, and upgrading the weak LDP to a full LDP would require finiteness for $\eta$ in a neighborhood of $0$ to establish exponential tightness. In both cases, the examples mentioned above show that a one-sided weak LDP is optimal within our class of models. 

\begin{remark}
Proposition \ref{minoyklem2} and Remark \ref{completelowerboundmkt} allows one to extend Theorem \ref{longtimebehavth}(i) to 
$1-M_k(t)$, for any $k \geq 1$: namely, for any $\eta \in (-\theta(\Lambda),\infty)$,
\begin{align*}
\frac{1}{t}\log \mathbb{E}[(1-M_k(t))^{\eta}] \underset{t \to \infty}{\longrightarrow} -(\phi_S(\eta) \wedge \lambda_{k+1}(\Lambda)).
\end{align*}
Using this, the proof of Theorem \ref{thmpgd} then extends verbatim to yield a full LDP for $-\log(1-M_k(t))/t$, for any $k \in \{2,\ldots,N(\Lambda)-1\}$, upon replacing $W_1(t)$ by $M_k(t)$ and $\lambda_2(\Lambda)$ by $\lambda_{k+1}(\Lambda)$ throughout; the resulting rate function is $\mathcal{J}_{k+1}(\cdot)$. 
\end{remark}

The rest of the paper is organized as follows. Section \ref{secstochflows} introduces the stochastic flow framework, its connection to the $\Lambda$-coalescent with dust, and the Poisson representation of Theorem \ref{represwktviamut}, which underpins all subsequent results. Section \ref{longtimebehav} derives the long time asymptotic results, namely Theorems \ref{longtimebehavth}, \ref{thmpgd} and \ref{thmpgdwk}, and Corollary \ref{corlgn}, taking the results of Sections \ref{secstochflows}, \ref{poisrep} and \ref{itoform} as given. Section \ref{poisrep} proves the Poisson representation of Theorem \ref{represwktviamut} together with the supporting results stated in Section \ref{secstochflows}. Section \ref{itoform} proves the stochastic integral representation of Theorem \ref{intsto} and the pseudo-generator formula of Theorem \ref{ito}. The appendices collect the proofs of the foundational properties of the flow of inverses (Appendix \ref{propflows}), auxiliary results relating the flow to the $\Lambda$-coalescent (Appendix \ref{coaltrajy}), and the proof of Remark \ref{propnlambd} (Appendix \ref{proofestnlbd}).

\section{The stochastic flow framework} \label{secstochflows}

\subsection{$\Lambda$-Fleming-Viot flow and flow of inverses} \label{model}

Let $N(ds,dr,du)$ be a Poisson random measure on $(0,\infty) \times (0,1)^2$ with intensity measure $ds \otimes r^{-2} \Lambda(dr) \otimes du$. $N$ can be seen as a random collection of mass $1$ atoms $(s,r,u) \in (0,\infty) \times (0,1)^2$. We refer to an atom $(s,r,u) \in N$ as \textit{a jump}, to $s$ as its \textit{time component}, and to $r$ and $u$ as respectively its \textit{$r$-component} and its \textit{$u$-component}. We define the set of jumping times by $J_N:=\{s>0, \exists (r,u) \in (0,1)^2 \ \text{s.t.} \ (s,r,u) \in N\}$. For any $t\geq 0$, $\mathcal{F}_{t}$ denotes the sigma-field generated by the random measure $N(\cdot \cap (0,t] \times (0,1)^2)$. 

The $\Lambda$-Fleming-Viot flow is defined as the solution of the following SDE: 
\begin{align}
X_{w}(x) = x + \int_{(0,w] \times (0,1)^2} r \left ( \mathds{1}_{\{u\leq X_{s-}(x)\}}-X_{s-}(x) \right ) N(ds,dr,du), \label{defflowxbysde}
\end{align} 
almost surely for all $x \in [0,1]$ and $w \geq 0$. By \cite[Thm. 4.4]{10.1214/10-AOP629} this SDE defines a unique flow $(X_{w}(x), x \in [0,1], w \geq 0)$ that is called the $\Lambda$-process in \cite{BERTOIN2005307}. 
A jump $(s,r,u) \in N$ has the following interpretation: at time $s$ the individual "located" at $u \in [0,1]$ produces an offspring of size $r$ that replaces an identical amount of individuals chosen uniformly in the population. The quantity $X_{w}(x)$ represents the amount of individuals, in the population at time $w$, whose ancestor from time $0$ lies in $[0,x]$. For any $x \in [0,1]$, the process $(X_{w}(x))_{w\geq 0}$ is the so-called \textit{$\Lambda$-Wright-Fisher diffusion} with initial value $x$. 

In this paper, we are interested in the genealogy of a population that underwent the dynamic \eqref{defflowxbysde} from a very long time until present. The designated tool to study this is the so-called flow of inverses, see \cite{BERTOIN2005307}. Its heuristic definition and interpretation is as follows. We fix $t>0$ and consider a flow $(X_{-t,w}(x), x \in [0,1], w \in [-t,0])$ representing a population undergoing the dynamics \eqref{defflowxbysde} on the time interval $[-t,0]$, $0$ representing present day and $-t$ representing the starting time in the past. Defining a consistent collection $X_{-t,\cdot}(\cdot)$ solving \eqref{defflowxbysde} for all $t \in \mathbb{R}_+$ is non-trivial; following \cite{BLGI}, one can instead consider the dual flow associated with the $\Lambda$-coalescent, which is well-defined and represents a population undergoing the same dynamic. We then denote by $X_{-t,0}^{-1}(\cdot)$ the generalized inverse of the non-decreasing function $X_{-t,0}(\cdot)$. The link with the genealogy is now clear: a subinterval of $[0,1]$ on which $X_{-t,0}^{-1}(\cdot)$ is constant corresponds to a set of individuals, in the population at present time, whose ancestor from time $-t$ in the past is common. We refer to such subinterval as a \textit{$t$-family}. By increasing $t$ we look further in the past and obverse mergers of $t$-families when they get connected by new potential ancestors. 

In \cite{BERTOIN2005307} (see also \cite[Sec. 3.2]{BLGI}), the coalescing flow related to the genealogy of the $\Lambda$-process, called \textit{flow of inverses}, is defined as described above, by taking the generalized inverse in flows of bridges. In our case, we define it from an SDE and then justify that it is equal in law to the one from \cite{BERTOIN2005307}. More precisely, we consider the stochastic flow $(Y_{0,t}(y), y \in [0,1], t \geq 0)$ solving 
\begin{align}
Y_{0,t}(y) = y + \int_{(0,t] \times (0,1)^2} \left ( \mr_{r,u}(Y_{0,s-}(y))-Y_{0,s-}(y) \right ) N(ds,dr,du), \label{defflowybysde}
\end{align} 
almost surely for all $y \in [0,1]$ and $t \geq 0$, where $\mr_{r,u}(z):=\mathsf{Median}\{\frac{z-r}{1-r}, \frac{z}{1-r}, u\}$ for $z\in[0,1]$. 
For a population that underwent the dynamic \eqref{defflowxbysde} on $[-t,0]$, $Y_{0,t}(y)$ represents the position of the ancestor from time $-t$ of an individual "located" at position $y$ at time $0$. In particular, a jump $(s,r,u) \in N$ in \eqref{defflowybysde} has the following interpretation: at time $-s$ a new potential ancestor appears at the location $u$, all individuals from time $-(s-)$ that are located in the interval 
\begin{align}
I_{r,u}:=[u(1-r),u(1-r)+r] \label{defiru}
\end{align} 
adopt this ancestor, so the corresponding ancestral lines coalesce at $u$. It particular, each jump of $N$ results in a merger of families (with a fraction of the dust). It is justified in the following proposition that the stochastic flow solving \eqref{defflowybysde} is well-defined. 
\begin{prop} \label{propexistence}
Assume that \eqref{integassumption} holds true. There exists a unique stochastic flow $(Y_{0,t}(y), y \in [0,1], t \geq 0)$ such that, almost surely, the following properties hold: 
\begin{enumerate}[(i)]
\item \eqref{defflowybysde} holds for all $y \in [0,1]$ and $t\geq 0$;
\item for every $y \in [0,1]$, the trajectory $t \mapsto Y_{0,t}(y)$ is c\`ad-l\`ag; 
\item for every $t \geq 0$, the map $y \mapsto Y_{0,t}(y)$ is non-decreasing and continuous, and $Y_{0,t}(0)=0$, $Y_{0,t}(1)=1$. 
\end{enumerate}
\end{prop}
We refer to Appendix \ref{propflows} for the detailed construction of the flow $(Y_{0,t}(y), y \in [0,1], t \geq 0)$ and the proofs of its basic properties. In particular, Proposition \ref{propexistence} is proved in Appendix \ref{propflows0}. Proposition \ref{martpb}, also proved in Appendix \ref{propflows0}, shows that the $p$-point motion of this flow solves the martingale problem satisfied by the $p$-point motion of the flow of inverses of the $\Lambda$-process (see \cite[Thm. 5]{BERTOIN2005307}) and that, in our case, that martingale problem is well posed. This shows that, in our case, the process $(Y_{0,t}(\cdot))_{t\geq 0}$, defined as the solution of \eqref{defflowybysde}, is indeed equal in law to the flow of inverses of the $\Lambda$-process, defined in \cite{BERTOIN2005307}. 

\begin{remark}
That $Y_{0,\cdot}(\cdot)$ is equal in law to the flow of inverses of the $\Lambda$-process implies in particular that, for any $x,y \in [0,1]$ and $t\geq 0$, we have $\mathbb{P}(X_t(x) \geq y)=\mathbb{P}(x \geq Y_{0,t}(y))$. In other words, the one-point motions of the flows \eqref{defflowxbysde} and \eqref{defflowybysde} are Siegmund duals. This last point was already observed in \cite[Thm. 2.5]{cordhumvech2022}. 
\end{remark}

\begin{remark}[A population model for $Y$]
Even if $Y$ is a tool to understand the genealogy of $X$, it also has a population model interpretation of its own. Consider an infinite population that is continuously distributed in $[0,1]$. If $(s,r,u)\in N$, at time $s$ a catastrophe occurs and kills all individuals in the sub-interval $I_{r,u}$. 
After the catastrophe, the remaining individuals in the population reproduce uniformly (preserving their order) such as to instantaneously refill all the interval $[0,1]$. Then, for any $z \in [0,1]$, the descendants at time $s$ of the individuals that occupied $[0,z]$ at time $s-$ can be seen to be $[0,\mr_{r,u}(z)]$. For any $t\geq 0$, the descendants at time $t$ of the individuals that occupied the interval $[0,y]$ at time $0$ is then $[0,Y_{0,t}(y)]$. 
\end{remark}

It is also of interest to start the flows $Y$ at specific times and to compose them. We show in Proposition \ref{propcomposition} from Appendix \ref{shiftedflows} that a countable family of stochastic flows $\{ (Y_{s,t}(y), y \in [0,1], t \geq s), s \in J_N \cup \{0\}\}$ can be defined on the same probability space, such that we have the following composition property: almost surely, for any $s_1,s_2 \in J_N \cup \{0\}$ with $s_1<s_2$, 
\begin{align}
\forall t \geq s_2, \forall y \in [0,1], \ Y_{s_1,t}(y)=Y_{s_2,t}(Y_{s_1,s_2}(y)). \label{compprop}
\end{align} 
Moreover, each flow in this family satisfies a shifted version of \eqref{defflowybysde} (see Proposition \ref{propcomposition}). A consequence of \eqref{defflowybysde} and the composition property \eqref{compprop} is that the following holds almost surely: for any $t>0$ and $x \in [0,1]$, if for some $(s,r,u) \in N$ with $s \in (0,t]$ we have $x \in Y_{0,s-}^{-1}(I_{r,u})$ then 
\begin{align}
Y_{0,s}(x)=\mr_{r,u}(Y_{0,s-}(x))=u \ \text{and} \ Y_{0,t}(x)=Y_{s,t}(Y_{0,s}(x))=Y_{s,t}(u). \label{collage}
\end{align}
In other words, each jumping time $s \in J_N$ yields a merger event for some trajectories of the flow $Y$. The following proposition, proved in Appendix \ref{sectionnocontmergers}, shows the important property that mergers of trajectories of the flow $Y$ cannot occur continuously but only at jumping time $s \in J_N$. 
\begin{prop} \label{nonzerodiff}
Assume that \eqref{integassumption} holds true. Almost surely, two distinct trajectories $Y_{0,t}(y_1)$ and $Y_{0,t}(y_2)$ can only merge at a jumping time $s \in J_N$, and only if $y_1$ and $y_2$ both belong to $Y_{0,s-}^{-1}(I_{r,u})$ for some atom $(s,r,u) \in N$. More precisely,
\begin{align*}
& \mathbb{P} \left ( \forall t>0, \forall y_1 \neq y_2 \in [0,1], \ Y_{0,t}(y_1)\neq Y_{0,t}(y_2) \ \text{or} \ \exists (s,r,u) \in N \ \text{s.t.} \ s\leq t, \ y_1,y_2 \in Y_{0,s-}^{-1}(I_{r,u}) \right ) \\
& \qquad \qquad \qquad \qquad \qquad \qquad \qquad \qquad \qquad \qquad \qquad \qquad \qquad \qquad \qquad \qquad \qquad \qquad \qquad =1. 
\end{align*}
The same property holds with $Y_{0,t}(\cdot)$ replaced by $Y_{0,t-}(\cdot)$ and $s\leq t$ replaced by $s<t$. 
\end{prop}

\subsection{From flow of inverses to $\Lambda$-coalescent} \label{model2}

We now explain how the flow $Y$ from Proposition \ref{propexistence} naturally provides a construction of the process of masses of blocks of a $\Lambda$-coalescent with dust. We first introduce some notation that will be used throughout. 
We write $U \sim \mathcal{U}([0,1])$ (resp. $(U_i)_{i\geq 1}\sim \mathcal{U}([0,1])^{\otimes \mathbb{N}}$) if $U$ is a uniform random variable (resp. a sequence of iid uniform random variables) on $(0,1)$. Similarly we write $Z \sim \Ber(r)$ (resp. $(Z_i)_{i\geq 1}\sim\Ber(r)^{\otimes \mathbb{N}}$) if $Z$ is a Bernoulli random variable (resp. a sequence of iid Bernoulli random variables) with parameter $r$, i.e. $\mathbb{P}(Z=1)=1-\mathbb{P}(Z=0)=r$. We also use the following set-theoretic notation: $Q:=[0,1] \cap \mathbb{Q}$, for any set $A \subset [0,1]$ we write $\overline{A}$ and $A^{\mathrm{o}}$ for respectively the closure and interior of $A$, and $A^c := [0,1] \setminus A$. We denote by $\mathcal{B}([0,1])$ the family of Borel sets in $[0,1]$. 

To a realization of the flow $Y$ we first associate a partition process via the sampling procedure of \cite{BLGI}: let $(U_i)_{i\geq 1}\sim \mathcal{U}([0,1])^{\otimes \mathbb{N}}$ that is independent of $N$ (and therefore of $Y$), we define a process $(\pi^Y_t)_{t\geq 0}$ of random partitions of $\mathbb{N}$ by the equivalence relation $i \sim_{\pi^Y_t} j \Leftrightarrow Y_{0,t}(U_i)=Y_{0,t}(U_j)$. The following lemma is in line with \cite[Thm. 1]{BLGI} and relates $Y$ to the $\Lambda$-coalescent. 
\begin{lemma} \label{lawofpartbyblg1and20}
Assume that \eqref{integassumption} holds true. The partition process $(\pi^Y_t)_{t\geq 0}$ is a $\Lambda$-coalescent. 
\end{lemma}

For all $t \geq 0$ let $m_t$ (resp. $m_{t-}$) be the Stieltjes measure on $[0,1]$ associated to the non-decreasing function $Y_{0,t}(\cdot)$ (resp. $Y_{0,t-}(\cdot)$), i.e. $m_t(A):=\int_{[0,1]} \mathds{1}_A(x) dY_{0,t}(x)$ (resp. $m_{t-}(A):=\int_{[0,1]} \mathds{1}_A(x) dY_{0,t-}(x)$) for $A \in \mathcal{B}([0,1])$. Proposition \ref{propexistence} ensures that, almost surely, for all $t \geq 0$ the measures $m_t$ and $m_{t-}$ are well-defined. It will turn out later that these measures have a simple expression (see Section \ref{calcmeas}). Let $C_t := \mathrm{Supp}(m_t)^c$ (resp. $C_{t-} := \mathrm{Supp}(m_{t-})^c$). As shown in Proposition \ref{poissrepstep1} below, almost surely for all $t>0$, $C_{t-}=\cup_{s \in (0,t)} C_{s}$. Let $(\mathcal{O}_k(t))_{k\geq 1}$ be an enumeration of the open connected components of $C_{t}$ such that $|\mathcal{O}_1(t)|\geq |\mathcal{O}_2(t)| \geq \dots$ \footnote{If several components have the same length, their order is determined by the $U_i'$ of smallest index that they contain, where $(U_i')_{i\geq 1} \sim \mathcal{U}([0,1])^{\otimes \mathbb{N}}$ is independent from everything else.}. We set $W_k(t):=|\mathcal{O}_k(t)|$ and $M_k(t):=\sum_{1 \leq j \leq k} W_j(t)$. Since $Y_{0,t}(\cdot)$ is constant on each $\mathcal{O}_k(t)$, let us denote by $V_k(t)$ the value taken by $Y_{0,t}(\cdot)$ on $\mathcal{O}_k(t)$; we then have $Y_{0,t}^{-1}(\{V_k(t)\})=\overline{\mathcal{O}_k(t)}$. We note that $k_1 \neq k_2 \Rightarrow V_{k_1}(t) \neq V_{k_2}(t)$. We enlarge the probability space by adding a new sequence $(\tilde U_k)_{k \geq 1}$ that is independent from everything else and, if for some $t$, $C_t$ has only finitely many open components $\mathcal{O}_1(t), \ldots, \mathcal{O}_K(t)$ (which occurs if the measure $\Lambda$ satisfies \eqref{strongintegassumption}) then for all $k > K$, we set $\mathcal{O}_k(t):=\emptyset$ (so $W_k(t)=0$) and $V_k(t):=\tilde U_k$. We similarly define $(\mathcal{O}_k(t-), W_k(t-), V_k(t-))_{k\geq 1}$ from $C_{t-}$ instead of $C_{t}$. What we informally called \textit{$t$-families} earlier are the open connected components of $C_{t}$ or, equivalently, the jump intervals of $X_{-t,0}(\cdot)$ (note however that $X_{-t,0}(\cdot)$ is not formally defined in our framework). 

\begin{prop} \label{combrepres}
Assume that \eqref{integassumption} holds true. Almost surely, the following holds 
\begin{itemize}
\item $(C_t)_{t \geq 0}$ is a nested interval-partition in the sense of \cite[Def. 1.3]{10.1214/20-AAP1641}; 
\item $(\pi^Y_t)_{t\geq 0}$ is the partition process obtained from the paintbox based on $(C_t)_{t \geq 0}$ in the sense of \cite[eq. (2)]{10.1214/20-AAP1641}; 
\item For any $k\geq 1$, $t \mapsto W_k(t)$ is c\`ad-l\`ag and $\lim_{s \rightarrow t, s<t}W_k(s)=W_k(t-)$ for all $t>0$; 
\item For any $k\geq 1$, $\{ t \geq 0 \ \text{s.t.} \ W_k(t)\neq W_k(t-) \} \subset J_N$; in particular, for any fixed $T\geq 0$ and $k\geq 1$, $(W_k(t))_{t\geq 0}$ is almost surely continuous at $T$; 
\item $t \mapsto (W_k(t))_{k \geq 1}$ is a c\`ad-l\`ag (for the topology considered in \cite[Prop. 1.1]{lecturebertoin2010}) version of the process of ordered non-zero masses of blocks of a $\Lambda$-coalescent. 
\end{itemize}
\end{prop}
Lemma \ref{lawofpartbyblg1and20} and Proposition \ref{combrepres} are rather intuitive but we provide justifications for them in Sections \ref{coaltrajy} and \ref{relationinvcoal} respectively. 
Based on Proposition \ref{combrepres}, $(W_k(t))_{k \geq 1, t \geq 0}$ is distributed as the process of ordered masses of blocks of a $\Lambda$-coalescent; we thus now refer to it as such. That relation can be seen as a restatement of the well-known fact that the jumps sizes of $X_{-t,0}(\cdot)$ are the masses of blocks of a $\Lambda$-coalescent at time $t$, \cite[Sec. 3.3]{BLGI}, \cite[(1.4)]{labbe2014a}. 
The combination of Proposition \ref{combrepres} with Lemma \ref{lawofpartbyblg1and20} also shows that $(C_t)_{t \geq 0}$ is a nested interval-partition associated with a $\Lambda$-coalescent. We thus recover \cite[Prop. 1.11]{10.1214/20-AAP1641}. Our construction is slightly different from the one in \cite{10.1214/20-AAP1641}, as it is based on the flow of inverses $Y$ (itself constructed in Proposition \ref{propexistence}), while their construction is based on nested compositions of $\mathbb{N}$, see \cite[Sec. 3]{10.1214/20-AAP1641}. However, the $\Lambda$-Fleming-Viot flow underlies both constructions \cite[Cor. 3.6]{10.1214/20-AAP1641}. Working with the flow of inverses will allow us to derive a useful Poisson representation in Theorem \ref{represwktviamut}. 

\subsection{From block masses to stochastic integrals} \label{model3}

Let $(S_t)_{t\geq 0}$ be the subordinator defined by 
\begin{align}
S_t := -\int_{(0,t] \times (0,1)^2} \log(1-r) N(ds,dr,du). \label{defsubst}
\end{align}
By \cite[Thm. 19.3]{KenIti1999} we see that $(S_t)_{t\geq 0}$ is well-defined under assumption \eqref{integassumption}. By the L\'evy-Kintchine formula, the Laplace exponent of $(S_t)_{t\geq 0}$ is $\phi_S(\cdot)$, defined in \eqref{laplaceexponentlk}, i.e. for any $\lambda,t\geq 0$ we have $\mathbb{E}[e^{-\lambda S_t}]=e^{-t \phi_S(\lambda)}$. 

It is well-known that, for $\Lambda$ satisfying \eqref{integassumption}, the total mass of the dust of the $\Lambda$-coalescent at time $t$ is given by $e^{-S_t}$ \cite[Prop. 26]{pitman1999}, see also \cite{lecturebertoin2010,ref27ofsurvey,surveylambdacoal}. This classical fact, which can be stated as $\sum_{k\geq 1}W_k(t)=1-e^{-S_t}$, is, not surprisingly, also recovered from our construction, see Corollary \ref{calczonecollage} and Remark \ref{sumwkteq1minusst} from Section \ref{calcmeas}. More importantly, $(S_t)_{t\geq 0}$, plays a role in the Poisson representations from Theorems \ref{intsto} and \ref{represwktviamut}, and the formula from Theorem \ref{ito}. 

For any $t \geq 0$ and $(r,u) \in (0,1)^2$, we introduce indicators tracking which blocks are affected by a potential jump at $(t,r,u)$. Recall that $I_{r,u}$ is defined in \eqref{defiru}. We let $Z_k(t,r,u):=\mathds{1}_{V_k(t) \in I_{r,u}}$  (resp. $Z_k(t-,r,u):=\mathds{1}_{V_k(t-) \in I_{r,u}}$) and $\beta_k(t,r,u) := \sum_{j=1}^kZ_j(t,r,u)$ (resp. $\beta_k(t-,r,u) := \sum_{j=1}^kZ_j(t-,r,u)$). The following lemma, which is a consequence of \cite[Lem. 2]{BLGI} and is justified in Section \ref{relationinvcoal}, shows that the locations $V_k(t)$ of the blocks are uniformly distributed and independent of their masses $W_k(t)$. 
\begin{lemma} \label{lawofvkandzk}
For any fixed $t\geq 0$, $(V_k(t))_{k\geq 1}$ and $(W_k(t))_{k\geq 1}$ are independent, and $(V_k(t))_{k\geq 1}\sim \mathcal{U}([0,1])^{\otimes \mathbb{N}}$. In particular, for any $(t,r,u) \in [0,\infty) \times (0,1)^2$, $(Z_k(t,r,u))_{k\geq 1}$ and $(W_k(t))_{k\geq 1}$ are independent and $(Z_k(t,r,u))_{k\geq 1}\sim\Ber(r)^{\otimes \mathbb{N}}$. 
\end{lemma}

We now introduce a random measure $\mu_t$ which encodes the masses of the blocks, or equivalently the lengths of the intervals coalesced by $Y$ at time $t$. It is defined for any $t \geq 0$ by 
\begin{align}
\mu_t := \sum_{j\geq 1} W_j(t) \delta_{V_j(t)}. \label{represwktviamut2}
\end{align}
For any $t>0$ we similarly define $\mu_{t-}:=\sum_{j\geq 1} W_j(t-) \delta_{V_j(t-)}$. A measure similar to $\mu_t$ appeared in \cite[eq. (1.4)]{labbe2014a} in the lookdown representation of the $\Lambda$-Fleming-Viot flow, but its distributional properties have not been studied there. The following result provides an almost sure representation of $\mu_t$ in terms of the flow $Y$. 
\begin{theo} \label{represwktviamut}
Assume that \eqref{integassumption} holds true. For any fixed $t\geq 0$ we have almost surely 
\begin{align}
\mu_t = \sum_{(s,r,u) \in N, s \in (0,t]} re^{-S_{s-}} \delta_{Y_{s,t}(u)}. \label{defmut}
\end{align}
Moreover, we have almost surely that, for all $t \in J_N$, 
\begin{align}
\mu_t = \sum_{(s,r,u) \in N, s \in (0,t]} re^{-S_{s-}} \delta_{Y_{s,t}(u)} \ \text{and} \ \mu_{t-} = \sum_{(s,r,u) \in N, s \in (0,t)} re^{-S_{s-}} \delta_{Y_{s,t-}(u)}. \label{represwktviamut1}
\end{align}
\end{theo}
It is transparent from the Poisson representation of Theorem \ref{represwktviamut} that each jump $(s,r,u) \in N$ represents a merger of existing blocks together with a fraction $r$ of the dust. The representation from Theorem \ref{represwktviamut} yields $W_k(t)=\sum_{(s,r,u) \in N, s \in (0,t]} re^{-S_{s-}} \mathds{1}_{Y_{s,t}(u)=V_k(t)}$. Unfortunately, this expression is not a stochastic integral because the integrand is not progressively measurable. However, Theorem \ref{represwktviamut} is the key ingredient in the 
proof of Theorem \ref{intsto}: by identifying, for each jump $(t,r,u) \in N$, the increment of $W_k(\cdot)$ and $M_k(\cdot)$ at time $t$, it leads to the definitions of $K^k_t(r,u)$ and $H^k_t(r,u)$ below.

For any $t \geq 0$, $r,u \in (0,1)$, and $k\geq 1$ we set 
\begin{align}
H^k_{t}(r,u) := & \mathds{1}_{\beta_k(t,r,u)=0} \left (e^{-S_{t}}r+\sum_{j>k}Z_j(t,r,u) W_j(t)-W_k(t) \right )_+ \nonumber \\
+ & \mathds{1}_{\beta_k(t,r,u)\neq 0} \left (e^{-S_{t}}r+\sum_{j>k}Z_j(t,r,u) W_j(t) \right ) \nonumber \\
+ & \sum_{j>k}(1-Z_j(t,r,u)) W_j(t) \mathds{1}_{\sum_{i=1}^j(1-Z_i(t,r,u)) \leq k-1}. \label{prepCF}
\end{align}
For any $t \geq 0$, $r,u \in (0,1)$, and $k\geq 2$ we set 
\begin{align}
K^k_{t}(r,u) := & \mathds{1}_{\beta_k(t,r,u)=0} \nonumber \\
& \times \left ( \mathsf{Median} \left \{ W_{k-1}(t), W_k(t), e^{-S_{t}}r+\sum_{j>k}Z_j(t,r,u) W_j(t) \right \} - W_k(t)\right ) \nonumber \\
+ & \mathds{1}_{\beta_k(t,r,u)=1, Z_k(t,r,u)=1} \nonumber \\
& \times \left (\mathsf{Min} \left \{ W_k(t) + e^{-S_{t}}r+\sum_{j>k}Z_j(t,r,u) W_j(t), W_{k-1}(t) \right \} - W_k(t) \right ) \nonumber \\
+ & \mathds{1}_{\beta_k(t,r,u)\geq 2} \left ( \sum_{j>k}(1-Z_j(t,r,u)) W_j(t) \mathds{1}_{\sum_{i=1}^j(1-Z_i(t,r,u)) = k-1} -W_k(t)\right ). \label{prepCFlb}
\end{align}
For $k=1$ we set $K^1_{t}(r,u) := H^1_{t}(r,u)$. We similarly define $H^k_{t-}(r,u)$ and $K^k_{t-}(r,u)$. As stated in Theorem \ref{intsto}, the quantities $H^k_{t-}(r,u)$ and $K^k_{t-}(r,u)$ represent the increment of respectively $M_k(\cdot)$ and $W_k(\cdot)$ at $t$ if $(t,r,u)$ is a jump. Having defined these quantities at times $t$ that are not necessarily jumping times makes separation of randomness possible, for example in Theorem \ref{ito}. 

\section{Long time behavior} \label{longtimebehav}

In this section, we prove the asymptotic results for the moments of the block masses (Theorem \ref{longtimebehavth}), the large deviation principles (Theorems \ref{thmpgd} and \ref{thmpgdwk}), and the law of large numbers (Corollary \ref{corlgn}), taking the representations of Theorems \ref{intsto} and \ref{represwktviamut} and the pseudo-generator formula of Theorem \ref{ito} as given. The proofs of the latter results are deferred to Sections \ref{poisrep} and \ref{itoform}, as they require more involved technical arguments related to the flow of inverses. Recall that we always assume that \eqref{integassumption} holds true. 
For $k \geq 1$ let 
\begin{align*}
\phi_k(\eta) & := \int_{(0,1)} (1-(1-r)^k) ( 1-(1-r)^\eta ) r^{-2} \Lambda(dr), \\ 
\psi_k(\eta) & := \int_{(0,1)} (1-r)^\eta(1-(1-r)^k - kr(1-r)^{k-1}) r^{-2} \Lambda(dr). 
\end{align*}
These coefficients will appear below in some differential inequalities related to the functions $t \mapsto \mathbb{E} [ (1-M_k(t))^{\eta}]$. Note that, for any $\eta \in (0,\infty)$ (resp. $\eta \in (-\theta(\Lambda),0)$), $\phi_k(\eta)$ increases (resp. decreases) to $\phi_S(\eta)$ as $k\to\infty$. 

The proof of Theorem \ref{longtimebehavth} requires establishing upper and lower bounds. When $\eta \in [0,1]$, the lower bounds 
follow from showing that both $e^{-t\phi_S(\eta)}$ and $e^{-t\lambda_k(\Lambda)}$ are lower bounds (up to a multiplicative constant) for $\mathbb{E}[W_k(t)^\eta]$. The upper bounds are more involved and require distinguishing the cases $\phi_S(\eta)>\lambda_k(\Lambda)$ and $\phi_S(\eta)\leq\lambda_k(\Lambda)$. The proofs of the upper bounds are provided in Section \ref{longtimeupperbound} and based on Theorem \ref{ito}. The proofs of the lower bounds are provided in Sections \ref{longtimeupperbound} and \ref{longtimelowerboundspe} and based, depending on the case, on Theorem \ref{ito} or on Theorem \ref{represwktviamut}. 

\subsection{Bounds based on Theorem \ref{ito}} \label{longtimeupperbound}

In this subsection we apply Theorem \ref{ito} to derive some bounds required for the proof of Theorem \ref{longtimebehavth}. The following lemma provides a bound for $\mathbb{E} [ (1-M_k(t))^{\eta}]$ that is almost optimal for large values of $k$. 
\begin{lemma} \label{minoyklem}
For any $\eta \in [0,1]$ (resp. $\eta \in (-\theta(\Lambda),0)$), $k \geq 1$ and $t \geq 0$ we have 
\begin{align}
\mathbb{E} [ (1-M_k(t))^{\eta}] \leq e^{-t \phi_k(\eta)} \ \text{(resp.} \ \geq e^{-t \phi_k(\eta)}\text{)}. \label{minoyk}
\end{align}
\end{lemma}

\begin{proof}
From the definition of $H^k_{\cdot}(\cdot,\cdot)$ in \eqref{prepCF} we have almost surely 
\begin{align}
H^k_t(r,u) & \geq (1-\mathds{1}_{Z_1(t,r,u)=\cdots=Z_k(t,r,u)=0}) \left (e^{-S_t}r+\sum_{j>k}Z_j(t,r,u) W_j(t) \right ) \nonumber \\
& + (1-Z_{k+1}(t,r,u)) W_{k+1}(t) \mathds{1}_{\sum_{i=1}^{k+1}(1-Z_i(t,r,u)) \leq k-1} \nonumber \\
& = \mathds{1}_{\beta_k(t,r,u)\geq 1} \left (e^{-S_t}r+\sum_{j>k}Z_j(t,r,u) W_j(t) \right ) + \mathds{1}_{Z_{k+1}(t,r,u)=0, \beta_k(t,r,u)\geq 2} W_{k+1}(t) \label{boundhkt1} \\
& \geq \mathds{1}_{\beta_k(t,r,u)\geq 1} \left (e^{-S_t}r+\sum_{j>k}Z_j(t,r,u) W_j(t) \right ) \label{boundhkt2}. 
\end{align}
The intermediary bound \eqref{boundhkt1} will come useful later on. We now bound the quantity that appears in the expression of $\frac{d}{dt} \mathbb{E}[(1-M_k(t))^{\eta}]$ provided by Theorem \ref{ito}, starting with the case $\eta \in [0,1]$. Using the bound \eqref{boundhkt2}, recalling from Lemma \ref{lawofvkandzk} and Remark \ref{zkindepst} that $(Z_j(t,r,u))_{j\geq 1}$ and $\beta_k(t,r,u)$ are independent of $(W_j(t))_{j\geq 1}$ and $S_t$ and that $(Z_j(t,r,u))_{j\geq 1}\sim\Ber(r)^{\otimes \mathbb{N}}$, using Jensen inequality, and recalling from Remark \ref{sumwkteq1minusst} that $e^{-S_t}+\sum_{j>k} W_j(t)=1-M_k(t)$ we get 
\begin{align}
& \mathbb{E} \left [ \left (1-M_k(t)-H^k_t(r,u) \right )^{\eta} - (1-M_k(t))^{\eta} \right ] \nonumber \\
\leq & \mathbb{E} \left [ \mathds{1}_{\beta_k(t,r,u)\geq 1} \left ( (1-M_k(t)-e^{-S_t}r-\sum_{j>k}Z_j(t,r,u) W_j(t) )^{\eta} - (1-M_k(t))^{\eta} \right ) \right ] \nonumber \\
= & (1-(1-r)^k) \mathbb{E} \left [ \mathbb{E} \left [ (1-M_k(t)-e^{-S_t}r-\sum_{j>k}Z_j(t,r,u) W_j(t) )^{\eta} \big | M_k(t),S_t \right ] - (1-M_k(t))^{\eta} \right ] \nonumber \\
\leq & (1-(1-r)^k) \mathbb{E} \left [ \mathbb{E} \left [ 1-M_k(t)-e^{-S_t}r-\sum_{j>k}Z_j(t,r,u) W_j(t) \big | M_k(t),S_t \right ]^{\eta} - (1-M_k(t))^{\eta} \right ] \nonumber \\
= & (1-(1-r)^k) \mathbb{E} \left [ ( 1-M_k(t)-e^{-S_t}r-r\sum_{j>k} W_j(t) )^{\eta} - (1-M_k(t))^{\eta} \right ] \nonumber \\
= & - ( 1-(1-r)^\eta ) (1-(1-r)^k) \mathbb{E} \left [ (1-M_k(t))^{\eta} \right ]. \label{calcjensen}
\end{align}
Now, applying \eqref{derfmkt} from Theorem \ref{ito} with $f(x):=(1-x)^{\eta}$, plugging in the above bound, and using the definition of $\phi_k(\eta)$, we obtain 
\begin{align}
\frac{d}{dt} \mathbb{E}[(1-M_k(t))^{\eta}] \leq - \phi_k(\eta) \mathbb{E}[(1-M_k(t))^{\eta}]. \label{ineqdiffsimple}
\end{align}
This yields $\frac{d}{dt} \log \mathbb{E}[(1-M_k(t))^{\eta}] \leq -\phi_k(\eta)$ and \eqref{minoyk} follows in the case $\eta \in [0,1]$. 

We note that the two inequalities in \eqref{calcjensen} came from the fact that the function $x \mapsto x^{\eta}$ is non-decreasing and concave when $\eta \in [0,1]$. In the case $\eta \in (-\theta(\Lambda),0)$, that function is non-increasing and convex, so the calculation \eqref{calcjensen} can be performed again, with all inequalities being reversed. We thus get that \eqref{ineqdiffsimple} holds true with "$\leq$" replaced by "$\geq$" and \eqref{minoyk} follows in the case $\eta \in (-\theta(\Lambda),0)$. 
\end{proof}

The following lemma allows one to transform the task of proving a bound for $\mathbb{E} [ (1-M_k(t))^{\eta}]$ into the task of proving a bound for $\mathbb{E} [ (1-M_{k+1}(t))^{\eta}]$. 
\begin{lemma} \label{minoyklembis}
For any $\eta \in [0,1]$ (resp. $\eta \in (-\theta(\Lambda),0)$), $k \geq 1$ and $t \geq 0$ we have 
\begin{align}
\frac{d}{dt} \mathbb{E}[(1-M_k(t))^{\eta}] \leq \ \text{(resp.} \ \geq \text{)} - \lambda_{k+1}(\Lambda) \mathbb{E}[(1-M_k(t))^{\eta}] + \psi_{k+1}(\eta) \mathbb{E}[(1-M_{k+1}(t))^{\eta}]. \label{minoykbis}
\end{align}
\end{lemma}

\begin{proof}
We first note that the bound \eqref{boundhkt1} can be re-written as follows. 
\begin{align}
H^k_t(r,u) & \geq (\mathds{1}_{\beta_k(t,r,u)=1,Z_{k+1}(t,r,u)=1} + \mathds{1}_{\beta_k(t,r,u)\geq 2}) \left (W_{k+1}(t) + e^{-S_t}r+\sum_{j>k+1}Z_j(t,r,u) W_j(t) \right ) \nonumber \\ 
& + \mathds{1}_{\beta_k(t,r,u)=1,Z_{k+1}(t,r,u)=0} \left (e^{-S_t}r+\sum_{j>k+1}Z_j(t,r,u) W_j(t) \right ) \label{boundhkt3}. 
\end{align}
We start with the case $\eta \in [0,1]$. The idea is to proceed as in the calculation \eqref{calcjensen} from the proof of Lemma \ref{minoyklem} but to use the bound \eqref{boundhkt3} instead of \eqref{boundhkt2}. We obtain 
\begin{align*}
& \mathbb{E} \left [ \left (1-M_k(t)-H^k_t(r,u) \right )^{\eta} - (1-M_k(t))^{\eta} \right ] \\
\leq & \mathbb{E} \left [ (\mathds{1}_{\beta_k(t,r,u)=1,Z_{k+1}(t,r,u)=1} + \mathds{1}_{\beta_k(t,r,u)\geq 2}) (1-M_{k+1}(t)-e^{-S_t}r-\sum_{j>k+1}Z_j(t,r,u) W_j(t) )^{\eta} \right ] \\
& \qquad + \mathbb{E} \left [ \mathds{1}_{\beta_k(t,r,u)= 1,Z_{k+1}(t,r,u)=0} (1-M_k(t)-e^{-S_t}r-\sum_{j>k+1}Z_j(t,r,u) W_j(t) )^{\eta} \right ] \\
& \qquad \qquad - \mathbb{E} \left [ \mathds{1}_{\beta_k(t,r,u)\geq 1} (1-M_k(t))^{\eta} \right ] \\
\leq & (kr^2(1-r)^{k-1}+1-(1-r)^k-kr(1-r)^{k-1}) \\
& \qquad \times \mathbb{E} \left [ \mathbb{E} \left [ 1-M_{k+1}(t)-e^{-S_t}r-\sum_{j>k+1}Z_j(t,r,u) W_j(t) \big | M_k(t),S_t \right ]^{\eta} \right ] \\
& \qquad \qquad + kr(1-r)^{k} \mathbb{E} \left [ (1-M_k(t))^{\eta} \right ] - (1-(1-r)^k) \mathbb{E} \left [ (1-M_k(t))^{\eta} \right ] \\
= & (1-(1-r)^{k}-kr(1-r)^{k}) (1-r)^{\eta} \mathbb{E} \left [ (1-M_{k+1}(t))^{\eta} \right ] \\
& \qquad - (1-(1-r)^{k}-kr(1-r)^{k}) \mathbb{E} \left [ (1-M_k(t))^{\eta} \right ] \\
= & (1-(1-r)^{k+1}-(k+1)r(1-r)^{k}) \left ( (1-r)^{\eta} \mathbb{E} \left [ (1-M_{k+1}(t))^{\eta} \right ] - \mathbb{E} \left [ (1-M_k(t))^{\eta} \right ] \right ). 
\end{align*}
Now, applying \eqref{derfmkt} from Theorem \ref{ito} with $f(x):=(1-x)^{\eta}$, plugging in the above bound, and using the definition of $\lambda_{k+1}(\Lambda)$ and $\psi_{k+1}(\eta)$, we get \eqref{minoykbis}. 

Again, the two inequalities in the above calculation came from the fact that the function $x \mapsto x^{\eta}$ is non-decreasing and concave when $\eta \in [0,1]$. In the case $\eta \in (-\theta(\Lambda),0)$, that function is non-increasing and convex, so the calculation can be performed again with all inequalities being reversed, and \eqref{minoykbis} follows in the case $\eta \in (-\theta(\Lambda),0)$. 
\end{proof}

The following result provides optimal bounds for $\mathbb{E} [ (1-M_k(t))^{\eta}]$. The idea is to use Lemma \ref{minoyklembis} iteratively until we reach indices for which the bound from Lemma \ref{minoyklem} becomes sharp. 
\begin{prop} \label{minoyklem2}
For any $\eta \in [0,1]$ and $k \geq 1$, 
\begin{itemize}
\item[(i)] If $\phi_S(\eta)>\lambda_{k+1}(\Lambda)$ then there is $C>0$ such that for any $t \geq 0$ we have 
\begin{align}
\mathbb{E} [ (1-M_k(t))^{\eta}] \leq C e^{-t \lambda_{k+1}(\Lambda)}. \label{minoykopt}
\end{align}
\item[(ii)] If $\phi_S(\eta) \leq \lambda_{k+1}(\Lambda)$ then for any $\epsilon>0$, there is $C(\epsilon)>0$ such that for any $t \geq 0$ we have 
\begin{align}
e^{-t \phi_S(\eta)} \leq \mathbb{E} [ (1-M_k(t))^{\eta}] \leq C(\epsilon) e^{-t (\phi_S(\eta)-\epsilon)}. \label{minoykoptpostcr}
\end{align}
\end{itemize}
For any $\eta \in (-\theta(\Lambda),0)$, $k \geq 1$, and $\epsilon>0$, there is $C(\epsilon)>0$ such that for any $t \geq 0$ we have 
\begin{align}
C(\epsilon) e^{-t (\phi_S(\eta)+\epsilon)} \leq \mathbb{E} [ (1-M_k(t))^{\eta}] \leq e^{-t \phi_S(\eta)}. \label{minoykoptpostcretaneg}
\end{align}
\end{prop}

\begin{proof}
Let us fix $\eta \in (0,1]$ and $k\geq 1$ (the case $\eta=0$ is obvious). We set $f_k(s) := e^{s \lambda_{k+1}(\Lambda)}\mathbb{E} [ (1-M_k(s))^{\eta}]$. We have $f_k(0)=1$ and, using Lemma \ref{minoyklembis}, we get 
\begin{align*}
\frac{d}{ds} f_k(s) & = \lambda_{k+1}(\Lambda) e^{s \lambda_{k+1}(\Lambda)}\mathbb{E} [ (1-M_k(s))^{\eta}] +  e^{s \lambda_{k+1}(\Lambda)} \frac{d}{ds} \mathbb{E}[(1-M_k(s))^{\eta}] \\
& \leq \psi_{k+1}(\eta) e^{s \lambda_{k+1}(\Lambda)} \mathbb{E}[(1-M_{k+1}(s))^{\eta}]. 
\end{align*}
Integrating this inequality on $[0,t]$ and multiplying both sides by $e^{-t\lambda_{k+1}(\Lambda)}$ we get, for $t \geq 0$, 
\begin{align}
\mathbb{E} [ (1-M_k(t))^{\eta}] \leq e^{-t\lambda_{k+1}(\Lambda)} \left ( 1 + \psi_{k+1}(\eta) \int_0^t e^{s \lambda_{k+1}(\Lambda)}\mathbb{E}[(1-M_{k+1}(s))^{\eta}]ds \right ). \label{optminoderyk}
\end{align}
Iterating \eqref{optminoderyk} we get that for any $n \geq 1$ and $t \geq 0$, $e^{t\lambda_{k+1}(\Lambda)}\mathbb{E} [ (1-M_k(t))^{\eta}]$ is smaller than 
\begin{align}
 & 1 + \sum_{j=1}^{n-1} \left ( \prod_{i=1}^j \psi_{k+i}(\eta) \right ) \int_{[0,t]^j} \left ( \prod_{i=1}^j e^{s_i(\lambda_{k+i}(\Lambda)-\lambda_{k+1+i}(\Lambda))} \right ) \mathds{1}_{s_j\leq \cdots \leq s_1} ds_1 ... ds_j \nonumber \\
+ & \left ( \prod_{i=1}^n \psi_{k+i}(\eta) \right ) \int_{[0,t]^n} \left ( \prod_{i=1}^{n-1} e^{s_i(\lambda_{k+i}(\Lambda)-\lambda_{k+1+i}(\Lambda))} \right ) e^{s_n \lambda_{k+n}(\Lambda)} \mathbb{E}[(1-M_{k+n}(s_n))^{\eta}] \mathds{1}_{s_n\leq \cdots \leq s_1} ds_1 ... ds_n, \label{optminoderykbis}
\end{align}
with the conventions $\sum_{j=1}^{0}\cdots=0$ and $\prod_{i=1}^0 \cdots=1$. Since $\int_{[0,t]^j} ( \ldots ) \mathds{1}_{s_j\leq \cdots \leq s_1} ds_1 ... ds_j \leq \int_{[0,\infty)^j} ( \ldots ) ds_1 ... ds_j$ we have 
\begin{align}
\int_{[0,t]^j} \left ( \prod_{i=1}^j e^{s_i (\lambda_{k+i}(\Lambda)-\lambda_{k+1+i}(\Lambda))} \right ) \mathds{1}_{s_j\leq \cdots \leq s_1} ds_1 ... ds_j \leq \prod_{i=1}^j \frac1{\lambda_{k+1+i}(\Lambda)-\lambda_{k+i}(\Lambda)}. \label{optminoderykterm1}
\end{align}

Let us assume that $\phi_S(\eta) \leq \lambda_{k+1}(\Lambda)$ and fix $\epsilon>0$. We fix $n$ large enough so that $\phi_S(\eta)-\epsilon<\phi_{k+n}(\eta)$. We thus have $\phi_S(\eta)-\epsilon<\phi_{k+n}(\eta)<\phi_S(\eta)\leq \lambda_{k+1}(\Lambda)$. Using Lemma \ref{minoyklem} and integrating the variables one by one we get that the second integral in \eqref{optminoderykbis} is smaller than 
\begin{align}
& \int_{[0,t]^{n-1}} \left ( \prod_{i=1}^{n-1} e^{s_i(\lambda_{k+i}(\Lambda)-\lambda_{k+1+i}(\Lambda))} \right ) \frac{e^{s_{n-1}(\lambda_{k+n}(\Lambda)-\phi_{k+n}(\eta))}}{\lambda_{k+n}(\Lambda)-\phi_{k+n}(\eta)} \mathds{1}_{s_{n-1}\leq \cdots \leq s_1} ds_1 ... ds_{n-1} \nonumber \\
\leq & \left ( \prod_{i=1}^n \frac1{\lambda_{k+i}(\Lambda)-\phi_{k+n}(\eta)} \right ) e^{t(\lambda_{k+1}(\Lambda)-\phi_{k+n}(\eta))} \leq  e^{t\lambda_{k+1}(\Lambda)} \left ( \prod_{i=1}^n \frac1{\lambda_{k+i}(\Lambda)-\phi_{k+n}(\eta)} \right ) e^{-t(\phi_S(\eta)-\epsilon)}. \label{optminoderykterm2postcr}
\end{align}
Combining \eqref{optminoderykterm1} and \eqref{optminoderykterm2postcr} with \eqref{optminoderykbis} we get that the upper bound in \eqref{minoykoptpostcr} holds for some choice of $C(\epsilon)>0$. For the lower bound in \eqref{minoykoptpostcr}, just note from Remark \ref{sumwkteq1minusst} that $e^{-S_t} \leq 1-M_k(t)$ so $\mathbb{E}[(1-M_k(t))^{\eta}] \geq \mathbb{E}[e^{-\eta S_t}]=e^{-t\phi_S(\eta)}$. 

Let us now assume that $\phi_S(\eta)>\lambda_{k+1}(\Lambda)$. For $n$ large enough we have $\lambda_{k+1}(\Lambda)<\phi_{k+n}(\eta)<\phi_S(\eta)$ and $\phi_{k+n}(\eta)<\lambda_{k+n}(\Lambda)$. Note also that, for $n$ large enough, $\phi_{k+n}(\eta)$ does not coincide with any coefficient $\lambda_j(\Lambda)$. We assume that $n$ is chosen such that the above requirements are satisfied. Let $m:=\min \{ j \leq n, \lambda_{k+j}(\Lambda)>\phi_{k+n}(\eta)\}$. Note that $m \geq 2$ and $\lambda_{k+m-1}(\Lambda)<\phi_{k+n}(\eta)<\lambda_{k+m}(\Lambda)$. Proceeding as in \eqref{optminoderykterm2postcr} for the variables $s_n,\ldots,s_m$ and as in \eqref{optminoderykterm1} for the variables $s_{m-1},\ldots,s_1$ (where, by convention, the product $\prod_{i=1}^{m-2}(\cdots)$ and the integral $\int_{[0,t]^{m-2}}(\cdots)$ equal $1$ when $m=2$), we get that the second integral in \eqref{optminoderykbis} is smaller than 
\begin{align}
\left ( \prod_{i=m}^n \frac1{\lambda_{k+i}(\Lambda)-\phi_{k+n}(\eta)} \right ) \times \frac{1}{\phi_{k+n}(\eta)-\lambda_{k+m-1}(\Lambda)} \times \left ( \prod_{i=1}^{m-2} \frac1{\lambda_{k+1+i}(\Lambda)-\lambda_{k+i}(\Lambda)} \right ). \label{optminoderykterm2}
\end{align}
Combining \eqref{optminoderykterm1} and \eqref{optminoderykterm2} with \eqref{optminoderykbis} we get that \eqref{minoykopt} holds for some choice of $C>0$. 

Let us fix $\eta \in (-\theta(\Lambda),0)$ and $k\geq 1$. Since $e^{-S_t} \leq 1-M_k(t)$ (see Remark \ref{sumwkteq1minusst}) we have $\mathbb{E}[(1-M_k(t))^{\eta}] \leq \mathbb{E}[e^{-\eta S_t}]=e^{-t\phi_S(\eta)}$, yielding the upper bound in \eqref{minoykoptpostcretaneg}. We now fix $\epsilon>0$ such that $\phi_S(\eta)+\epsilon<0$. Using Lemma \ref{minoyklembis} and that $\mathbb{E}[(1-M_k(t))^{\eta}] \leq \mathbb{E}[(1-M_{k+1}(t))^{\eta}]$ we get 
\begin{align*}
\frac{d}{dt} \mathbb{E}[(1-M_k(t))^{\eta}] \geq (\psi_{k+1}(\eta)- \lambda_{k+1}(\Lambda)) \mathbb{E}[(1-M_{k+1}(t))^{\eta}]. 
\end{align*}
We note that $\psi_{k+1}(\eta) > \lambda_{k+1}(\Lambda)$ since $\eta<0$. Integrating the above inequality we get 
\begin{align}
\mathbb{E} [ (1-M_k(t))^{\eta}] \geq (\psi_{k+1}(\eta)- \lambda_{k+1}(\Lambda)) \int_0^t \mathbb{E}[(1-M_{k+1}(s))^{\eta}]ds. \label{optminoderykneg}
\end{align}
Iterating \eqref{optminoderykneg} we get that for any $n \geq 1$ and $t \geq 0$, 
\begin{align}
\mathbb{E} [ (1-M_k(t))^{\eta}] \geq \left ( \prod_{i=1}^n (\psi_{k+i}(\eta)- \lambda_{k+i}(\Lambda)) \right ) \int_{[0,t]^n} \mathbb{E}[(1-M_{k+n}(s_n))^{\eta}] \mathds{1}_{s_n\leq \cdots \leq s_1} ds_1 ... ds_n. \label{optminoderyknegiter}
\end{align}
We fix $n$ large enough so that $\phi_{k+n}(\eta)<\phi_S(\eta)+\epsilon$. We thus have $\phi_S(\eta)<\phi_{k+n}(\eta)<\phi_S(\eta)+\epsilon<0$. Using Lemma \ref{minoyklem} and integrating the variables one by one in \eqref{optminoderyknegiter} we get 
\begin{align*}
\mathbb{E} [ (1-M_k(t))^{\eta}] & \geq \left ( \prod_{i=1}^n (\psi_{k+i}(\eta)- \lambda_{k+i}(\Lambda)) \right ) \int_{[0,t]^{n-1}} \frac{e^{-s_{n-1}\phi_{k+n}(\eta)}-1}{-\phi_{k+n}(\eta)} \mathds{1}_{s_{n-1}\leq \cdots \leq s_1} ds_1 ... ds_{n-1} \nonumber \\
& = \left ( \prod_{i=1}^n (\psi_{k+i}(\eta)- \lambda_{k+i}(\Lambda)) \right ) \times \left ( \frac{e^{-t\phi_{k+n}(\eta)}}{(-\phi_{k+n}(\eta))^n} - \sum_{i=0}^{n-1} \frac{t^i}{i!(-\phi_{k+n}(\eta))^{n-i}} \right ). 
\end{align*}
We thus get that for any choice of $C(\epsilon)$ such that $0<C(\epsilon)<(-\phi_{k+n}(\eta))^{-n}\prod_{i=1}^n (\psi_{k+i}(\eta)- \lambda_{k+i}(\Lambda))$, \eqref{minoykoptpostcretaneg} holds true for large $t$. Decreasing $C(\epsilon)$ if necessary we obtain that the lower bound in \eqref{minoykoptpostcretaneg} holds true for all $t\geq 0$. 
\end{proof}

We now derive a bound in the other direction, as an easy consequence of Theorem \ref{ito}. 
\begin{lemma} \label{minozklem}
For any $\eta \geq 0$ and $k \geq 2$ there is $C>0$ such that for any $t \geq 1$, 
\begin{align}
\mathbb{E}[W_k(t)^\eta] \geq C e^{-\lambda_k(\Lambda) t}. \label{minozk}
\end{align}
\end{lemma}

\begin{proof}
From the definition of $K^k_{\cdot}(\cdot,\cdot)$ in \eqref{prepCFlb} we have almost surely $K^k_t(r,u) \geq - \mathds{1}_{\beta_k(t,r,u)\geq 2} W_k(t)$. Applying \eqref{derfwkt} from Theorem \ref{ito} with $f(x):=x^{\eta}$, recalling from Lemma \ref{lawofvkandzk} that $\beta_k(t,r,u)$ is independent of $(W_j(t))_{j\geq 1}$ and that $\beta_k(t,r,u)$ follows the binomial distribution with parameter $(k,r)$, and using the definition of $\lambda_k(\Lambda)$, we obtain 
\begin{align*}
\frac{d}{dt} \mathbb{E}[W_k(t)^\eta] & \geq \int_{(0,1)^2} \mathbb{E} \left [ (W_k(t)- \mathds{1}_{\beta_k(t,r,u)\geq 2} W_k(t) )^{\eta} - W_k(t)^{\eta} \right ] r^{-2} \Lambda(dr) du \\
& = -\int_{(0,1)} \mathbb{E}[\mathds{1}_{\beta_k(t,r,u)\geq 2} W_k(t)^\eta] r^{-2} \Lambda(dr)du \\
& = -\mathbb{E}[W_k(t)^\eta] \int_{(0,1)} (1-(1-r)^k - kr(1-r)^{k-1}) r^{-2} \Lambda(dr) = -\lambda_k(\Lambda) \mathbb{E}[W_k(t)^\eta]. 
\end{align*}
Therefore $\frac{d}{dt} \log \mathbb{E}[W_k(t)^\eta] \geq -\lambda_k(\Lambda)$ so, for any $t\geq 1$, $\mathbb{E}[W_k(t)^\eta] \geq \mathbb{E}[W_k(1)^\eta]e^{-\lambda_k(\Lambda) (t-1)}$. We have clearly $\mathbb{P}(W_k(1)>0)>0$ so $\mathbb{E}[W_k(1)^\eta]>0$. We thus get \eqref{minozk}. 
\end{proof}

\begin{remark}
Alternatively, Lemma \ref{minozklem} is derived as follows. Let $B^k_{1,t}$ be the event where the $k$ blocks with mass $W_1(1),\ldots,W_k(1)$ at time $1$ do not merge on the time interval $(1,t]$. Since these blocks form a $\Lambda$-coalescent by Proposition \ref{combrepres} we have that, on $\{W_k(1)>0\}$, $\mathbb{P}(B^k_{1,t}|\mathcal{F}_{1})=e^{-(t-1)\lambda_k(\Lambda)}$. Moreover, on $\{W_k(1)>0\} \cap B^k_{1,t}$, we have $W_k(t) \geq W_k(1)$. Therefore, using conditioning with respect to $\mathcal{F}_{1}$, we get $\mathbb{E}[W_k(t)^\eta] \geq e^{-(t-1)\lambda_k(\Lambda)}\mathbb{E}[W_k(1)^\eta]$. A similar but slightly more elaborate argument will be used below to justify Proposition \ref{longtimelowerboundspeprop}, and a combination of the two arguments will be used to prove the large deviation lower bound \eqref{ldplowerbound} in Section \ref{longtimebehavccl}. 
\end{remark}

\begin{remark} \label{completelowerboundmkt}
For any $k \geq 1$ we have $W_{k+1}(t) \leq 1-M_k(t)$ so $\mathbb{E} [ (1-M_k(t))^{\eta}] \geq \mathbb{E}[W_{k+1}(t)^{\eta}]$ when $\eta>0$. Combining this with Lemma \ref{minozklem} allows to complete \eqref{minoykopt} with a lower bound of the form $\mathbb{E} [ (1-M_k(t))^{\eta}] \geq c e^{-t \lambda_{k+1}(\Lambda)}$, for some constant $c$ and all $t\geq 1$. 
\end{remark}

\subsection{A bound based on Theorem \ref{represwktviamut}} \label{longtimelowerboundspe}

In this subsection we establish a lower bound for $\mathbb{E}[W_k(t)^\eta]$ that complements those of Section \ref{longtimeupperbound} and that is based on the Poisson representation of Theorem \ref{represwktviamut}. 
\begin{prop} \label{longtimelowerboundspeprop}
For any $\eta \geq 0$ and $k\geq 2$, there is $C>0$ such that for any $t\geq 1$ we have 
\begin{align}
\mathbb{E} [W_k(t)^{\eta}] \geq C e^{-t\phi_S(\eta)}. \label{longtimelowerboundspeprop0}
\end{align}
\end{prop}
The idea of the proof is relatively simple and consists in using the fact that each jump of $N$ gives rise to a block that is at least proportional to the size of the dust at the jump time (see Theorem \ref{represwktviamut}). Therefore, if $N$ has $k$ large jumps on an interval $[t-1,t]$, if the resulting $k$ blocks do not merge within time $t$, and if the size of the dust does not change too much on $[t-1,t]$, then there are at least $k$ blocks at time $t$ that are at least proportional to the size of the dust, and so is $W_k(t)$. We now make this heuristic rigorous. Let $t_1,t_2 \geq 0$ with $t_1<t_2$, $k \geq 2$, $\epsilon \in (0,1)$ and $\alpha \in (0,\infty]$. We set 
\begin{align*}
E^{\epsilon}_{t_1,t_2}:= \{ (s,r,u) \in N \ \text{s.t.} \ s \in (t_1,t_2], \ r>\epsilon \}, \qquad M^{\epsilon}_{t_1,t_2}:= \sharp E^{\epsilon}_{t_1,t_2}. 
\end{align*}
For $i \in \{1,\dots,M^{\epsilon}_{t_1,t_2}\}$, let $(s_i,r_i,u_i)$ be the $i^{th}$ element of $E^{\epsilon}_{t_1,t_2}$, where the ordering is such that $s_1<s_2<\ldots<s_{M^{\epsilon}_{t_1,t_2}}$. Let 
\begin{align}
\mathcal{E}(t_1,t_2,k,\epsilon,\alpha) := \left \{ M^{\epsilon}_{t_1,t_2}= k, S_{t_2}-S_{t_1} \leq \alpha, \forall i \neq j \in \{1,\ldots,k\}, Y_{s_i,t_2}(u_i) \neq Y_{s_j,t_2}(u_j) \right \}. \label{defspeevinter}
\end{align}
Note that the event $\mathcal{E}(t_1,t_2,k,\epsilon,\alpha)$ is independent from the sigma-field $\mathcal{F}_{t_1}$ from Section \ref{model}. 

\begin{lemma} \label{minoonspeevt}
Let $t_1,t_2 \geq 0$ with $t_1<t_2$, $k \geq 2$, $\epsilon \in (0,1)$ and $\alpha >0$. On the event $\mathcal{E}(t_1,t_2,k,\epsilon,\alpha)$ we have almost surely $W_k(t_2)\geq \epsilon e^{-\alpha} e^{-S_{t_1}}$. 
\end{lemma}

\begin{proof}
Assume we are on the event $\mathcal{E}(t_1,t_2,k,\epsilon,\alpha)$ and on the probability one event where \eqref{defmut} holds true at $t=t_2$ (see Theorem \ref{represwktviamut}). Note from \eqref{defmut} that for any $i \in \{1,\dots,k\}$, 
\begin{align*}
\mu_{t_2}(\{Y_{s_i,t_2}(u_i)\}) \geq r_i e^{-S_{s_i}} > \epsilon e^{-(S_{t_2} - S_{t_1})} e^{-S_{t_1}} \geq \epsilon e^{-\alpha} e^{-S_{t_1}}. 
\end{align*}
This shows that, for each $i \in \{1,\dots,k\}$, $Y_{s_i,t_2}(u_i)$ is an atom of $\mu_{t_2}$ with weight larger than $\epsilon e^{-\alpha} e^{-S_{t_1}}$, and, since $Y_{s_i,t_2}(u_i)\neq Y_{s_j,t_2}(u_j)$ for $i,j \in \{1,\dots,k\}$ with $i\neq j$, these atoms are pairwise distinct. Therefore $\mu_{t_2}$ has at least $k$ atoms with weight larger than $\epsilon e^{-\alpha} e^{-S_{t_1}}$. Since, by \eqref{represwktviamut2}, $W_k(t_2)$ is the $k^{th}$ largest weight of atoms of $\mu_{t_2}$, we get $W_k(t_2)\geq \epsilon e^{-\alpha} e^{-S_{t_1}}$. 
\end{proof}

\begin{lemma} \label{probspeevt}
Let $w>0$, $k \geq 2$, $\epsilon \in (0,\max \mathrm{Supp} \Lambda)$. For all $\alpha>0$ large there is $c(w,k,\epsilon,\alpha)>0$ such that for all $t\geq 0$ we have almost surely $\mathbb{P}(\mathcal{E}(t,t+w,k,\epsilon,\alpha)|\mathcal{F}_t) \geq c(w,k,\epsilon,\alpha)$. 
\end{lemma}
The proof of this lemma relies on properties that are established in Appendix \ref{coaltrajy} and is therefore deferred to that appendix (see proof of Lemma \ref{probspeevt} in Appendix \ref{coaltrajysec}). 

\begin{proof} [Proof of Proposition \ref{longtimelowerboundspeprop}]
Fix $\eta \geq 0$, $k \geq 2$ and $\epsilon \in (0,\max \mathrm{Supp} \Lambda)$. According to Lemma \ref{probspeevt} there is $\alpha>0$ and a constant $c>0$ such that for any $t\geq 1$ we have $\mathbb{P}(\mathcal{E}(t-1,t,k,\epsilon,\alpha)|\mathcal{F}_{t-1}) \geq c$. According to Lemma \ref{minoonspeevt} we have $W_k(t)\geq \epsilon e^{-\alpha} e^{-S_{t-1}}$ almost surely on $\mathcal{E}(t-1,t,k,\epsilon,\alpha)$. We thus get that for any $t\geq 1$, 
\begin{align*}
\mathbb{E} [W_k(t)^{\eta}] \geq \epsilon^\eta e^{-\alpha \eta} \mathbb{E} [e^{-\eta S_{t-1}} \mathds{1}_{\mathcal{E}(t-1,t,k,\epsilon,\alpha)}] = \epsilon^\eta e^{-\alpha \eta} \mathbb{E} [e^{-\eta S_{t-1}} \mathbb{P}(\mathcal{E}(t-1,t,k,\epsilon,\alpha)|\mathcal{F}_{t-1})], 
\end{align*}
Combining with Lemma \ref{probspeevt} and the definition of $\phi_S(\cdot)$ in Section \ref{model3} we get 
\begin{align*}
\mathbb{E} [W_k(t)^{\eta}] \geq c \epsilon^\eta e^{-\alpha \eta} \mathbb{E} [e^{-\eta S_{t-1}}] = c \epsilon^\eta e^{-\alpha \eta} e^{-(t-1)\phi_S(\eta)}, 
\end{align*}
which yields \eqref{longtimelowerboundspeprop0}. 
\end{proof}

\subsection{Conclusion: Proof of the main results} \label{longtimebehavccl}

In this subsection we prove Theorem \ref{longtimebehavth}, Theorem \ref{thmpgd}, Corollary \ref{corlgn}, and Theorem \ref{thmpgdwk}. 

\begin{proof} [Proof of Theorem \ref{longtimebehavth}]
When $\eta \in [0,1]$, \eqref{lgtimeboundw1t} follows directly from the combination of Proposition \ref{minoyklem2} with Remark \ref{completelowerboundmkt}. Then recall that for any $k \geq 2$ we have $W_k(t) \leq 1-M_{k-1}(t)$ so, for $\eta \in [0,1]$, 
\begin{align*}
\mathbb{E}[W_k(t)^{\eta}] \leq \mathbb{E} [ (1-M_{k-1}(t))^{\eta}]. 
\end{align*}
Combining this with Propositions \ref{minoyklem2} and \ref{longtimelowerboundspeprop}, and Lemma \ref{minozklem}, we obtain that \eqref{lgtimeboundwkt} holds true when $k \geq 2$ and $\eta \in [0,1]$. 

Then, for $\eta>1$ and $k \in \{2,..., N(\Lambda)-1\}$, we have $\phi_S(1)>\lambda_{k}(\Lambda)$ and 
\begin{align*}
\mathbb{E}[W_k(t)^{\eta}] & \leq \mathbb{E}[W_k(t)] \leq \mathbb{E} [1-M_{k-1}(t)], \\
\mathbb{E}[W_2(t)^{\eta}] & \leq \mathbb{E}[(1-W_1(t))^{\eta}] \leq \mathbb{E}[1-W_1(t)] = \mathbb{E} [1-M_{1}(t)]. 
\end{align*}
Combining these with \eqref{minoykopt} from Proposition \ref{minoyklem2} and Lemma \ref{minozklem} we obtain that \eqref{lgtimeboundw1t} holds true when $\eta>1$ and that \eqref{lgtimeboundwkt} holds true when $k \in \{2,..., N(\Lambda)-1\}$ and $\eta>1$. 

When $\eta \in (-\theta(\Lambda),0)$, we see from \eqref{minoykoptpostcretaneg} in Proposition \ref{minoyklem2} that $\frac1{t}\log \mathbb{E}[(1-W_1(t))^{\eta}]$ converges to $-\phi_S(\eta)$ as $t \to \infty$. Since $\phi_S(\eta)<0$ when $\eta \in (-\theta(\Lambda),0)$, we obtain that \eqref{lgtimeboundw1t} holds true in this case as well. Finally, if $\eta \in (-\infty,-\theta(\Lambda)]$, we fix $\epsilon>0$ and note that $\mathbb{E}[(1-W_1(t))^{\eta}] \geq \mathbb{E}[(1-W_1(t))^{-\theta(\Lambda)+\epsilon}]$ so, by \eqref{lgtimeboundw1t} applied at $-\theta(\Lambda)+\epsilon$, we get 
\begin{align} 
\liminf_{t \to \infty} \frac1{t}\log \mathbb{E}[(1-W_1(t))^{\eta}] \geq -(\phi_S(-\theta(\Lambda)+\epsilon) \wedge \lambda_2(\Lambda)). \label{caseinfinite}
\end{align}
When \eqref{continuitycondition} holds true, then the right-hand side converges to $\infty$ as $\epsilon \to 0$ so \eqref{lgtimeboundw1t} also holds true when $\eta \in (-\infty,-\theta(\Lambda)]$. 
\end{proof}

We now prove Theorem \ref{thmpgd} using Theorem \ref{longtimebehavth}, G\"artner-Ellis Theorem, and a construction of suitable events. 
\begin{proof} [Proof of Theorem \ref{thmpgd}]
Assume that \eqref{integassumption} and \eqref{continuitycondition} hold true. By Theorem \ref{longtimebehavth} we get that, for any $\eta \in \mathbb{R}$,  
\begin{align}
\frac1{t}\log \mathbb{E}[e^{-\eta \log (1-W_1(t))}] & \underset{t \to \infty}{\longrightarrow} -(\phi_S(-\eta) \wedge \lambda_2(\Lambda)). \label{asymptfunctional02psik}
\end{align}
Recall that the function in the right-hand side of \eqref{asymptfunctional02psik} is finite on the domain $\eta \in (-\infty,\theta(\Lambda))$ and, since $\theta(\Lambda)>0$ by assumption \eqref{continuitycondition}, that interval contains the origin. Therefore the family $(-\log (1-W_1(t))/t)_{t>0}$ satisfies Assumption 2.3.2 of \cite{dembozeitouni}. Then, the upper bound in G\"artner-Ellis Theorem (see \cite[Thm. 2.3.6(a)]{dembozeitouni}) yields directly the upper bound \eqref{ldpupperbound}. 

The lower bound in G\"artner-Ellis Theorem (see \cite[Thm. 2.3.6(b)]{dembozeitouni}) yields that, for any open set $G \subset \mathbb{R}$, 
\begin{align}
\liminf_{t \rightarrow \infty} \frac{1}{t} \log \mathbb{P} \left ( \frac{-\log (1-W_1(t))}{t} \in G \right ) \geq -\inf_{x \in G \cap \mathcal{E}xp} \mathcal{J}_2(x), \label{ldplowerboundlvlkpartial}
\end{align}
where $\mathcal{E}xp$ is the set of \textit{exposed points} of $\mathcal{J}_2(\cdot)$ (see \cite[Def. 2.3.3]{dembozeitouni} for a definition), whose \textit{exposing hyperplane} belongs to $(-\infty,\theta(\Lambda))$. Let $\gamma_2 := \inf \{ \eta \geq 0, \phi_S(\eta)>\lambda_2(\Lambda) \}$, as in the discussion after \eqref{defiaslegendrezeta}. Then $\eta \mapsto -(\phi_S(-\eta) \wedge \lambda_2(\Lambda))$ is $\mathcal{C}^1$ on $(-\gamma_2,\theta(\Lambda))$ with increasing derivative $\phi_S'(-\cdot)$. That derivative has limit $\phi_S'(\gamma_2)$ and $\infty$ at respectively $-\gamma_2$ and $\theta(\Lambda)$. Indeed, if the limit at $\theta(\Lambda)$ was finite, then $\phi_S(-\cdot)$ would also have a finite limit at $\theta(\Lambda)$, which would contradict \eqref{continuitycondition}. Therefore, for any $x \in (\phi_S'(\gamma_2),\infty)$ there is $\eta \in (-\gamma_2,\theta(\Lambda))$ such that $\phi_S'(-\eta)=x$. By \cite[Lem. 2.3.9]{dembozeitouni} we deduce that $(\phi_S'(\gamma_2),\infty) \subset \mathcal{E}xp$. 

To prove \eqref{ldplowerbound} we only need to show that, for any open set $G \subset \mathbb{R}$ and $\epsilon>0$ we have 
\begin{align}
\liminf_{t \rightarrow \infty} \frac{1}{t} \log \mathbb{P} \left ( \frac{-\log (1-W_1(t))}{t} \in G \right ) \geq -\inf_{x \in G} \mathcal{J}_2(x) - \epsilon. \label{ldplowerboundlvlkepsilon}
\end{align}
We thus fix an open set $G \subset \mathbb{R}$ and $\epsilon>0$. If $G \subset (-\infty,0)$ there is nothing to prove so we assume that $G \cap [0,\infty) \neq \emptyset$. If there is $z \in (\phi_S'(\gamma_2),\infty) \cap G$ such that 
\begin{align}
\mathcal{J}_2(z) < \inf_{x \in G} \mathcal{J}_2(x)+\epsilon, \label{approxinf}
\end{align}
then \eqref{ldplowerboundlvlkepsilon} follows from \eqref{ldplowerboundlvlkpartial} since $(\phi_S'(\gamma_2),\infty) \subset \mathcal{E}xp$. Recall from the discussion after \eqref{defiaslegendrezeta} that $\mathcal{J}_2(\cdot)$ is continuous on $[0,\phi_S'(\gamma_2)]$. Therefore, if there is no $z \in (\phi_S'(\gamma_2),\infty) \cap G$ satisfying \eqref{approxinf}, then there is $z \in (0,\phi_S'(\gamma_2)) \cap G$ satisfying \eqref{approxinf}. Let us fix $\delta >0$ such that $(z-\delta,z+\delta)\subset (0,\phi_S'(\gamma_2)) \cap G$ and $(1+\delta/2z)\phi_S'(\gamma_2)<\phi_S'(0)$. This is possible since $\phi_S'(\gamma_2)<\phi_S'(0)$, as $\phi_S'(\cdot)$ is decreasing. We set $t_0:=tz/\phi_S'(\gamma_2) \in (0,t)$ and define the event 
\begin{align*}
A_t := \{ S_{t_0} < (1+\delta/2z) t_0 \phi_S'(\gamma_2) \} \cap \mathcal{E}(t_0,t_0+1,2,h,\alpha) \cap B_{t_0+1,t}, 
\end{align*}
where $\mathcal{E}(\cdot,\cdot,\cdot,\cdot,\cdot)$ is as in \eqref{defspeevinter}, $h \in (0,\max \mathrm{Supp} \Lambda)$, $\alpha$ is provided by Lemma \ref{probspeevt} so that $c(1,2,h,\alpha)>0$, and $B_{t_1,t_2}$ refers to the event where the two blocks with mass $W_1(t_1)$ and $W_2(t_1)$ at time $t_1$ do not merge on the time interval $(t_1,t_2]$. Since these blocks form a $\Lambda$-coalescent by Proposition \ref{combrepres}, they merge at rate $\lambda_2(\Lambda)$ (see the discussion after \eqref{pushingrates}). We thus have $\mathbb{P}(B_{t_1,t_2}|\mathcal{F}_{t_1})=e^{-(t_2-t_1)\lambda_2(\Lambda)}$. By Cram\'er's theorem (see \cite[Thm. 2.2.3]{dembozeitouni}) the family $(S_t/t)_{t>0}$ satisfies a LDP with rate function $\mathcal{J}(\cdot)$, the Legendre transform of $\eta \mapsto -\phi_S(-\eta)$. We thus get 
\begin{align}
\liminf_{t \rightarrow \infty} \frac{1}{t_0} \log \mathbb{P} (S_{t_0} < (1+\delta/2z) t_0 \phi_S'(\gamma_2)) \geq -\mathcal{J}(\phi_S'(\gamma_2)). \label{cramersub}
\end{align}
Conditioning with respect to $\mathcal{F}_{t_0+1}$ and then to $\mathcal{F}_{t_0}$ and using the above estimates and Lemma \ref{probspeevt} we get 
\begin{align*}
\mathbb{P} (A_t) & \geq c(1,2,h,\alpha) e^{-(t-t_0-1)\lambda_2(\Lambda)} \mathbb{P} (S_{t_0} < (1+\delta/2z) t_0 \phi_S'(\gamma_2)). 
\end{align*}
Combining this with \eqref{cramersub} we get 
\begin{align}
\liminf_{t \rightarrow \infty} \frac1{t} \log \mathbb{P} (A_t) & \geq -(1-z/\phi_S'(\gamma_2))\lambda_2(\Lambda) - (z/\phi_S'(\gamma_2)) \mathcal{J}(\phi_S'(\gamma_2)) = -\mathcal{J}_2(z). \label{probapat}
\end{align}
For the last equality we have used that $\mathcal{J}_2(z)$ is affine on $[0,\phi_S'(\gamma_2)]$, taking values $\lambda_2(\Lambda)$ and $\mathcal{J}(\phi_S'(\gamma_2))$ at respectively $0$ and $\phi_S'(\gamma_2)$ (see the discussion after \eqref{defiaslegendrezeta}). We note that, on $B_{t_1,t_2}$ we have almost surely $W_2(t_2)\geq W_2(t_1)$. Combining with Lemma \ref{minoonspeevt} we obtain that, on $A_t$, we have almost surely 
\begin{align*}
1-W_1(t) \geq W_2(t) \geq W_2(t_0+1) \geq h e^{-\alpha} e^{-S_{t_0}} \geq h e^{-\alpha} e^{-(1+\delta/2z) t_0 \phi_S'(\gamma_2)} = h e^{-\alpha} e^{-(z+\delta/2) t}, 
\end{align*}
and the above is larger than $e^{-(z+\delta) t}$ for all (deterministically) large $t$. We get that, for large $t$, 
\begin{align}
\mathbb{P} \left ( \frac{-\log (1-W_1(t))}{t} < z+\delta \right ) \geq \mathbb{P} (A_t). \label{ldplowerboundlarge}
\end{align}
Using the upper bound \eqref{ldpupperbound} and that $\mathcal{J}_2(\cdot)$ is decreasing on $(0,\phi_S'(\gamma_2))$ (see the discussion after \eqref{defiaslegendrezeta}) we get that for large $t$, 
\begin{align}
\mathbb{P} \left ( \frac{-\log (1-W_1(t))}{t} \in [0,z-\delta] \right ) \leq e^{-t \mathcal{J}_2(z-\delta/2)} << e^{-t \mathcal{J}_2(z)}. \label{ldplowerboundsmall}
\end{align}
Combining \eqref{ldplowerboundlarge} and \eqref{ldplowerboundsmall} we obtain that, for all large $t$, 
\begin{align*}
\mathbb{P} \left ( \frac{-\log (1-W_1(t))}{t} \in G \right ) & \geq \mathbb{P} \left ( \frac{-\log (1-W_1(t))}{t} \in (z-\delta,z+\delta) \right ) \\
& = \mathbb{P} \left ( \frac{-\log (1-W_1(t))}{t} < (z+\delta) t \right ) \\
& \qquad \qquad - \mathbb{P} \left ( \frac{-\log (1-W_1(t))}{t} \in [0,z-\delta] \right ) \\
& \geq \mathbb{P} (A_t) - e^{-t \mathcal{J}_2(z-\delta/2)}. 
\end{align*}
Using \eqref{probapat} we get that the $\liminf$ in \eqref{ldplowerboundlvlkepsilon} is larger than $-\mathcal{J}_2(z)$. Since $z$ has been chosen so that it satisfies \eqref{approxinf}, we deduce that \eqref{ldplowerboundlvlkepsilon} holds true, concluding the proof. 
\end{proof}

\begin{proof} [Proof of Corollary \ref{corlgn}]
We see from \eqref{defiaslegendrezeta} that $\mathcal{J}_k(x)\leq 0$ if and only if $\phi_S(\eta) \wedge \lambda_k(\Lambda)\leq \eta x$ for all $\eta \in \mathbb{R}$ and, since $\phi_S(0)=0$ and $\phi_S(\cdot) \wedge \lambda_k(\Lambda)$ is concave, this occurs only for $x=\phi_S'(0)$. Therefore, $\mathcal{J}_k(x)>0$ for all $x \neq \phi_S'(0)$. As a Legendre transform, the function $\mathcal{J}_k(\cdot)$ is convex (see for example \cite[Lem. 2.2.5(a)]{dembozeitouni}). Therefore, $\mathcal{J}_k(\cdot)$ is decreasing and positive on $(-\infty,\phi_S'(0))$ and increasing and positive on $(\phi_S'(0),\infty)$. Indeed, convexity of $\mathcal{J}_k(\cdot)$ and uniqueness of $\phi_S'(0)$ as a zero together rule out any non-negative (resp. non-positive) increment before (resp. after) $\phi_S'(0)$. For any $\epsilon>0$, setting $F_{\epsilon}:=(-\infty,\phi_S'(0)-\epsilon] \cup [\phi_S'(0)+\epsilon,\infty)$ we have $c(\epsilon):=\frac1{2}\inf_{x \in F_{\epsilon}} \mathcal{J}_2(x)>0$ so, applying \eqref{ldpupperbound} with $F=F_{\epsilon}$ we get that for all large $t$, 
\begin{align*}
\mathbb{P} \left ( \left | \frac{-\log (1-W_1(t))}{t} - \phi_S'(0) \right | >\epsilon \right ) \leq e^{-tc(\epsilon)}. 
\end{align*}
By the Borel-Cantelli lemma we deduce that the a.s. convergence \eqref{lgn} holds true for $t$ restricted to positive integers. Since $t \mapsto W_1(t)$ is almost surely increasing \eqref{lgn} easily follows. 
\end{proof}

We now prove Theorem \ref{thmpgdwk} using Theorem \ref{longtimebehavth} and the construction of suitable events from the proof of Theorem \ref{thmpgd}. 
\begin{proof} [Proof of Theorem \ref{thmpgdwk}]
Assume that \eqref{integassumption} holds true. For $\eta \in \mathbb{R}$ let us define 
\begin{align*}
\ell_k(\eta) := \limsup_{t \to \infty} \frac1{t}\log \mathbb{E}[e^{-\eta \log W_k(t)}] \in [-\infty,\infty]. 
\end{align*}
By \cite[Thm. 4.5.3]{dembozeitouni}, $\ell_k(\cdot)$ is a convex function and, if $\mathcal{L}_k$ denotes its Legendre transform (i.e. $\mathcal{L}_k(x):=\sup \{ \eta x -\ell_k(\eta); \eta \in \mathbb{R} \}$), then the upper bound \eqref{ldpupperboundwk} holds true with $\mathcal{J}_k(\cdot)$ replaced by $\mathcal{L}_k(\cdot)$. To prove \eqref{ldpupperboundwk}, we are thus left to prove that, for any $x \in (-\infty,\phi_S'(0)]$, we have $\mathcal{L}_k(x)=\mathcal{J}_k(x)$. By Theorem \ref{longtimebehavth} we have that $\ell_k(\cdot)=-(\phi_S(-\cdot) \wedge \lambda_k(\Lambda))$ on $(-\infty,0]$. This ensures that $\mathcal{L}_k(x)=\mathcal{J}_k(x)$ for all values of $x$ such that the suprema in the definitions of $\mathcal{L}_k(x)$ and $\mathcal{J}_k(x)$ are both approached via sequences of $\eta\in (-\infty,0]$. In particular, $\mathcal{L}_k(x)=\mathcal{J}_k(x)=\infty$ when $x<0$. For $x\in [0,\phi_S'(0))$, the supremum in the definition \eqref{defiaslegendrezeta} of $\mathcal{J}_k(\cdot)$ is reached at some $\eta_x<0$. Indeed, if the supremum can be approached via a sequence of $\eta\in (-\infty,-\gamma_k]$ then $\eta_x=-\gamma_k<0$, otherwise the supremum is reached on $\eta_x\in (-\gamma_k,\theta(\Lambda))$ such that $x=\phi_S'(-\eta_x)$ and, since $\phi_S'(-\eta_x)=x<\phi_S'(0)$ and $\phi_S'(\cdot)$ is decreasing, we have $\eta_x<0$. As a consequence, for any $\eta \in (\eta_x,0]$ we have $(\ell_k(\eta)-\ell_k(\eta_x))/(\eta-\eta_x)\geq x$. By convexity of $\ell_k(\cdot)$, the left-hand side is non-decreasing in $\eta$ so this inequality extends to all $\eta \geq 0$. Therefore the supremum in the definition of $\mathcal{L}_k(\cdot)$ is also reached at $\eta_x<0$ and we deduce that $\mathcal{L}_k(x)=\mathcal{J}_k(x)$. For $x=\phi_S'(0)$ we have $\mathcal{J}_k(x)=0$ and, as a Legendre transform, $\mathcal{L}_k(\cdot)$ is lower semicontinuous at $x$ (see \cite[Proof of Lem. 2.2.5]{dembozeitouni}) so $0 \leq \mathcal{L}_k(x) \leq \liminf_{z \to x-} \mathcal{L}_k(z)=\liminf_{z \to x-} \mathcal{J}_k(z)=0$ (indeed, $\mathcal{J}_k(\cdot)$ is convex and finite on $[0,\infty)$, and therefore continuous on $(0,\infty)$) thus $\mathcal{L}_k(x)=\mathcal{J}_k(x)$. This completes the proof of the upper bound \eqref{ldpupperboundwk}. 

To prove the lower bound \eqref{ldplowerboundwk} we only need to show that, for any open set $G \subset (-\infty,\phi_S'(0))$ and $\epsilon>0$ we have 
\begin{align}
\liminf_{t \rightarrow \infty} \frac{1}{t} \log \mathbb{P} \left ( \frac{-\log W_k(t)}{t} \in G \right ) \geq -\inf_{x \in G} \mathcal{J}_k(x) - \epsilon. \label{ldplowerboundlvlkepsilonwk}
\end{align}
We thus fix an open set $G \subset (-\infty,\phi_S'(0))$ and $\epsilon>0$. If $G \subset (-\infty,0)$ there is nothing to prove so we assume that $G \cap [0,\phi_S'(0)) \neq \emptyset$. As in the proof of Theorem \ref{thmpgd}, either there is $z \in (\phi_S'(\gamma_k),\phi_S'(0)) \cap G$ satisfying 
\begin{align}
\mathcal{J}_k(z) < \inf_{x \in G} \mathcal{J}_k(x)+\epsilon, \label{approxinfwk}
\end{align}
or there is $z \in (0,\phi_S'(\gamma_k)) \cap G$ satisfying \eqref{approxinfwk}. In the second case we can proceed exactly as in the proof of Theorem \ref{thmpgd}, replacing $\gamma_2$ by $\gamma_k$ and the event $A_t$ there by 
\begin{align*}
\tilde A_t := \{ S_{t_0} < (1+\delta/2z) t_0 \phi_S'(\gamma_k) \} \cap \mathcal{E}(t_0,t_0+1,k,h,\alpha) \cap B^k_{t_0+1,t}, 
\end{align*}
where $\mathcal{E}(\cdot,\cdot,\cdot,\cdot,\cdot)$ is as in \eqref{defspeevinter}, $h \in (0,\max \mathrm{Supp} \Lambda)$, $\alpha$ is provided by Lemma \ref{probspeevt} so that $c(1,k,h,\alpha)>0$, and $B^k_{t_1,t_2}$ refers to the event where the $k$ blocks with mass $W_1(t_1), \ldots, W_k(t_1)$ at time $t_1$ do not merge on the time interval $(t_1,t_2]$. This adaptation yields \eqref{ldplowerboundlvlkepsilonwk}. 

We are now left with the case where there is $z \in (\phi_S'(\gamma_k),\phi_S'(0)) \cap G$ satisfying \eqref{approxinfwk}. Then we fix $\delta >0$ such that $(z-\delta,z+\delta)\subset (\phi_S'(\gamma_k),\phi_S'(0)) \cap G$ and define the event 
\begin{align*}
\hat A_t := \{ S_{t-1} < (z+\delta/2) (t-1) \} \cap \mathcal{E}(t-1,t,k,h,\alpha), 
\end{align*}
where $\mathcal{E}(\cdot,\cdot,\cdot,\cdot,\cdot)$ is as in \eqref{defspeevinter}, $h \in (0,\max \mathrm{Supp} \Lambda)$, $\alpha$ is provided by Lemma \ref{probspeevt} so that $c(1,k,h,\alpha)>0$. Recall from the proof of Theorem \ref{thmpgd} that $(S_t/t)_{t>0}$ satisfies a LDP with rate function $\mathcal{J}(\cdot)$. We thus get 
\begin{align}
\liminf_{t \rightarrow \infty} \frac{1}{t} \log \mathbb{P} (S_{t-1}<(z+\delta/2) (t-1)) \geq -\mathcal{J}(z)=-\mathcal{J}_k(z), \label{cramersubwk}
\end{align}
where we have used that $\mathcal{J}_k(\cdot)$ coincides with $\mathcal{J}(\cdot)$ on $[\phi_S'(\gamma_k),\infty)$ (see the discussion after \eqref{defiaslegendrezeta}). Conditioning with respect to $\mathcal{F}_{t_0}$ and using the above estimates and Lemma \ref{probspeevt} we get 
\begin{align*}
\mathbb{P} (\hat A_t) & \geq c(1,k,h,\alpha) \mathbb{P} (S_{t-1}<(z+\delta/2) (t-1)). 
\end{align*}
Combining this with \eqref{cramersub} we get 
\begin{align}
\liminf_{t \rightarrow \infty} \frac1{t} \log \mathbb{P} (\hat A_t) & \geq -\mathcal{J}_k(z). \label{probapatwk}
\end{align}
By Lemma \ref{minoonspeevt} we obtain that, on $\hat A_t$ (when $t \geq 1$), we have almost surely 
\begin{align*}
W_k(t) \geq h e^{-\alpha} e^{-S_{t-1}} \geq h e^{-\alpha} e^{-(z+\delta/2) (t-1)} = h e^{-\alpha+(z+\delta/2)} e^{-(z+\delta/2) t}, 
\end{align*}
and the above is larger than $e^{-(z+\delta) t}$ for all (deterministically) large $t$. We get that, for large $t$, 
\begin{align}
\mathbb{P} \left ( \frac{-\log (W_k(t))}{t} < z+\delta \right ) \geq \mathbb{P} (\hat A_t). \label{ldplowerboundlargewk}
\end{align}
Using the upper bound \eqref{ldpupperboundwk} and that $\mathcal{J}_k(\cdot)$ is decreasing on $[0,\phi_S'(0)]$ (see the discussion after \eqref{defiaslegendrezeta}) we get that for large $t$, 
\begin{align}
\mathbb{P} \left ( \frac{-\log (W_k(t))}{t} \in [0,z-\delta] \right ) \leq e^{-t \mathcal{J}_k(z-\delta/2)} << e^{-t \mathcal{J}_k(z)}. \label{ldplowerboundsmallwk}
\end{align}
Then we can conclude as in the end of the proof of Theorem \ref{thmpgd}, using \eqref{ldplowerboundlargewk}, \eqref{probapatwk} and \eqref{ldplowerboundsmallwk}, that the $\liminf$ in \eqref{ldplowerboundlvlkepsilonwk} is larger than $-\mathcal{J}_k(z)$. Since $z$ has been chosen so that it satisfies \eqref{approxinfwk}, we deduce that \eqref{ldplowerboundlvlkepsilonwk} holds true, concluding the proof. 
\end{proof}

\section{From the flow to a Poisson representation of block masses} \label{poisrep}

Recall that we always assume that \eqref{integassumption} holds true. 

\subsection{Explicit representation of the $t$-families} \label{poisrep1ststep}

In this subsection we prove Proposition \ref{poissrepstep1}, which is an important step toward the Poisson representation of Theorem \ref{represwktviamut}. It identifies the sets $C_t$ of Section \ref{model2} -- whose connected components are the $t$-families -- with the sets $D_t$ defined explicitly below in terms of the jumps of $N$. For any $t>0$ we set 
\begin{align}
D_t := \cup_{(s,r,u) \in N, s \in (0,t]} Y_{0,s-}^{-1}(I_{r,u}^{\mathrm{o}}), \qquad D_{t-} := \cup_{s \in (0,t)} D_{s} = \cup_{(s,r,u) \in N, s \in (0,t)} Y_{0,s-}^{-1}(I_{r,u}^{\mathrm{o}}), \label{cadlagsets}
\end{align} 
where $I_{r,u}$ is defined in \eqref{defiru}. The combination of Proposition \ref{nonzerodiff} with Lemma \ref{avoidboundaries2} (see Appendix \ref{poisrep1ststepappend}) yields the following lemma. 
\begin{lemma} \label{nonzerodiffrmk2}
Almost surely, if $Y_{0,t}(y_a)=Y_{0,t}(y_b)$ (resp. $Y_{0,t-}(y_a)=Y_{0,t-}(y_b)$) for some $t>0$ and $y_a,y_b \in [0,1]$ such that $y_a < y_b$, then $(y_a,y_b) \subset D_t$ (resp. $(y_a,y_b) \subset D_{t-}$). 
\end{lemma}

\begin{prop} \label{poissrepstep1}
We have almost surely that $C_t=D_t$ and $C_{t-}=D_{t-}$ for all $t\geq 0$. In particular $C_{t-}=\cup_{s \in (0,t)} C_{s}$. 
\end{prop}

\begin{proof}
Note from \eqref{collage} that, almost surely, for any $(s,r,u) \in N$ and $t \geq s$, $Y_{0,t}(\cdot)$ is constant on $Y_{0,s-}^{-1}(I_{r,u}^{\mathrm{o}})$. Therefore, $m_t(Y_{0,s-}^{-1}(I_{r,u}^{\mathrm{o}}))=0$ and $Y_{0,s-}^{-1}(I_{r,u}^{\mathrm{o}}) \subset C_t$. We thus get that almost surely, for all $t\geq 0$, $D_t \subset C_t$. 

We now assume that we are on the probability one events provided by Proposition \ref{nonzerodiff} and Lemma \ref{nonzerodiffrmk2}, fix $t \geq 0$ and prove that $C_t \subset D_t$. Note that $\{0,1\} \subset D_t^c$ and Proposition \ref{nonzerodiff} implies that $\{0,1\} \subset C_t^c$. Now let $x \in D_t^c \cap (0,1)$ and $\epsilon \in (0,x \wedge (1-x))$. Since $(x-\epsilon,x+\epsilon) \nsubseteq D_t$, by Lemma \ref{nonzerodiffrmk2} we have $Y_{0,t}(x+\epsilon)>Y_{0,t}(x-\epsilon)$ so $m_t([x-\epsilon,x+\epsilon])>0$. Since this is satisfied for any $x \in D_t^c \cap (0,1)$ and small $\epsilon >0$ we get that $D_t^c \cap (0,1) \subset \mathrm{Supp}(m_t)=C_t^c$ so $C_t \subset D_t$. This concludes the proof that $C_t=D_t$ almost surely for all $t\geq 0$, and the proof for $C_{t-}=D_{t-}$ follows by the same argument. 
\end{proof}

\subsection{From flow of inverses to $\Lambda$-coalescent: proofs} \label{relationinvcoal}

In this subsection we prove Proposition \ref{combrepres} and Lemma \ref{lawofvkandzk}, which together establish that the construction of Section \ref{model2} indeed yields the process of ordered block masses of a $\Lambda$-coalescent. We recall that $(U_i)_{i\geq 1}$ and $\pi^Y_t$ are defined in Section \ref{model2}. 

\begin{proof}[Proof of Proposition \ref{combrepres}]

\textit{First point.} Let us fix a realization in the probability one events from Propositions \ref{propexistence} and \ref{poissrepstep1}. Thanks to Proposition \ref{poissrepstep1}, we only need to prove the claim for $(D_t)_{t\geq 0}$. For this, we need to verify that almost surely $(D_t)_{t\geq 0}$ is non-decreasing and $(D_t^c)_{t\geq 0}$ is c\`ad-l\`ag for the Hausdorff distance $d_H(\cdot,\cdot)$. The non-decreasing property follows from the definition of $(D_t)_{t\geq 0}$ in \eqref{cadlagsets}. The existence of left limits for $(D_t^c)_{t\geq 0}$ in the $d_H(\cdot,\cdot)$ topology follows from the non-decreasing property of $(D_t)_{t\geq 0}$ and \cite[Prop. 2.5.6]{fractal}. Note that these left limits are equal to the sets $D_{t-}^c$ from \eqref{cadlagsets}. We now prove right continuity by contradiction. Suppose there are $t \geq 0$, $\epsilon>0$ and a sequence $(t_n)_{n \geq 1}$ decreasing to $t$ such that $d_H(D_t^c,D_{t_n}^c)>\epsilon$ for all $n \geq 1$, then for each $n \geq 1$ there is $x_n \in D_t^c$ such that $\text{dist}(x_n,D_{t_n}^c)>\epsilon$. By compactness of $D_t^c$, there is a subsequence $(x_{n(i)})_{i \geq 1}$ converging to some $x \in D_t^c=C_t^c$. If $x \in (0,1)$, let $\tilde \epsilon \in (0,\epsilon)$ be such that $(x-\tilde \epsilon,x+\tilde \epsilon) \subset (0,1)$. There is $i_0 \geq 1$ such that $|x_{n(i)}-x|<\epsilon-\tilde \epsilon$ for all $i \geq i_0$. We get that $i \geq i_0 \Rightarrow d(x,D_{t_{n(i)}}^c)>\tilde \epsilon \Rightarrow (x-\tilde \epsilon,x+\tilde \epsilon) \subset D_{t_{n(i)}}=C_{t_{n(i)}}$. In particular, $Y_{0,t_{n(i)}}(x+\tilde \epsilon/2)-Y_{0,t_{n(i)}}(x-\tilde \epsilon/2)=0$ for all $i \geq i_0$. By Proposition \ref{propexistence}(ii) we get $Y_{0,t}(x+\tilde \epsilon/2)-Y_{0,t}(x-\tilde \epsilon/2)=0$, which contradicts $x \in C_t^c$. The case $x \in \{0,1\}$ is treated similarly. This ends the proof of the first point. 

\textit{Second point.} Let $\mathcal{S}$ be as defined in Lemma \ref{diffcountable}. By that lemma, the event $\mathcal{E} := \{ \{ U_i, i \geq 1 \} \cap \mathcal{S} = \emptyset \}$ has probability one. Let us consider a realization in this event and in the probability one events from Propositions \ref{propexistence}, \ref{poissrepstep1} and \ref{nonzerodiff}. Recall that, by definition of $(\pi^Y_t)_{t\geq 0}$ in Section \ref{model2}, $i \sim_{\pi^Y_t} j$ for some $t$ if and only if $Y_{0,t}(U_i)=Y_{0,t}(U_j)$. By Proposition \ref{nonzerodiff} and the definition of $\mathcal{E}$, this implies that $U_i,U_j \in Y^{-1}_{0,s-}(I_{r,u}^{\mathrm{o}})$ for some $(s,r,u) \in N$ with $s \in (0,t]$. All sets $Y^{-1}_{0,s-}(I_{r,u}^{\mathrm{o}})$ are open intervals by Proposition \ref{propexistence} so $U_i$ and $U_j$ lie in the same open connected component of $D_t$ and therefore of $C_t$ by Proposition \ref{poissrepstep1}. Reciprocally, if $U_i$ and $U_j$ lie in the same open connected component of $C_t$ then $m_t([U_i,U_j])=0$ by definition of $C_t$ in Section \ref{model2} so $Y_{0,t}(U_i)=Y_{0,t}(U_j)$ and $i \sim_{\pi^Y_t} j$. This concludes the proof of the second point. 

\textit{Third point.} The first point of the proposition and its proof (together with Proposition \ref{poissrepstep1}) imply that, almost surely, $(C_t)_{t\geq 0}$ is c\`ad-l\`ag for the topology on interval partitions considered in \cite[Sec. 1.1.2]{lecturebertoin2010} and that the left limit at any $t>0$ is given by $C_{t-}$. By \cite[Prop. 1.2]{lecturebertoin2010}, the function that maps $C_t$ (resp. $C_{t-}$) to the sequence $(W_k(t))_{k \geq 1}$ (resp. $(W_k(t-))_{k \geq 1}$) is continuous. This yields the third point. 

\textit{Fourth point.} The combination of the proof of the third point with Proposition \ref{poissrepstep1} shows that, almost surely, we have $\{ t \geq 0 \ \text{s.t.} \ W_k(t)\neq W_k(t-) \} \subset \{ t \geq 0 \ \text{s.t.} \ C_t \neq C_{t-} \} = \{ t \geq 0 \ \text{s.t.} \ D_t \neq D_{t-} \} \subset J_N$, where the last inclusion is a consequence of \eqref{cadlagsets}. This yields the fourth point. 

\textit{Fifth point.} The combination of the third point with \cite[Prop. 1.1]{lecturebertoin2010} shows that $t \mapsto (W_k(t))_{k \geq 1}$ is c\`ad-l\`ag. By the second point, \cite[Prop 1.3]{lecturebertoin2010}, and the definition of $(W_k(t))_{k \geq 1}$ in Section \ref{model2}, we get that, for any $t \geq 0$, the partition $\pi^Y_t$ almost surely possesses asymptotic frequencies and the ordered non-zero masses of its blocks are given by $(W_k(t))_{k \geq 1}$. Since $(\pi^Y_t)_{t\geq 0}$ is a $\Lambda$-coalescent by Lemma \ref{lawofpartbyblg1and20} the fifth point follows. 
\end{proof}

\begin{proof}[Proof of Lemma \ref{lawofvkandzk}]
By the second point of Proposition \ref{combrepres}, the non-singleton blocks of $\pi^Y_t$ are given by $(A_k)_{k\geq 1}$ where $A_k:=\{i \geq 1, U_i \in \mathcal{O}_k(t)\}$. By the law of large numbers, $\lim_{n \rightarrow \infty} \sharp (A_k \cap [\![ 1,n ]\!])/n=P(U_1 \in \mathcal{O}_k(t))=W_k(t)$. In particular, $(W_k(t))_{k\geq 1}$ is a function of the partition $\pi^Y_t$. 
Moreover, for any $k$ such that $A_k \neq \emptyset$ and any $i \in A_k$ we have $Y_{0,t}(U_i)=V_k(t)$. If $A_k = \emptyset$ then $V_k(t)=\tilde U_k$ (see Section \ref{model2}). By the discussion after Proposition \ref{propexistence}, $Y_{0,t}(\cdot)$ is equal in law to the inverse of a bridge $B$. From the definition of $\pi^Y_t$ and the above, we see that $B$, $\pi^Y_t$ and $(V_k(t))_{k\geq 1}$ are as the bridge, the partition and a subfamily to the sequence considered in \cite[Lem. 2]{BLGI} (the subfamily associated to non-singleton blocks of $\pi^Y_t$). By that lemma, $(V_k(t))_{k\geq 1}$ and $\pi^Y_t$ are independent and $(V_k(t))_{k\geq 1}\sim \mathcal{U}([0,1])^{\otimes \mathbb{N}}$. All the claims of the lemma follow. 
\end{proof}

The following remark extends Lemma \ref{lawofvkandzk} to jump times of $N$, and will be used in the proof of Theorem \ref{intsto}.
\begin{remark} \label{atjumptime}
For $\eta\in(0,1)$, let $(s^{\eta}_i,r^{\eta}_i,u^{\eta}_i)_{i\geq 1}$ be the enumeration of $\{(s,r,u) \in N, r>\eta\}$ by increasing time components. For any $i \geq 1$, the argument in the proof of Lemma \ref{lawofvkandzk} can be applied at $s^{\eta}_i$ (instead of a fixed $t$) and shows that, conditionally on $s^{\eta}_i$, $(V_k(s^{\eta}_i))_{k\geq 1}$ is independent of $(W_k(s^{\eta}_i))_{k\geq 1}$ and $(V_k(s^{\eta}_i))_{k\geq 1}\sim \mathcal{U}([0,1])^{\otimes \mathbb{N}}$. 
 \end{remark}

\subsection{The flow as a measure-valued process} \label{calcmeas}

In this subsection we study the flow $Y$ as a measure-valued process. The main result, Proposition \ref{calcmeasimgrecipflow}, shows that the Stieltjes measure $m_t$ (see Section \ref{model2}) is proportional to the Lebesgue measure on $D_t^c$, with proportionality constant $e^{S_t}$. This yields Proposition \ref{measimgrec}, a key ingredient in the proof of Theorem \ref{represwktviamut}. Before stating the main result, we introduce a subordinator $(L_t)_{t\geq 0}$ defined by 
\begin{align}
L_t := \int_{(0,t] \times (0,1)^2} r N(ds,dr,du). \label{defsublt}
\end{align}
By \cite[Thm. 19.3]{KenIti1999} and \eqref{integassumption} we see that $(L_t)_{t\geq 0}$ is well-defined. Recall the subordinator $(S_t)_{t\geq 0}$ defined in \eqref{defsubst}. By It{\^o}'s formula (see e.g. \cite[Thm. II.5.1]{ikedawatanabe}) we have almost surely, 
\begin{align}
\forall t \geq 0, \ e^{-S_t} & = 1 - \int_{(0,t] \times (0,1)^2} e^{-S_{s-}} r N(ds,dr,du) = 1 - \int_0^t e^{-S_{s-}} d L_s. \label{sdeexps}
\end{align}
Finally, recall that $\mathcal{B}([0,1])$ denotes the family of Borel sets in $[0,1]$. 
\begin{prop} \label{calcmeasimgrecipflow}
We have, 
\begin{align*}
\mathbb{P} \left ( \forall t \geq 0, \forall A \in \mathcal{B}([0,1]) \ \text{s.t.} \ A \subset D_t^c, \ m_t(A)=|A| e^{S_t} \right ) & = 1, \\
\mathbb{P} \left ( \forall t \geq 0, \forall A \in \mathcal{B}([0,1]) \ \text{s.t.} \ A \subset D_{t-}^c, \ m_{t-}(A)=|A| e^{S_{t-}} \right ) & = 1. 
\end{align*}
\end{prop}
The proof proceeds in three steps: we first establish an identity \eqref{ipp} for $y \in Q=[0,1] \cap \mathbb{Q}$ via integration by parts, then extend it to all $y \in [0,1]$ by density, and finally use it to identify $m_t$ on $D_t^c$ via a test function argument.
\begin{proof}
Fix $y \in Q$. Using integration by parts (see e.g. \cite[Thm. 4.4.13]{Apple09}), \eqref{sdeexps} and \eqref{defflowybysde}, we get that, almost surely for all $t \geq 0$, 
\begin{align}
e^{-S_t} Y_{0,t}(y) -y = & - \int_{(0,t] \times (0,1)^2} e^{-S_{s-}} r Y_{0,s-}(y) N(ds,dr,du) \nonumber \\
& + \int_{(0,t] \times (0,1)^2} e^{-S_{s-}} \left ( \mr_{r,u}(Y_{0,s-}(y))-Y_{0,s-}(y) \right ) N(ds,dr,du) \nonumber \\
& - \int_{(0,t] \times (0,1)^2} e^{-S_{s-}} r \left ( \mr_{r,u}(Y_{0,s-}(y))-Y_{0,s-}(y) \right ) N(ds,dr,du) \nonumber \\
= & \int_{(0,t] \times (0,1)^2} e^{-S_{s-}} \left ( (1-r)\mr_{r,u}(Y_{0,s-}(y))-Y_{0,s-}(y) \right ) N(ds,dr,du) \nonumber \\
= & - \int_{(0,t] \times (0,1)^2} e^{-S_{s-}} \left ( \int_0^{Y_{0,s-}(y)} \mathds{1}_{z\in I_{r,u}^{\mathrm{o}}} dz \right ) N(ds,dr,du), \label{ipp}
\end{align}
where we have used Lemma \ref{dermruymsy} from Appendix \ref{usefulestimsde} for the last equality. 

Let us now consider a realization in the probability one events given by Proposition \ref{propexistence} and Remark \ref{contfrom01tocadlag}, and in the probability one event where $(S_t)_{t \geq 0}$ and $(L_t)_{t \geq 0}$ are well-defined and c\`ad-l\`ag, where \eqref{sdeexps} holds true, and where \eqref{ipp} holds true for all $y \in Q$ and $t \geq 0$. Note that for any $y \in [0,1]$ and $(s,r,u) \in N$ we have 
\begin{align}
0 \leq e^{-S_{s-}} \int_0^{Y_{0,s-}(y)} \mathds{1}_{z\in I_{r,u}^{\mathrm{o}}} dz \leq r \label{majointegrand}
\end{align}
For any $t \geq 0$, $y \in [0,1]$ and $(y_n)_{n\geq 1}$ in $Q$ converging to $y$, Remark \ref{contfrom01tocadlag} and \eqref{majointegrand} allow to apply dominated convergence in the right-hand side of \eqref{ipp} (applied at $y_n$) while Proposition \ref{propexistence} yields convergence of the left-hand side. We thus get that, on the above probability one events, \eqref{ipp} holds true for all $y \in [0,1]$ and $t \geq 0$. 

We still fix a realization in the above mentioned events. Let $\varphi \in \mathcal{C}^{\infty}([0,1])$. Using the definition of $m_t$ and integration by parts for Stieltjes integrals we get that for all $t \geq 0$, 
\begin{align*}
e^{-S_t} \int_0^1 \varphi(x) m_t(dx) & = e^{-S_t} \int_0^1 \varphi(x) dY_{0,t}(x) = e^{-S_t} \varphi(1) - \int_0^1 \varphi'(x) e^{-S_t} Y_{0,t}(x) dx. 
\end{align*}
Plugging \eqref{ipp} into the above we get 
\begin{align}
e^{-S_t} \int_0^1 \varphi(x) m_t(dx) & = e^{-S_t} \varphi(1) - \int_0^1 \varphi'(x) x dx \nonumber \\
& + \int_0^1 \int_{(0,t] \times (0,1)^2} \varphi'(x) e^{-S_{s-}} \left ( \int_0^{Y_{0,s-}(x)} \mathds{1}_{z\in I_{r,u}^{\mathrm{o}}} dz \right ) N(ds,dr,du) dx. \label{calmtipp}
\end{align}
Using \eqref{majointegrand} we get 
\begin{align*}
& \int_0^1 \int_{(0,t] \times (0,1)^2} \left | \varphi'(x) e^{-S_{s-}} \left ( \int_0^{Y_{0,s-}(x)} \mathds{1}_{z\in I_{r,u}^{\mathrm{o}}} dz \right ) \right | N(ds,dr,du) dx \leq L_t \int_0^1 \left | \varphi'(x) \right | dx < \infty, 
\end{align*}
where $(L_t)_{t\geq 0}$ is defined in \eqref{defsublt}. We can thus use Fubini's theorem for the last term in \eqref{calmtipp}. Using that along with integration by parts we get that the last term in \eqref{calmtipp} equals 
\begin{align*}
&\int_{(0,t] \times (0,1)^2} e^{-S_{s-}} \left ( \int_0^1 \varphi'(x) \left ( \int_0^{Y_{0,s-}(x)} \mathds{1}_{z\in I_{r,u}^{\mathrm{o}}} dz \right ) dx \right ) N(ds,dr,du) \\
= &\int_{(0,t] \times (0,1)^2} e^{-S_{s-}} \left ( \varphi(1)r - \int_0^1 \varphi(x) \mathds{1}_{Y_{0,s-}(x)\in I_{r,u}^{\mathrm{o}}} dY_{0,s-}(x) \right ) N(ds,dr,du) \\
= & \varphi(1) (1-e^{-S_t}) - \int_{(0,t] \times (0,1)^2} e^{-S_{s-}} \left ( \int_0^1 \varphi(x) \mathds{1}_{Y_{0,s-}(x)\in I_{r,u}^{\mathrm{o}}} dY_{0,s-}(x) \right ) N(ds,dr,du) \\
\end{align*}
where we have used \eqref{sdeexps}. Plugging this in \eqref{calmtipp} and using that $\varphi(1)-\int_0^1 \varphi'(x) x dx=\int_0^1 \varphi(x) dx$ we get 
\begin{align}
& e^{-S_t} \int_0^1 \varphi(x) m_t(dx) + \int_{(0,t] \times (0,1)^2} e^{-S_{s-}} \left ( \int_0^1 \varphi(x) \mathds{1}_{Y_{0,s-}(x)\in I_{r,u}^{\mathrm{o}}} dY_{0,s-}(x) \right ) N(ds,dr,du) \nonumber \\
& \qquad \qquad \qquad \qquad \qquad \qquad \qquad \qquad \qquad \qquad \qquad \qquad \qquad \qquad \qquad = \int_0^1 \varphi(x) dx. \label{egposmeas}
\end{align}
Each of the two sides of \eqref{egposmeas} is the integral of $\varphi$ with respect to a finite positive measure. Since the above holds for all $t \geq 0$ and $\varphi \in \mathcal{C}^{\infty}([0,1])$, we get that the underlying positive measures on $[0,1]$ are equal for all $t \geq 0$. It remains to restrict this equality of measures to $D_t^c$. Note that for all $(s,r,u)\in N$ with $s \in (0,t]$ and $x \in D_t^c$ we have $x\notin Y_{0,s-}^{-1}(I_{r,u}^{\mathrm{o}})$ so the second measure in the left-hand side of \eqref{egposmeas} does not charge $D_t^c$. Using this together with the equality of measures implied by \eqref{egposmeas} we get that, for any $t \geq 0$ and any Borel set $A \subset D_t^c$, $e^{-S_t} m_t(A) = |A|$. The same reasoning works with $t$ replaced by $t-$ so the result follows. 
\end{proof}

An important consequence of Proposition \ref{calcmeasimgrecipflow} is the following: 
\begin{prop} \label{measimgrec}
We have, 
\begin{align}
\mathbb{P} \left ( \forall t \geq 0, \forall A \in \mathcal{B}([0,1]), \ |Y_{0,t}^{-1}(A) \cap D_t^c| = e^{-S_t} |A|, |Y_{0,t-}^{-1}(A) \cap D_{t-}^c| = e^{-S_{t-}} |A| \right ) = 1. \label{measimgrec0}
\end{align}
\end{prop}

\begin{proof}
We fix a realization in the probability one events given by Propositions \ref{poissrepstep1} and \ref{calcmeasimgrecipflow}. Let $t \geq 0$ and $A \subset [0,1]$ be a Borel set. Note that $m_t(D_t)=0$ for all $t \geq 0$ by Proposition \ref{poissrepstep1} and the definition of $C_t$. Using Proposition \ref{calcmeasimgrecipflow}, $m_t(D_t)=0$, the definition of $m_t$ and the change of variable $y=Y_{0,t}(x)$, we get 
\begin{align*}
|Y_{0,t}^{-1}(A) \cap D_t^c| & = e^{-S_t} m_t(Y_{0,t}^{-1}(A) \cap D_t^c) = e^{-S_t} m_t(Y_{0,t}^{-1}(A)) \\
& = e^{-S_t} \int_0^1 \mathds{1}_{Y_{0,t}(x) \in A} dY_{0,t}(x) = e^{-S_t} \int_0^1 \mathds{1}_{y \in A} dy = e^{-S_t} |A|. 
\end{align*}
The same reasoning works with $t$ replaced by $t-$ so the result follows. 
\end{proof}

Applying Proposition \ref{measimgrec} with $A=[0,1]$, together with Proposition \ref{poissrepstep1}, we obtain the following corollary. 

\begin{cor} \label{calczonecollage}
We have almost surely $|C_t|=|D_t|=1-e^{-S_t}$ and $|C_{t-}|=|D_{t-}|=1-e^{-S_{t-}}$ for all $t \geq 0$. 
\end{cor}

\begin{remark} \label{sumwkteq1minusst}
Corollary \ref{calczonecollage} allows in particular to recover the classical fact, mentioned in Section \ref{model3}, that almost surely 
$\sum_{k\geq 1}W_k(t)=1-e^{-S_t}$ for all $t \geq 0$.
\end{remark}

\begin{remark} \label{zkindepst}
Combining Remark \ref{sumwkteq1minusst} with Lemma \ref{lawofvkandzk}, we deduce that for any $t \geq 0$ the sequence $(Z_k(t,r,u))_{k\geq 1}$ is independent of $S_t$.
\end{remark}

\subsection{Poisson representation: Proof of Theorem \ref{represwktviamut}}

In this subsection we prove Theorem \ref{represwktviamut}. The key idea is to group the jumps of $N$ according to which connected component of $C_t$ they contribute to and then to use Proposition \ref{measimgrec} to compute the total weight contributed by each group.

\begin{proof}[Proof of Theorem \ref{represwktviamut}]
We only prove \eqref{defmut} as the proof of \eqref{represwktviamut1} is similar (and uses that $N$ has countably many jumps). We know from Proposition \ref{poissrepstep1} that $C_t=D_t=\cup_{(s,r,u) \in N, s \in (0,t]} Y_{0,s-}^{-1}(I_{r,u}^{\mathrm{o}})$. However, the open intervals appearing in this countable union are not disjoint as two such intervals can be included in one-another. Let us assume that we are on the probability one events from \eqref{collage} and Propositions \ref{propexistence} and \ref{poissrepstep1}. In order to separate the open connected components of $C_t$ we define an equivalence relation $\simeq$ on $\{ (s,r,u) \in N \ \text{s.t.} \ s \in (0,t] \}$ by writing $(s_1,r_1,u_1)\simeq(s_2,r_2,u_2)$ if and only if $Y_{0,s_1-}^{-1}(I_{r_1,u_1}^{\mathrm{o}})$ and $Y_{0,s_2-}^{-1}(I_{r_2,u_2}^{\mathrm{o}})$ are in the same open connected component of $D_t=C_t$. 

For $j\geq 1$ such that $\mathcal{O}_j(t) \neq \emptyset$ let us denote by $\mathcal{C}_j(t)$ the equivalence class for $\simeq$ that is associated to the open connected component $\mathcal{O}_j(t)$ of $C_t=D_t$ (and let $\mathcal{C}_j(t):=\emptyset$ otherwise). We see from \eqref{collage} that $V_j(t)$, the value taken by $Y_{0,t}(\cdot)$ on $\mathcal{O}_j(t)$, is given by $Y_{s,t}(u)$ for any choice of $(s,r,u) \in \mathcal{C}_j(t)$. Let us denote by $\nu_t$ the measure defined by the right-hand side of \eqref{defmut}. We thus get 
\begin{align}
\nu_t = \sum_{j\geq 1} \left ( \sum_{(s,r,u) \in \mathcal{C}_j(t)} re^{-S_{s-}} \right ) \delta_{V_j(t)}. \label{represwktviamut4}
\end{align}

We are thus left to prove that $\sum_{(s,r,u) \in \mathcal{C}_j(t)} re^{-S_{s-}}=W_j(t)$ for all $j \geq 1$. The case $\mathcal{O}_j(t) = \emptyset$ is trivial so we only consider $j \geq 1$ such that $\mathcal{O}_j(t) \neq \emptyset$. For this we further assume that we are on the probability one event from Lemma \ref{piecewisecte} (see Appendix \ref{poisrep1ststepappend}), denote by $\mathcal{U}\subset [0,1]$ the set of measure one produced by Lemma \ref{piecewisecte}, and show that 
\begin{align}
\mathcal{U} \cap \left ( \cup_{(s,r,u)\in \mathcal{C}_j(t)} Y_{0,s-}^{-1}(I_{r,u}^{\mathrm{o}}) \right ) \subset \cup_{(s,r,u)\in \mathcal{C}_j(t)} \left ( Y_{0,s-}^{-1}(I_{r,u}^{\mathrm{o}}) \cap D_{s-}^c \right ) \subset \cup_{(s,r,u)\in \mathcal{C}_j(t)} Y_{0,s-}^{-1}(I_{r,u}^{\mathrm{o}}). \label{represwktviamut5}
\end{align}
The second inclusion of \eqref{represwktviamut5} is trivial so we only prove the first one. The idea is to show that each $z \in \mathcal{U}$ that belongs to some $Y_{0,s-}^{-1}(I_{r,u}^{\mathrm{o}})$ with $(s,r,u) \in \mathcal{C}_j(t)$ must in fact belong to $Y_{0,\hat s-}^{-1}(I_{\hat r,\hat u}^{\mathrm{o}}) \cap D_{\hat s-}^c$ for the earliest jump $(\hat s, \hat r, \hat u)$ that affects $z$. Let $z$ belong to the set in the left-hand side of \eqref{represwktviamut5}. By definition there exists $(\tilde s,\tilde r,\tilde u)\in \mathcal{C}_j(t)$ such that $z\in Y_{0,\tilde s-}^{-1}(I_{\tilde r,\tilde u}^{\mathrm{o}})$. By Lemma \ref{piecewisecte}, there is $(\hat s,\hat r,\hat u)\in N$ such that $\hat s \in (0,\tilde s]$ is the smallest increase time of the process $(J_{s}(z))_{s \geq 0}$. By Lemma \ref{piecewisecte} we have $z\in Y_{0,\hat s-}^{-1}(I_{\hat r,\hat u}^{\mathrm{o}})$ and $z \notin Y^{-1}_{0,s-}(I_{r,u}^{\mathrm{o}})$ for any $(s,r,u) \in N$ such that $s\in(0,\hat s)$ so $z \in D_{\hat s-}^c$. We thus get $z\in Y_{0,\hat s-}^{-1}(I_{\hat r,\hat u}^{\mathrm{o}}) \cap D_{\hat s-}^c$. The intersection $Y_{0,\tilde s-}^{-1}(I_{\tilde r,\tilde u}^{\mathrm{o}})\cap Y_{0,\hat s-}^{-1}(I_{\hat r,\hat u}^{\mathrm{o}})$ is non-empty as it contains $z$ so $(\tilde s,\tilde r,\tilde u)\simeq(\hat s,\hat r,\hat u)$. In particular $(\hat s,\hat r,\hat u)\in \mathcal{C}_j(t)$ so $z \in \cup_{(s,r,u)\in \mathcal{C}_j(t)} (Y_{0,s-}^{-1}(I_{r,u}^{\mathrm{o}}) \cap D_{s-}^c)$. This proves \eqref{represwktviamut5}. 
Then, for $(s_1,r_1,u_1), (s_2,r_2,u_2)\in \mathcal{C}_j(t)$ with $s_1<s_2$, we have 
\begin{align*}
Y_{0,s_1-}^{-1}(I_{r_1,u_1}^{\mathrm{o}}) \cap D_{s_1-}^c \subset Y_{0,s_1-}^{-1}(I_{r_1,u_1}^{\mathrm{o}}) \subset D_{s_1} \subset D_{s_2-} \subset \left ( Y_{0,s_2-}^{-1}(I_{r_2,u_2}^{\mathrm{o}}) \cap D_{s_2-}^c \right )^c. 
\end{align*}
Therefore, for any $(s_1,r_1,u_1), (s_2,r_2,u_2)\in \mathcal{C}_j(t)$, 
\begin{align}
s_1 \neq s_2 \Rightarrow \left ( Y_{0,s_1-}^{-1}(I_{r_1,u_1}^{\mathrm{o}}) \cap D_{s_1-}^c \right ) \cap \left ( Y_{0,s_2-}^{-1}(I_{r_2,u_2}^{\mathrm{o}}) \cap D_{s_2-}^c \right ) = \emptyset. \label{represwktviamut6}
\end{align}

Let us now further assume that we are on the probability one event from Proposition \ref{measimgrec}. By definition of $\mathcal{C}_j(t)$ we have $\mathcal{O}_j(t)=\cup_{(s,r,u)\in \mathcal{C}_j(t)} Y_{0,s-}^{-1}(I_{r,u}^{\mathrm{o}})$ so, using \eqref{represwktviamut5}, \eqref{represwktviamut6}, and Proposition \ref{measimgrec} we get that, for our fixed realization of the process, 
\begin{align}
W_j(t) & = |\mathcal{O}_j(t)| = \left | \cup_{(s,r,u)\in \mathcal{C}_j(t)} Y_{0,s-}^{-1}(I_{r,u}^{\mathrm{o}}) \right | = \left | \cup_{(s,r,u)\in \mathcal{C}_j(t)} \left ( Y_{0,s-}^{-1}(I_{r,u}^{\mathrm{o}}) \cap D_{s-}^c \right ) \right | \nonumber \\
& = \sum_{(s,r,u) \in \mathcal{C}_j(t)} \left | Y_{0,s-}^{-1}(I_{r,u}^{\mathrm{o}}) \cap D_{s-}^c \right | = \sum_{(s,r,u) \in \mathcal{C}_j(t)} e^{-S_{s-}} |I_{r,u}^{\mathrm{o}}| = \sum_{(s,r,u) \in \mathcal{C}_j(t)} e^{-S_{s-}} r. \label{represwktviamut7}
\end{align}
Then, the combination of \eqref{represwktviamut4} with \eqref{represwktviamut7} and \eqref{represwktviamut2} yields \eqref{defmut}, which concludes the proof. 
\end{proof}

\section{Stochastic integral representation for $W_k(t)$} \label{itoform}

In this section we prove Theorems \ref{intsto} and \ref{ito}. A key step is Lemma \ref{wksfromwksms}, which identifies the increment of $W_k(\cdot)$ at each jump of $N$ explicitly. Theorem \ref{intsto} is then proved by showing that the contribution of small jumps vanishes, and Theorem \ref{ito} by applying It{\^o}'s formula to Theorem \ref{intsto}. Recall that we always assume \eqref{integassumption}. 

\subsection{First step: behavior of $W_k(t)$ at a jump} \label{behavjump}

We start with the following key lemma. 
\begin{lemma} \label{wksfromwksms}
We have almost surely that, for any $(s,r,u)\in N$, $W_k(s)=W_k(s-)+K^k_{s-}(r,u)$, $M_k(s)=M_k(s-)+H^k_{s-}(r,u)$, and $H^k_{s-}(r,u)=\sum_{j=1}^k K^j_{s-}(r,u)$, where $K^k_{\cdot}(\cdot,\cdot)$ and $H^k_{\cdot}(\cdot,\cdot)$ are defined in \eqref{prepCFlb} and \eqref{prepCF}. 
\end{lemma}

\begin{proof}
Let $(s,r,u)\in N$. Combining \eqref{represwktviamut2} with \eqref{represwktviamut1} from Theorem \ref{represwktviamut} and using \eqref{defflowybysdetspsauts} (see Appendix \ref{shiftedflows}) and the definition of $Z_j(s-,r,u)$ in Section \ref{model3} we get $\mu_{s-} = \sum_{j\geq 1} W_j(s-) \delta_{V_j(s-)}$ and 
\begin{align}
\sum_{j\geq 1} W_j(s) \delta_{V_j(s)} & = \mu_{s} = re^{-S_{s-}} \delta_{u} + \sum_{(\tilde s, \tilde r, \tilde u) \in N, \tilde s \in (0,s)} \tilde re^{-S_{\tilde s-}} \delta_{\mr_{r,u}(Y_{\tilde s,s-}(\tilde u))} \nonumber \\ 
& = re^{-S_{s-}} \delta_{u} + \sum_{j\geq 1} W_j(s-) \delta_{\mr_{r,u}(V_j(s-))} \nonumber \\ 
& = \sum_{j\geq 1, Z_j(s-,r,u)=0} W_j(s-) \delta_{\mr_{r,u}(V_j(s-))} + \left ( e^{-S_{s-}}r + \sum_{j\geq 1} Z_j(s-,r,u) W_j(s-) \right ) \delta_{u}. \label{exprmus}
\end{align}
Note that all the Dirac measures appearing in the above expression are distinct, since the $V_j(s-)$ are distinct and $\mr_{r,u}$ is injective on $I_{r,u}^c$, where $I_{r,u}$ is defined in \eqref{defiru}. 

\textit{We first assume that we are on the event $\{\beta_k(s-,r,u)=0\}$.} We thus have $Z_j(s-,r,u)=0$ for all $j\in\{1,\ldots,k\}$ so all terms $W_j(s-) \delta_{\mr_{r,u}(V_j(s-))}$ for $j\in\{1,\ldots,k\}$ appear in the sum $\sum_{j\geq 1, Z_j(s-,r,u)=0}\cdots$ from \eqref{exprmus}. We distinguish three sub-cases according to the size of the factor of $\delta_u$ in \eqref{exprmus} relative to $W_k(s-)$ and $W_{k-1}(s-)$. If that factor is in $(0,W_k(s-))$ then the $k$ largest factors appearing in \eqref{exprmus} are $W_1(s-),\ldots,W_k(s-)$. We thus have $W_k(s)-W_k(s-)=0$ and $M_k(s)-M_k(s-)=0$ in that case, which agrees with the expressions of $K^k_{s-}(r,u)$ and $H^k_{s-}(r,u)$ (see \eqref{prepCFlb} and \eqref{prepCF}). If the factor of $\delta_u$ in \eqref{exprmus} is in $[W_k(s-),W_{k-1}(s-))$ then this factor is the $k^{th}$ largest factor appearing in \eqref{exprmus} (while the first $k-1$ are $W_1(s-),\ldots,W_{k-1}(s-)$). In this case we thus have $W_k(s)-W_k(s-)=M_k(s)-M_k(s-)=(\text{factor of } \delta_u)-W_k(s-)$, which agrees with the expressions of $K^k_{s-}(r,u)$ and $H^k_{s-}(r,u)$. Finally, if the factor of $\delta_u$ in \eqref{exprmus} is in $[W_{k-1}(s-),1]$ then this factor is one of the $(k-1)^{th}$ largest factor appearing in \eqref{exprmus} and the $k^{th}$ is $W_{k-1}(s-)$. In this case we thus have $W_k(s)-W_k(s-)=W_{k-1}(s-)-W_k(s-)$ and that $M_k(s)-M_k(s-)$ is as in the previous case. This agrees with the expressions of $K^k_{s-}(r,u)$ and $H^k_{s-}(r,u)$. 

\textit{We now assume that we are on the event $\{\beta_k(s-,r,u)\neq 0\}$ and study $M_k(s)$.} In this case a number equal to $\beta_k(s-,r,u)$ of the terms $W_j(s-)$, for $j\in\{1,\ldots,k\}$, appears in the factor of $\delta_u$ in \eqref{exprmus} so this factor is one of the $k$ largest factors appearing in \eqref{exprmus}. $M_k(s)$ is thus equal to the factor of $\delta_u$ in \eqref{exprmus} plus the sum of the $k-\beta_k(s-,r,u)$ terms $W_j(s-)$, for $j\in\{1,\ldots,k\}$ such that $Z_j(s-,r,u)=0$, plus the $\beta_k(s-,r,u)-1$ largest terms $W_j(s-)$, for $j>k$ such that $Z_j(s-,r,u)=0$. This agrees with the expression of $M_k(s-)+H^k_{s-}(r,u)$ and completes the proof of $M_k(s)=M_k(s-)+H^k_{s-}(r,u)$. 

\textit{We now assume that we are on the event $\{\beta_k(s-,r,u)=1\}$ and study $W_k(s)$.} Let us denote by $J$ the unique $j\in\{1,\ldots,k\}$ that is such that  $Z_j(s-,r,u)=1$. If $J\neq k$ then the $k-1$ largest factors appearing in \eqref{exprmus} are the factor of $\delta_u$ in \eqref{exprmus} and the terms $W_j(s-)$ for $j\in\{1,\ldots,k-1\}\setminus\{J\}$ and $W_k(s-)$ is the $k^{th}$ largest factor appearing in \eqref{exprmus} so $W_k(s)=W_k(s-)$. If $J=k$ then the factor of $\delta_u$ in \eqref{exprmus} contains $W_k(s-)$ but no other $W_j(s-)$ for $j\in\{1,\ldots,k-1\}$. If the factor of $\delta_u$ in \eqref{exprmus} is smaller than $W_{k-1}(s-)$ then it is the $k^{th}$ largest factor appearing in \eqref{exprmus} and is thus equal to $W_k(s)$, if it is larger than $W_{k-1}(s-)$ then $W_{k-1}(s-)$ is the $k^{th}$ largest factor appearing in \eqref{exprmus} so $W_k(s)=W_{k-1}(s-)$. In all cases the obtained expression of $W_k(s)$ agrees with the expression of $W_k(s-)+K^k_{s-}(r,u)$. 

\textit{We now assume that we are on the event $\{\beta_k(s-,r,u)\geq 2\}$ and study $W_k(s)$.} In this case a number equal to $\beta_k(s-,r,u)\geq 2$ of the terms $W_j(s-)$, for $j\in\{1,\ldots,k\}$, appears in the factor of $\delta_u$ in \eqref{exprmus} so this factor is one of the $k-1$ largest factors appearing in \eqref{exprmus}. The $k-\beta_k(s-,r,u)$ terms $W_j(s-)$, for $j\in\{1,\ldots,k\}$, such that $Z_j(s-,r,u)=0$ are all part of the $k-1$ largest factors appearing in \eqref{exprmus}, and so are the $\beta_k(s-,r,u)-2$ largest terms $W_j(s-)$, for $j>k$, such that $Z_j(s-,r,u)=0$. $W_k(s)$ is thus equal to the $\beta_k(s-,r,u)-1$ largest terms $W_j(s-)$, for $j>k$, such that $Z_j(s-,r,u)=0$. This agrees with the expression of $W_k(s-)+K^k_{s-}(r,u)$ and completes the proof of $W_k(s)=W_k(s-)+K^k_{s-}(r,u)$. 

Finally, the third claim $H^k_{s-}(r,u)=\sum_{j=1}^k K^k_{s-}(r,u)$ is a consequence of the first two claims and of the fact that $M_k(\cdot)=\sum_{1 \leq j \leq k} W_j(\cdot)$ by definition. 
\end{proof}

\begin{remark} \label{h=sumknormaltime}
The proof of Lemma \ref{wksfromwksms} only uses the pathwise structure of $N$ and $Y$. Therefore, fixing $t \geq 0$ and $(r,u) \in (0,1)^2$, and adding artificially the extra jump $(t,r,u)$ to the process, we see that, for all $t \geq 0$ and $(r,u) \in (0,1)^2$, we have almost surely $H^k_{t}(r,u)=\sum_{j=1}^k K^j_{t}(r,u)$. 
\end{remark}

\subsection{Stochastic integral representation for $W_k(t)$: Proof of Theorem \ref{intsto}}

In this subsection we prove Theorem \ref{intsto}. We start with two preparatory lemmas establishing integrability conditions on the integrand $K^k_{s-}(r,u)$.
\begin{lemma} \label{itoprep}
For any $k \geq 1$ there is a constant $C_k$ such that for any $\eta\in(0,1]$, and $t \geq 0$, 
\begin{align}
& \mathbb{E} \left [ \int_{(0,t] \times (0,\eta] \times (0,1)} \left | K^k_{s-}(r,u) \right | N(\dd s, \dd r, \dd u) \right ] \leq t C_k \int_{(0,\eta]} r^{-1} \Lambda(dr)<\infty. \label{prepsmalljumps}
\end{align}
In particular we have almost surely for any $k \geq 1$, $\eta\in(0,1]$ and $t\geq 0$, 
\begin{align}
\int_{(0,t] \times (0,\eta] \times (0,1)} \left | K^k_{s-}(r,u) \right | N(\dd s, \dd r, \dd u) < \infty. \label{itoforwkprep}
\end{align}
\end{lemma}

\begin{proof}
We fix $k \geq 1$ and denote by $A^{\eta}_{f,k}(t)$ the quantity in \eqref{itoforwkprep}. \eqref{itoforwkprep} follows easily from \eqref{prepsmalljumps} and from the fact that $A^{\eta}_{f,k}(t)$ is almost surely non-decreasing in $t$ and $\eta$, so we only prove \eqref{prepsmalljumps}. By the compensation formula we get 
\begin{align}
\mathbb{E} \left [ A^{\eta}_{f,k}(t) \right ] = \int_{(0,t]} \left ( \int_{(0,\eta] \times (0,1)} \mathbb{E} \left [ \left | K^k_{s}(r,u) \right | \right ] r^{-2} \Lambda(dr) du \right ) ds. \label{espafk}
\end{align}
Using the expression of $K^k_{\cdot}(\cdot,\cdot)$ in \eqref{prepCFlb} we can see that 
\begin{align*}
\left | K^k_{s}(r,u) \right | \leq \left ( e^{-S_{s}}r+\sum_{j>k}Z_j(s,r,u) W_j(s) \right ) + \mathds{1}_{\beta_k(s,r,u)\geq 2}, 
\end{align*}
so, taking expectation and using Lemma \ref{lawofvkandzk} and Remark \ref{sumwkteq1minusst}, we get 
\begin{align*}
\mathbb{E} \left [ \left | K^k_{s}(r,u) \right | \right ] \leq r \mathbb{E} \left [ 1-\sum_{j=1}^k W_j(s) \right ] + (1-(1-r)^k-kr(1-r)^{k-1}) \leq C_k r, 
\end{align*}
for some constant $C_k>0$. Plugging the above into \eqref{espafk} we obtain \eqref{prepsmalljumps}, where the finiteness of this upper bound comes from \eqref{integassumption}. This concludes the proof. 
\end{proof}

\begin{lemma} \label{prepsmalljumpsbis}
For a Lipschitz function $f:[0,1]\rightarrow\mathbb{R}$, an integer $k\geq 1$, $t\geq 0$ and $\eta \in (0,1]$ we have 
\begin{align}
& \mathbb{E} \left [ \int_{(0,\eta] \times (0,1)} \left | f \left (W_k(t)+K^k_{t}(r,u) \right ) - f \left (W_k(t) \right ) \right | r^{-2} \Lambda(dr) du \right ] \leq C_{f,k} \int_{(0,\eta]} r^{-1} \Lambda(dr), \label{prepsmalljumpsbiseqwkt} \\
& \mathbb{E} \left [ \int_{(0,\eta] \times (0,1)} \left | f \left (M_k(t)+H^k_{t}(r,u) \right ) - f \left (M_k(t) \right ) \right | r^{-2} \Lambda(dr) du \right ] \leq \tilde C_{f,k} \int_{(0,\eta]} r^{-1} \Lambda(dr), \label{prepsmalljumpsbiseqmkt}
\end{align}
for some constant $C_{f,k}>0$ and $\tilde C_{f,k} := \sum_{j=1}^k C_{f,j}$. 
\end{lemma}
The proof uses Remark \ref{h=sumknormaltime} and the same argument as for Lemma \ref{itoprep} so we omit it. 

\begin{proof}[Proof of Theorem \ref{intsto}]
Let us fix $t \geq 0$ and $k \geq 1$. The set $\{ (s,r,u) \in N \ \text{s.t.} \ r>\eta \}$ is almost surely discrete for all $\eta \in (0,1)$. Let us enumerate its elements by $(s^{\eta}_i,r^{\eta}_i,u^{\eta}_i)$ where $s^{\eta}_1<s^{\eta}_2<\ldots$ and for convenience set $s^{\eta}_0:=0$. The strategy is as follows. By Lemma \ref{wksfromwksms}, the increments of $W_k$ at the large jump times $s^\eta_i$ are given by $K^k_{s^\eta_i-}(r^\eta_i, u^\eta_i)$. It thus remains to show that the cumulative variation of $W_k(\cdot)$ \textit{between} consecutive large jumps, which accounts for the sole contribution of small jumps ($r \leq \eta$) and which we denote by
\begin{align*}
\Sigma_{\eta}(t) := \sum_{i \geq 1} \left (W_k((s^{\eta}_{i+1}\wedge t)-) - W_k(s^{\eta}_i) \right ) \mathds{1}_{s^{\eta}_i<t},
\end{align*}
is negligible as $\eta \to 0$. 

\textit{Step 1: Control of $\mathbb{E}[|\Sigma_\eta(t)|]$.} For $w \geq 0$, we define 
\begin{align*}
D^{s^{\eta}_i}_{s^{\eta}_i+w}:=\cup_{(s,r,u) \in N, s \in (s^{\eta}_i,s^{\eta}_i+w]} Y_{s^{\eta}_i,s-}^{-1}(I_{r,u}^{\mathrm{o}}), 
\end{align*}
which is the analogue of the sets $D_t$ from \eqref{cadlagsets}, with the flow $Y_{0,\cdot}(\cdot)$ replaced by the shifted flow $Y_{s^{\eta}_i,s^{\eta}_i+\cdot}(\cdot)$ from Appendix \ref{shiftedflows}. We thus get  
\begin{align}
(D^{s^{\eta}_i}_{s^{\eta}_i+w}, S_{s^{\eta}_i+w}-S_{s^{\eta}_i})_{w \geq 0} \overset{(d)}{=} (D_{w}, S_{w})_{w \geq 0}, \qquad (D^{s^{\eta}_i}_{s^{\eta}_i+w})_{w \geq 0} \perp\!\!\!\perp \mathcal{F}_{s^{\eta}_i}. \label{markovdt}
\end{align}
By Corollary \ref{calczonecollage} we deduce that, on $\{ s^{\eta}_i<t \}$, we have almost surely 
\begin{align}
|D^{s^{\eta}_i}_{(s^{\eta}_{i+1} \wedge t)-}| = 1-e^{-(S_{(s^{\eta}_{i+1} \wedge t)-}-S_{s^{\eta}_i})} & \leq S_{(s^{\eta}_{i+1} \wedge t)-}-S_{s^{\eta}_i} \nonumber \\
& = \int_{(s^{\eta}_i,s^{\eta}_{i+1} \wedge t) \times (0,\eta] \times (0,1)} |\log(1-r)| N(ds,dr,du). \label{estimdsi}
\end{align}
To bound $\Sigma_\eta(t)$, we distinguish whether the $k$ largest blocks from time $s^\eta_i$ are affected by a small jump in the interval $(s^\eta_i, s^\eta_{i+1}\wedge t)$ or not. Note from \eqref{collage} applied to the flow $Y_{s^\eta_i, s^\eta_i+\cdot}(\cdot)$ that such a jump $(s,r,u)\in N$ affects the $j^{th}$ block if and only if $V_j(s^\eta_i) \in D^{s^{\eta}_i}_{(s^{\eta}_{i+1}\wedge t)-}$. We thus define 
\begin{align}
\mathcal{E}_i(\eta):= \{ s^{\eta}_i\geq t \} \cup \{ s^{\eta}_i<t,\ V_1(s^{\eta}_i),\ldots, V_k(s^{\eta}_i) \notin D^{s^{\eta}_i}_{(s^{\eta}_{i+1}\wedge t)-} \},
\end{align}
which is the event that none of these blocks is affected by a small jump before the next large jump. In the following, we bound $\sum_{i\geq 1}\mathbb{P}(\mathcal{E}_i(\eta)^c)$ and control $W_k((s^\eta_{i+1}\wedge t)-) - W_k(s^\eta_i)$ on $\mathcal{E}_i(\eta)$. 

\textit{Substep 1a: Bounding $\sum_{i\geq 1}\mathbb{P}(\mathcal{E}_i(\eta)^c)$.} Using \eqref{markovdt}, Remark \ref{atjumptime}, and \eqref{estimdsi} we get 
\begin{align}
& \sum_{i \geq 1} \mathbb{P}(\mathcal{E}_i(\eta)^c) \leq \sum_{i \geq 1} \sum_{j=1}^k \mathbb{P} \left ( s^{\eta}_i<t, V_j(s^{\eta}_i) \in D^{s^{\eta}_i}_{(s^{\eta}_{i+1}\wedge t)-} \right ) = \sum_{i \geq 1} k \mathbb{E} \left [ |D^{s^{\eta}_i}_{(s^{\eta}_{i+1}\wedge t)-}| \mathds{1}_{s^{\eta}_i<t} \right ] \nonumber \\
\leq & k \mathbb{E} \left [ \int_{(0,t] \times (0,\eta] \times (0,1)} |\log(1-r)| N(ds,dr,du) \right ] = k t \int_{(0,\eta]} |\log(1-r)| r^{-2}\Lambda(dr). \label{majosumpeieta}
\end{align}

\textit{Substep 1b: Representation of $\mu_{(s^{\eta}_{i+1}\wedge t)-}$.} Let $E_i(j):= \{ (s,r,u) \in N \ \text{s.t.} \ s \in (0,s^{\eta}_i], Y_{s,s^{\eta}_i}(u)=V_j(s^{\eta}_i) \}$. By \eqref{represwktviamut1} we have $W_j(s^{\eta}_i)=\sum_{(s,r,u) \in E_i(j)} re^{-S_{s-}}$. Using Theorem \ref{represwktviamut} and \eqref{compprop} we get that, on $\{ s^{\eta}_i<t \}$, $\mu_{(s^{\eta}_{i+1}\wedge t)-}$ equals 
\begin{align*}
 & \sum_{(s,r,u) \in N, s \in (0,s^{\eta}_{i+1}\wedge t)} re^{-S_{s-}} \delta_{Y_{s,(s^{\eta}_{i+1}\wedge t)-}(u)} \\
= & \sum_{(s,r,u) \in N, s \in (s^{\eta}_i,s^{\eta}_{i+1}\wedge t)} re^{-S_{s-}} \delta_{Y_{s,(s^{\eta}_{i+1}\wedge t)-}(u)} + \sum_{j \geq 1} \sum_{(s,r,u) \in E_i(j)} re^{-S_{s-}} \delta_{Y_{s^{\eta}_i,(s^{\eta}_{i+1}\wedge t)-}(Y_{s,s^{\eta}_i}(u))} \\
= & \sum_{(s,r,u) \in N, s \in (s^{\eta}_i,s^{\eta}_{i+1}\wedge t)} re^{-S_{s-}} \delta_{Y_{s,(s^{\eta}_{i+1}\wedge t)-}(u)} + \sum_{j \geq 1} W_j(s^{\eta}_i) \delta_{Y_{s^{\eta}_i,(s^{\eta}_{i+1}\wedge t)-}(V_j(s^{\eta}_i))}. 
\end{align*}
Let $\mathcal{U}:=(\mathcal{S})^c$ (where $\mathcal{S}$ is defined in Lemma \ref{diffcountable}) and note that $|\mathcal{U}|=1$ by Lemma \ref{diffcountable}. For $(s,r,u) \in N$ such that $s \in (s^{\eta}_i,s^{\eta}_{i+1}\wedge t)$ we choose $x_s \in Y_{0,s-}^{-1}(I_{r,u}^{\mathrm{o}}) \cap D_{s-}^c \cap \mathcal{U}$. By Proposition \ref{measimgrec} the latter set is non-empty so such a choice of $x_s$ exists. Note from \eqref{collage} that $Y_{s,(s^{\eta}_{i+1}\wedge t)-}(u)=Y_{0,(s^{\eta}_{i+1}\wedge t)-}(x_s)$. Then, for any $j \geq 1$ such that $W_j(s^{\eta}_i)>0$ we choose $x_j \in \mathcal{O}_j(s^{\eta}_i) \cap \mathbb{Q}$ and note that $V_j(s^{\eta}_i)=Y_{0,s^{\eta}_i}(x_j)$. By this and \eqref{compprop} we get that, on $\{ s^{\eta}_i<t \}$, we have almost surely 
\begin{align}
\mu_{(s^{\eta}_{i+1}\wedge t)-} = \sum_{(s,r,u) \in N, s \in (s^{\eta}_i,s^{\eta}_{i+1}\wedge t)} re^{-S_{s-}} \delta_{Y_{0,(s^{\eta}_{i+1}\wedge t)-}(x_s)} + \sum_{j \geq 1} W_j(s^{\eta}_i) \delta_{Y_{0,(s^{\eta}_{i+1}\wedge t)-}(x_j)}. \label{exprmusiplus1}
\end{align}

\textit{Substep 1c: The $k$ largest blocks from time $s^{\eta}_i$ are stable on $\mathcal{E}_i(\eta)$.} On $\{ s^{\eta}_i<t \}$, let $\mathcal{L}_i:=\{ j \geq 1, W_j(s^{\eta}_i)>0 \}$ and $\mathcal{J}_i(s^{\eta}_{i+1}\wedge t):= \mathcal{L}_i \cap \{ j \geq 1, V_j(s^{\eta}_i) \notin D^{s^{\eta}_i}_{(s^{\eta}_{i+1} \wedge t)-} \}$. We note that on $\mathcal{E}_i(\eta) \cap \{ s^{\eta}_i<t \}$ we have $\{1,\ldots,k\} \subset \mathcal{J}_i(s^{\eta}_{i+1}\wedge t) \cup \mathcal{L}_i^c$. We now justify that, on $\{ s^{\eta}_i<t \}$, for any $j \in \mathcal{J}_i(s^{\eta}_{i+1}\wedge t)$, the atom of $\mu_{(s^{\eta}_{i+1}\wedge t)-}$ at $Y_{0,(s^{\eta}_{i+1}\wedge t)-}(x_j)$ is of size exactly $W_j(s^{\eta}_i)$. Since the values $V_j(s^{\eta}_i)$'s are distinct, we get from Remark \ref{nonzerodiffrmk1} and Lemma \ref{diffcountable} that, for any $j \in \mathcal{J}_i(s^{\eta}_{i+1}\wedge t)$ and $\ell \in \mathcal{L}_i \setminus \{j\}$, if we have $Y_{0,(s^{\eta}_{i+1}\wedge t)-}(x_{j})=Y_{0,(s^{\eta}_{i+1}\wedge t)-}(x_{\ell})$, then there is $(s,r,u) \in N$ such that $s\in (s^{\eta}_i,s^{\eta}_{i+1}\wedge t)$ and $x_{j}, x_{\ell} \in Y_{0,s-}^{-1}(I_{r,u}^{\mathrm{o}})$. This implies that $Y_{s^{\eta}_i,s-}(V_{j}(s^{\eta}_i))=Y_{0,s-}(x_{j}) \in I_{r,u}^{\mathrm{o}}$ so $V_{j}(s^{\eta}_i) \in Y_{s^{\eta}_i,s-}^{-1}(I_{r,u}^{\mathrm{o}})$, contradicting $j \in \mathcal{J}_i(s^{\eta}_{i+1}\wedge t)$. Similarly, by Proposition \ref{nonzerodiff} and Lemma and \ref{diffcountable}, if there is $(s,r,u) \in N$ with $s \in (s^{\eta}_i,s^{\eta}_{i+1}\wedge t)$ and $j \in \mathcal{J}_i(s^{\eta}_{i+1}\wedge t)$ such that $Y_{0,(s^{\eta}_{i+1}\wedge t)-}(x_s)=Y_{0,(s^{\eta}_{i+1}\wedge t)-}(x_{j})$, then there is $(\tilde s, \tilde r, \tilde u) \in N$ such that $\tilde s\in (0,s^{\eta}_{i+1}\wedge t)$ and $x_s, x_j \in Y_{0,\tilde s-}^{-1}(I_{\tilde r, \tilde u}^{\mathrm{o}})$. Since $x_s \in D_{s-}^c$ we have necessarily $\tilde s\in (s,s^{\eta}_{i+1}\wedge t) \subset (s^{\eta}_i,s^{\eta}_{i+1}\wedge t)$ and $Y_{s^{\eta}_i,\tilde s-}(V_j(s^{\eta}_i))=Y_{0,\tilde s-}(x_j) \in I_{\tilde r, \tilde u}^{\mathrm{o}}$ so $V_{j}(s^{\eta}_i) \in Y_{s^{\eta}_i,\tilde s-}^{-1}(I_{\tilde r,\tilde u}^{\mathrm{o}})$. Again, this contradicts $j \in \mathcal{J}_i(s^{\eta}_{i+1}\wedge t)$. 

\textit{Substep 1d: Bounding $W_k((s^\eta_{i+1}\wedge t)-)$ on $\mathcal{E}_i(\eta)$.} The Substep 1c and \eqref{exprmusiplus1} yield that, on $\mathcal{E}_i(\eta) \cap \{ s^{\eta}_i<t \} \cap \{ W_k(s^{\eta}_i)>0 \}$, $\mu_{(s^{\eta}_{i+1}\wedge t)-}$ has, for each $j \in \{1,\ldots,k\}$, an atom at $Y_{0,(s^{\eta}_{i+1}\wedge t)-}(x_j)$ with weight $W_j(s^{\eta}_i)$. In particular, $\mathcal{E}_i(\eta) \cap \{ s^{\eta}_i<t \} \subset \{W_k((s^{\eta}_{i+1}\wedge t)-) \geq W_k(s^{\eta}_i) \}$. On $\mathcal{E}_i(\eta) \cap \{ s^{\eta}_i<t \} \cap \{ W_k(s^{\eta}_i)>0 \}$, if additionally to these $k$ atoms, $\mu_{(s^{\eta}_{i+1}\wedge t)-}$ has another atom at a point $a \in [0,1]$ with weight strictly larger than $W_k(s^{\eta}_i)$, then the Substep 1c shows that for all $j \in \mathcal{L}_i$ such that $Y_{0,(s^{\eta}_{i+1}\wedge t)-}(x_j)=a$ we have $j \notin \mathcal{J}_i(s^{\eta}_{i+1}\wedge t)$ so $V_j(s^{\eta}_i) \in D^{s^{\eta}_i}_{(s^{\eta}_{i+1} \wedge t)-}$. Combining with \eqref{represwktviamut2} and \eqref{exprmusiplus1} we get that, on $\mathcal{E}_i(\eta) \cap \{ s^{\eta}_i<t \} \cap \{ W_k((s^{\eta}_{i+1}\wedge t)-)>W_k(s^{\eta}_i) > 0 \}$, we have almost surely 
\begin{align}
W_k((s^{\eta}_{i+1}\wedge t)-) \leq \int_{(s^{\eta}_i,s^{\eta}_{i+1} \wedge t) \times (0,\eta] \times (0,1)} r N(ds,dr,du) + \sum_{j>k} W_j(s^{\eta}_i) \mathds{1}_{V_j(s^{\eta}_i) \in D^{s^{\eta}_i}_{(s^{\eta}_{i+1} \wedge t)-}}. \label{majowksiplus1}
\end{align}
On $\{ s^{\eta}_i<t \} \cap \{ W_k(s^{\eta}_i)=0 \}$, we see from \eqref{exprmusiplus1} that we have almost surely 
\begin{align}
W_k((s^{\eta}_{i+1}\wedge t)-) \leq \int_{(s^{\eta}_i,s^{\eta}_{i+1} \wedge t) \times (0,\eta] \times (0,1)} r N(ds,dr,du). \label{majowksiplus2}
\end{align}

\textit{Conclusion of Step 1.} Using that \eqref{majowksiplus1} holds true on $\mathcal{E}_i(\eta) \cap \{ s^{\eta}_i<t \} \cap \{ W_k((s^{\eta}_{i+1}\wedge t)-)>W_k(s^{\eta}_i)>0 \}$ and that \eqref{majowksiplus2} holds true on $\{ s^{\eta}_i<t \} \cap \{ W_k(s^{\eta}_i)=0 \}$, we get that almost surely, 
\begin{align*}
|\Sigma_{\eta}(t)| \leq & \sum_{i \geq 0} |W_k((s^{\eta}_{i+1}\wedge t)-) - W_k(s^{\eta}_i)| \mathds{1}_{s^{\eta}_i<t} \\
\leq & \sum_{i \geq 1} \mathds{1}_{\mathcal{E}_i(\eta)^c} + \int_{(0,t) \times (0,\eta] \times (0,1)} r N(ds,dr,du) + \sum_{i \geq 1} \mathds{1}_{s^{\eta}_i<t} \sum_{j>k} W_j(s^{\eta}_i) \mathds{1}_{V_j(s^{\eta}_i) \in D^{s^{\eta}_i}_{(s^{\eta}_{i+1} \wedge t)-}}. 
\end{align*}
Taking the expectation and using the compensation formula, \eqref{markovdt}, Remark \ref{atjumptime}, \eqref{majosumpeieta}, \eqref{estimdsi} and that $\sum_{j>k} W_j(s^{\eta}_i) \leq 1$ almost surely, we get 
\begin{align}
\mathbb{E} [ |\Sigma_{\eta}(t)| ] \leq & \sum_{i \geq 1} \mathbb{P}(\mathcal{E}_i(\eta)^c) + t \int_{(0,\eta]} r^{-1}\Lambda(dr) + \sum_{i \geq 1} \sum_{j>k} \mathbb{E}[W_j(s^{\eta}_i) \mathds{1}_{s^{\eta}_i<t} |D^{s^{\eta}_i}_{(s^{\eta}_{i+1} \wedge t)-}| ] \nonumber \\
\leq & 3 k t \int_{(0,\eta]} |\log(1-r)| r^{-2}\Lambda(dr) \underset{\eta \rightarrow 0}{\longrightarrow} 0, \label{esigmacv0}
\end{align}
where the convergence toward $0$ comes from \eqref{integassumption}. 

\textit{Step 2: Identification of $W_k(t)$ with $B_{f,k}(t)$.} We denote by $B_{f,k}(t)$ the right-hand side of \eqref{itoforwk0}. Lemma \ref{itoprep} shows that $B_{f,k}(t)$ is well-defined almost surely. We note that, almost surely, $W_k(t)=W_k(t-)$ and $B_{f,k}(t)=B_{f,k}(t-)$ by respectively Proposition \ref{combrepres} and the fact that $t \notin J_N$ almost surely. Using this and Lemma \ref{wksfromwksms} we get that almost surely, 
\begin{align*}
W_k(t) & = W_k(t-) = \sum_{i \geq 1} \left ( W_k(s^{\eta}_i) - W_k(s^{\eta}_i-) \right ) \mathds{1}_{s^{\eta}_i<t} + \Sigma_{\eta}(t) \\
& = \sum_{i \geq 1} K^k_{s^{\eta}_i-}(r^{\eta}_i,u^{\eta}_i) \mathds{1}_{s^{\eta}_i<t} + \Sigma_{\eta}(t) = B_{f,k}(t-) + \Sigma_{\eta}(t) - I_{\eta}(t-) = B_{f,k}(t) + \Sigma_{\eta}(t) - I_{\eta}(t), 
\end{align*}
where we have set $I_{\eta}(t):=\int_{(0,t] \times (0,\eta] \times (0,1)} K^k_{s-}(r,u) N(\dd s, \dd r, \dd u)$. By \eqref{prepsmalljumps} from Lemma \ref{itoprep} we get $\mathbb{E} [ |I_{\eta}(t)| ] \rightarrow 0$ as $\eta \rightarrow 0$. Combining with \eqref{esigmacv0} we get that $\Sigma_{\eta}(t) - I_{\eta}(t)$ converges to $0$ in probability as $\eta$ goes to $0$ so $W_k(t-)=B_{f,k}(t)$ almost surely. This proves that, for every fixed $t \in [0,\infty) \cap \mathbb{Q}$, \eqref{itoforwk0} holds almost surely. Finally, since both sides of \eqref{itoforwk0} are c\`ad-l\`ag in $t$ almost surely --- by Proposition \ref{combrepres} for the left-hand side and by properties of Poisson integrals for the right-hand side --- \eqref{itoforwk0} holds almost surely for all $t \geq 0$ simultaneously. Then, recalling that $M_k(t)=\sum_{1 \leq j \leq k} W_j(t)$ by definition, applying \eqref{itoforwk0} to $W_1(t),\ldots,W_k(t)$ and summing, and using that $H^k_{s-}(r,u)=\sum_{1 \leq j \leq k} K^k_{s-}(r,u)$ by Lemma \ref{wksfromwksms}, we obtain \eqref{itoformk0}. 
\end{proof}

\subsection{Pseudo-generator formula for $W_k(t)$: Proof of Theorem \ref{ito}}
In this subsection we prove Theorem \ref{ito} by applying It{\^o}'s formula to the stochastic integral representation of Theorem \ref{intsto}.
\begin{proof}[Proof of Theorem \ref{ito}]
We apply It\^o's formula to \eqref{itoforwk0}, take expectations via the compensation formula, and differentiate. The main step is to show that the resulting integrand $G^f_0(\cdot)$ is continuous, which we achieve by approximating it uniformly by truncated versions $G^f_\epsilon(\cdot)$.

Let $f$ be as in the statement of the theorem and $k \geq 1$. Since $W_k(t)$ is given by an uncompensated Poisson stochastic integral \eqref{itoforwk0}, It\^o's formula from \cite[Thm. II.5.1]{ikedawatanabe} extends to Lipschitz functions. Applying it to $f$ and \eqref{itoforwk0}, we get that, almost surely, for all $t\geq 0$, 
\begin{align}
f(W_k(t)) = f(0) + \int_{(0,t] \times (0,1)^2} \left ( f (W_k(s-)+K^k_{s-}(r,u) ) - f (W_k(s-) ) \right ) N(\dd s, \dd r, \dd u). \label{itoforwk}
\end{align}
We note from Lemma \ref{prepsmalljumpsbis} that 
\begin{align}
& \mathbb{E} \left [ \int_{(0,t] \times (0,1)^2} \left | f (W_k(s-)+K^k_{s-}(r,u) ) - f (W_k(s-) ) \right | N(\dd s, \dd r, \dd u) \right ] <\infty. \label{intitowelldef}
\end{align}
We then take the expectation on both sides of \eqref{itoforwk}. Thanks to \eqref{intitowelldef} we can apply the compensation formula for the right-hand side. 
Taking the obtained expression at $t_1$ and $t_2$, and taking the difference, we obtain that for any $k \geq 1$ and $t_2>t_1 \geq 0$, 
\begin{align}
& \mathbb{E}[f(W_k(t_2))]-\mathbb{E}[f(W_k(t_1))] \nonumber \\ = & \int_{(t_1,t_2]} \mathbb{E} \left [ \int_{(0,1)^2} \left ( f (W_k(s)+K^k_{s}(r,u) ) - f (W_k(s) ) \right ) r^{-2} \Lambda(dr) du \right ] ds = \int_{(t_1,t_2]} G^f_{0}(s) ds, \label{compformf}
\end{align}
where $G^f_{0}(s)$ denotes the expectation in the above integral. We need to show that $G^f_{0}(\cdot)$ is continuous. For this we introduce the truncated version
\begin{align*}
G^f_{\epsilon}(s) := \mathbb{E} \left[ \int_{(\epsilon,1)\times (0,1)} \left( f(W_k(s)+K^k_{s}(r,u)) - f(W_k(s)) \right) r^{-2} \Lambda(dr)\, du \right], \quad \epsilon \in (0,1),
\end{align*}
which avoids the singularity near $r=0$ and is therefore easier to handle. We fix $s \geq 0$ and show that $G^f_\epsilon(\cdot)$ is continuous at $s$. Note from Lemma \ref{lawofvkandzk} and Remark \ref{zkindepst} that $(Z_j(s,r,u))_{j\geq 1}$ is independent of $(W_j(s))_{j\geq 1}$ and $S_s$, and that $(Z_j(s,r,u))_{j\geq 1}\sim\Ber(r)^{\otimes \mathbb{N}}$. Using the definition of $K^k_{s}(r,u)$ in \eqref{prepCFlb} together with this we get that, if $k\geq 2$, $G^f_{\epsilon}(s)$ equals 
\begin{align*}
& \int_{(\epsilon,1)} (1-r)^k \mathbb{E} \left [ f \left ( \mathsf{Median} \left \{ W_{k-1}(s), W_k(s), e^{-S_{s}}r+\sum_{j>k}Z_j W_j(s) \right \} \right ) - f(W_k(s)) \right ] r^{-2} \Lambda(dr) \\
+ & \int_{(\epsilon,1)} r(1-r)^{k-1} \mathbb{E} \left [ f \left ( \mathsf{Min} \left \{ W_k(s) + e^{-S_{s}}r+\sum_{j>k}Z_j W_j(s), W_{k-1}(s) \right \} \right ) - f(W_k(s)) \right ] r^{-2} \Lambda(dr) \\
+ & \sum_{\ell=2}^k \binom{k}{\ell} \int_{(\epsilon,1)} r^{\ell}(1-r)^{k-\ell} \mathbb{E} \left [ f \left ( \sum_{j>k} (1-Z_j)W_j(s)\mathds{1}_{\sum_{i=k+1}^j(1-Z_i)=\ell-1} \right ) - f(W_k(s)) \right ] r^{-2} \Lambda(dr), 
\end{align*}
where $(Z_j)_{j\geq 1}$ is independent of $(W_j(s))_{j\geq 1}$ and $S_s$, and $(Z_j)_{j\geq 1}\sim\Ber(r)^{\otimes \mathbb{N}}$. To apply dominated convergence we now check that the arguments of $f$ are almost surely continuous at $s$:
\begin{itemize}
\item $S$ is almost surely continuous at $s$ since it is a L\'evy process;
\item each $W_j(\cdot)$ is almost surely continuous at $s$ by Proposition 
\ref{combrepres}; 
\item By Remark \ref{sumwkteq1minusst} we have almost surely $e^{-S_t}+\sum_{k\geq 1}W_k(t)=1$ for all $t \geq 0$ and each of the terms of this sum are continuous at $t=s$ by the two previous points. By a classical argument we deduce that the term $\sum_{j > k} Z_j W_j(\cdot)$ and the other similar term are almost surely continuous at $s$. 
\end{itemize}
The arguments of $f$ are therefore almost surely continuous at $s$ for any fixed $r \in (\epsilon, 1)$. Since all terms are bounded by 
$2\|f\|_\infty$, dominated convergence gives that $G^f_\epsilon(\cdot)$ is continuous at $s$. The case $k=1$ follows the same argument, using $K^1_s(r,u) = H^1_s(r,u)$ and \eqref{prepCF} in place of \eqref{prepCFlb}. For all $s\geq 0$ we have 
\begin{align*}
|G^f_0(s) - G^f_{\epsilon}(s)| & \leq \int_{(0,\epsilon) \times (0,1)} \mathbb{E} \left [ \left | f (W_k(s)+K^k_{s}(r,u) ) - f \left (W_k(s) \right ) \right | \right ] r^{-2} \Lambda(dr) \times du \\
& \leq C_{f,k} \int_{(0,\epsilon]} r^{-1} \Lambda(dr) \underset{\epsilon \rightarrow 0}{\longrightarrow} 0, 
\end{align*}
where we have used \eqref{prepsmalljumpsbiseqwkt} from Lemma \ref{prepsmalljumpsbis} for the last inequality and \eqref{integassumption} for the convergence toward $0$. Therefore $G^f_{\epsilon}(\cdot)$ converges uniformly to $G^f_0(\cdot)$ as $\epsilon$ goes to $0$. Since all functions $G^f_{\epsilon}(\cdot)$ are continuous, we get that $G^f_0(\cdot)$ is continuous. Combining this continuity with \eqref{compformf} we get that $(t \mapsto \mathbb{E}[f(W_k(t))])$ is of class $\mathcal{C}^1$ and that \eqref{derfwkt} holds true. The proof of \eqref{derfmkt} follows the same argument, with $M_k(t)$ and $H^k_{s-}(r,u)$ playing the roles of $W_k(t)$ and $K^k_{s-}(r,u)$ respectively, and with \eqref{itoformk0}, \eqref{prepCF}, and \eqref{prepsmalljumpsbiseqmkt} in place of \eqref{itoforwk0}, \eqref{prepCFlb}, and \eqref{prepsmalljumpsbiseqwkt}. 
\end{proof}

\appendix

\section{The flow of inverses: existence, shifted flows, and trajectory mergers} \label{propflows}

In this appendix we assume that \eqref{integassumption} holds true and establish the foundational properties of the flow of inverses $Y$. Section \ref{usefulestimsde} collects preliminary estimates. Section \ref{propflows0} proves Propositions \ref{propexistence} and \ref{martpb} on the existence, uniqueness, and law of $Y$. Section \ref{shiftedflows} defines the shifted flows $(Y_{s,s+\cdot}(\cdot))_{s \in J_N}$ and establishes the composition property (Proposition \ref{propcomposition}). Section \ref{sectionnocontmergers} proves that distinct trajectories of $Y$ can only merge at jumping times of $N$ (Proposition \ref{nonzerodiff}).

\subsection{Preliminary estimates} \label{usefulestimsde}

We collect here two estimates on the function $\mr_{r,u}(\cdot)$ appearing in the SDE \eqref{defflowybysde}. Lemma \ref{estimmryu} is used in the construction of the flow from Subsection \ref{propflows0}, and Lemma \ref{dermruymsy} is used in the proof of 
Proposition \ref{calcmeasimgrecipflow} from Section \ref{calcmeas}. 

\begin{lemma} \label{estimmryu}
For any $r, u \in (0,1)$ and $z \in [0,1]$ we have 
\begin{align}
\left | \mr_{r,u}(z)-z \right | \leq \frac{r}{1-r}. \label{smallneutral}
\end{align}
For any $r \in (0,1)$ and $1 \geq z_1 \geq z_2 \geq 0$ we have 
\begin{align}
\int_0^1 |\mr_{r,u}(z_1)-\mr_{r,u}(z_2) - z_1+z_2| du & \leq \frac{4r}{(1-r)^{2}}|z_1-z_2|. \label{largeneutral}
\end{align}
\end{lemma}

\begin{proof}
Let $r, u \in (0,1)$ and $z \in [0,1]$ and let us study $|\mr_{r,u}(z)-z|$. If $u \leq \frac{z-r}{1-r}$ then $|\mr_{r,u}(z)-z|=\frac{r}{1-r}(1-z)\leq \frac{r}{1-r}$. If $\frac{z-r}{1-r} \leq u \leq \frac{z}{1-r}$ then $|\mr_{r,u}(z)-z|=|u-z|\leq (z/(1-r)-z)\vee(z-(z-r)/(1-r))\leq \frac{r}{1-r}$. If $\frac{z}{1-r} \leq u$ then $|\mr_{r,u}(z)-z|=\frac{r}{1-r}z\leq \frac{r}{1-r}$. We thus get \eqref{smallneutral} in all cases. 

Let now $r, u \in (0,1)$ and $1 \geq z_1 \geq z_2 \geq 0$. We study $\mr_{r,u}(z_1)-\mr_{r,u}(z_2)$ by distinguishing the position of $u$ relative to $\frac{z_2-r}{1-r}$, $\frac{z_1-r}{1-r}$, $\frac{z_2}{1-r}$, $\frac{z_1}{1-r}$, as well as the relative positions of $\frac{z_1-r}{1-r}$ and $\frac{z_2}{1-r}$. 
\begin{enumerate}
\item[(i)] If $u \leq \frac{z_2-r}{1-r}$ or $\frac{z_1}{1-r} \leq u$ then 
\[ \mr_{r,u}(z_1)-\mr_{r,u}(z_2) = \frac{z_1-z_2}{1-r}. \]
\item[(ii)] If $\frac{z_2-r}{1-r} \leq u \leq \frac{z_1-r}{1-r} \wedge \frac{z_2}{1-r}$ then $\mr_{r,u}(z_1)-\mr_{r,u}(z_2) = \frac{z_1-r}{1-r} - u$ so 
\[ \frac{z_1-z_2}{1-r} - \frac{r}{1-r} \leq \mr_{r,u}(z_1)-\mr_{r,u}(z_2) \leq \frac{z_1-r}{1-r} - \frac{z_2-r}{1-r} = \frac{z_1-z_2}{1-r}. \]
\item[(iii)] If $\frac{z_1-r}{1-r} \leq u \leq \frac{z_2}{1-r}$ then 
\[ \mr_{r,u}(z_1)-\mr_{r,u}(z_2) = u - u = 0. \]
\item[(iv)] If $\frac{z_2}{1-r} \leq u \leq \frac{z_1-r}{1-r}$ then 
\[ \mr_{r,u}(z_1)-\mr_{r,u}(z_2) = \frac{z_1-r}{1-r} - \frac{z_2}{1-r} = \frac{z_1-z_2}{1-r} - \frac{r}{1-r}. \]
\item[(v)] If $\frac{z_1-r}{1-r} \vee \frac{z_2}{1-r} \leq u \leq \frac{z_1}{1-r}$ then $\mr_{r,u}(z_1)-\mr_{r,u}(z_2) = u - \frac{z_2}{1-r}$ so 
\[ \frac{z_1-z_2}{1-r} - \frac{r}{1-r} \leq \mr_{r,u}(z_1)-\mr_{r,u}(z_2) \leq \frac{z_1}{1-r} - \frac{z_2}{1-r} = \frac{z_1-z_2}{1-r}. \]
\end{enumerate}
In conclusion we get 
\begin{align*}
|\mr_{r,u}(z_1)-\mr_{r,u}(z_2) - z_1+z_2| & \leq \frac{r}{1-r} |z_1-z_2| + \frac{r}{1-r} \mathds{1}_{u \in [\frac{z_2-r}{1-r}, \frac{z_1-r}{1-r}] \cup [\frac{z_2}{1-r}, \frac{z_1}{1-r}]} \nonumber \\
& + |z_1-z_2| \mathds{1}_{u \in [\frac{z_1-r}{1-r}, \frac{z_2}{1-r}]}. 
\end{align*}
Integrating with respect to $u$ we get \eqref{largeneutral}. 
\end{proof}

\begin{lemma} \label{dermruymsy}
For $y \in [0,1]$, $(1-r)\mr_{r,u}(y)-y = - \int_0^y \mathds{1}_{z\in I_{r,u}^{\mathrm{o}}} dz$, where $I_{r,u}$ is defined in \eqref{defiru}. 
\end{lemma}

\begin{proof}
Using the definition of $\mr_{r,u}(\cdot)$ we get 
\begin{align*}
(1-r)\mr_{r,u}(y)-y & = y \mathds{1}_{y\leq(1-r)u} + (1-r)u\mathds{1}_{y\in I_{r,u}^{\mathrm{o}}} + (y-r) \mathds{1}_{y \geq (1-r)u+r}-y \nonumber \\
& = \left ( (1-r)u-y \right ) \mathds{1}_{y\in I_{r,u}^{\mathrm{o}}} - r \mathds{1}_{y \geq (1-r)u+r} = - \int_0^y \mathds{1}_{z\in I_{r,u}^{\mathrm{o}}} dz. 
\end{align*}
\end{proof}

\subsection{Construction of the flow of inverses and proof of Propositions \ref{propexistence} and \ref{martpb}} \label{propflows0}
The proof of Proposition \ref{propexistence} proceeds in several steps. We first introduce the approximating flow $Y^{\delta}_{0,\cdot}(\cdot)$ (which has finitely many jumps) and show its Lipschitz regularity with respect to the initial condition in Lemma \ref{propydelta}. We then construct the one-point motion of the target flow in Lemma \ref{existofsol} (allowing us to define $Y_{0,\cdot}(y)$ for $y \in Q:=[0,1] \cap \mathbb{Q}$) and show in Lemmas \ref{boundflowlem} and \ref{cvyepstoyentpsdesauts0} that it can be approximated by the one point motion of $Y^{\delta}_{0,\cdot}(\cdot)$. Then we pass to the limit in the bound from Lemma \ref{propydelta}, yielding the uniform continuity of $y \mapsto Y_{0,\cdot}(y)$ on $Q$, allowing us to extend the flow to all of $[0,1]$ by density. Finally, we show that the resulting flow solves \eqref{defflowybysde}. In Proposition \ref{martpb} we identifies the law of that flow with that of the flow of inverses of the $\Lambda$-process. 

We denote by $D([0,T])$ (resp. $D([0,\infty))$) the space of c\`ad-l\`ag functions from $[0,T]$ (resp. $[0,\infty)$) to $\mathbb{R}$. We sometimes use the metrics $d_T$ and $d_{\infty}$ on $D([0,T])$ and $D([0,\infty))$defined by 
\begin{align}
d_T(f,g):= 1 \wedge \sup_{t \in [0,T]} |f(t)-g(t)|, \ d_{\infty}(f,g):=\int_0^{\infty} e^{-T} d_T(f_{|[0,T]},g_{|[0,T]}) dT. \label{defdistunif}
\end{align}
We note that $(D([0,T]),d_T)$ (resp. $(D([0,\infty)),d_{\infty})$) is a complete metric space. Moreover, the topology induced by $d_T$ (resp. $d_{\infty}$) is stronger than the usual Skorokhod topology. 

For $\delta \in (0,1)$, we consider the flow $(Y^{\delta}_{0,t}(y), y \in [0,1], t \geq 0)$ defined by 
\begin{align}
Y^{\delta}_{0,t}(y) = y + \int_{(0,t] \times (\delta,1) \times (0,1)} \left ( \mr_{r,u}(Y^{\delta}_{0,s-}(y))-Y^{\delta}_{0,s-}(y) \right ) N(ds,dr,du), \label{defflowybysdedelta}
\end{align} 
for all $y\in [0,1]$ and $t\geq 0$. Comparing with \eqref{defflowybysde} we see that the flow $Y^{\delta}_{0,\cdot}(\cdot)$ is obtained similarly as $Y_{0,\cdot}(\cdot)$, but by keeping only the jumps of $N$ with a $r$-component larger than $\delta$. However, since such jumps occur at finite rate, the flow $Y^{\delta}_{0,\cdot}(\cdot)$ is much simpler. It can be defined as follows. Let $\delta \in (0,1)$ and $(s_k,r_k,u_k)_{k\geq 1}$ be the enumeration of the discrete set $\{(s,r,u) \in N, r>\delta\}$ with increasing time components. We set $Y^{\delta}_{0,t}(\cdot):=\mr_{r_n,u_n} \circ \ldots \circ \mr_{r_1,u_1}$ if $t \in [s_n,s_{n+1})$ and $Y^{\delta}_{0,t}(\cdot)$ to be the identity function if $t<s_1$. It is easy to see by induction that, for any $\delta \in (0,1)$, the flow such defined is the unique solution of \eqref{defflowybysdedelta}. 
\begin{lemma} \label{propydelta}
We have almost surely that, for all $T>0$, $\delta \in (0,1)$, and $a,b \in [0,1]$, 
\begin{align}
d_T(Y^{\delta}_{0,\cdot}(a), Y^{\delta}_{0,\cdot}(b)) \leq e^{S_T} |a-b|. \label{estcontydelta}
\end{align}
\end{lemma}
In the proof of Proposition \ref{propexistence}, we will show that we can take the limit as $\delta \to 0$ in this estimate, thus yielding uniform continuity for the target flow $Y_{0,\cdot}(y)$ with respect to $y$. 
\begin{proof}
Each function $\mr_{r,u}(\cdot)$ is Lipschitz continuous with Lipschitz constant $1/(1-r)$. We thus get that, for $t \in [s_n,s_{n+1})$ (resp. for $t<s_1$), $Y^{\delta}_{0,t}(\cdot)$ is Lipschitz continuous with Lipschitz constant $\prod_{k=1}^n 1/(1-r_k)\leq e^{S_t}$ (resp. $1 \leq e^{S_t}$). This proves \eqref{estcontydelta}. 
\end{proof}

In order to prove the existence of a unique stochastic flow satisfying \eqref{defflowybysde}, we first turn our attention to a simpler SDE. If there exist a flow $(Y_{0,t}(y), y \in [0,1], t \geq 0)$ satisfying \eqref{defflowybysde}, then a single trajectory $Y_{0,\cdot}(y)$ (also called \textit{one-point motion}) is solution of the SDE 
\begin{align}
Y_t = y + \int_{(0,t] \times (0,1)^2} \left ( \mr_{r,u}(Y_{s-})-Y_{s-} \right ) N(ds,dr,du), \ t\geq 0. \label{sdesingletraj}
\end{align}
We say that a process $(Y_t)_{t\geq 0}$ satisfying \eqref{sdesingletraj} almost surely for all $t\geq 0$ is a solution of \eqref{sdesingletraj} with initial value $y$. We call it a \textit{strong solution} if it is c\`ad-l\`ag and adapted to the filtration $(\mathcal{F}_{t})_{t\geq 0}$ (defined a little before \eqref{defflowxbysde}). We say that \textit{pathwise uniqueness} holds if any two solutions with same initial values are almost surely equal for all $t\geq 0$. The following lemma lays out some facts about SDE \eqref{sdesingletraj} and follows from \cite[Prop. B.5]{cordhumvech2022}. 
\begin{lemma} \label{existofsol}
For any $y\in[0,1]$, there exists a pathwise unique strong solution $(Y_t)_{t\geq 0}$ of \eqref{sdesingletraj} with initial value $y$. Moreover it satisfies $Y_t\in [0,1]$ for all $t\geq 0$. If $0\leq y_1\leq y_2\leq 1$ and $Y^1$ and $Y^2$ are the solutions of \eqref{sdesingletraj} with initial values $y_1$ and $y_2$ respectively, then $\mathbb{P}(Y^1_t\leq Y^2_t\ \text{for all }t\geq 0)=1$. 
\end{lemma}

The following lemma allows to control approximations of $Y$ (starting from $y\in[0,1]$) by $Y^{\delta}_{0,\cdot}(y)$. 
\begin{lemma} \label{boundflowlem}
Let $y\in[0,1]$ and $(Y_t)_{t\geq 0}$ be the unique strong solution of \eqref{sdesingletraj} with initial value $y$. Let also $Y^{\delta}_t:=Y^{\delta}_{0,t}(y)$. For any $T>0$, $\delta \in (0,1/2)$, and $\rho \in [1/2,1)$, we have 
\begin{align}
& \mathbb{E} \left [ d_T(Y_{\cdot}, Y^{\delta}_{\cdot}) \mathds{1}_{N((0,T]\times(\rho,1)\times (0,1))=0} \right ] \leq C_T(\rho) \times K(\delta), \label{boundflow}
\end{align}
where $K(\delta) := \int_{(0,\delta]} (1-r)^{-1} r^{-1} \Lambda(\dd r) < \infty$ and $C_T(\rho) := \frac{e^{4T\int_{(0,\rho]}(1-r)^{-2}r^{-1}\Lambda(\dd r)}-1}{4\int_{(0,\rho]}(1-r)^{-2}r^{-1}\Lambda(\dd r)} < \infty$. 
\end{lemma}
Note that $K(\delta) \to 0$ as $\delta \to 0$ by \eqref{integassumption}. 
\begin{proof}
We fix $y, T, \delta, \rho$ as in the statement of the lemma. Using \eqref{sdesingletraj} and \eqref{defflowybysdedelta} we get 
\begin{align*}
Y_t - Y^{\delta}_t = & \int_{(0,t] \times (0,\delta] \times (0,1)} \big(\mr_{r,u}(Y_{s-}) - Y_{s-}\big)N(\dd s, \dd r, \dd u) \nonumber \\
+ & \int_{(0,t] \times (\delta,1) \times (0,1)} \big(\mr_{r,u}(Y_{s-})-\mr_{r,u}(Y^{\delta}_{s-}) - Y_{s-}+Y^{\delta}_{s-}\big)N(\dd s, \dd r, \dd u), 
\end{align*}
for $t \in [0,T]$. We get that almost surely, 
\begin{align}
\sup_{t \in [0,T]} |Y_t - Y^{\delta}_t| \leq & \int_{(0,T] \times (0,\delta] \times (0,1)} \big|\mr_{r,u}(Y_{s-}) - Y_{s-}\big|N(\dd s, \dd r, \dd u) \label{sdediff2} \\
+ & \int_{(0,T] \times (\delta,1) \times (0,1)} \big|\mr_{r,u}(Y_{s-})-\mr_{r,u}(Y^{\delta}_{s-}) - Y_{s-}+Y^{\delta}_{s-}\big|N(\dd s, \dd r, \dd u). \nonumber
\end{align}
To handle the large jumps separately from the small ones, we introduce truncated versions $(\tilde Y_t)_{t\geq 0}$ and $(\tilde Y^{\delta}_t)_{t\geq 0}$ defined as $(Y_t)_{t\geq 0}$ and $(Y^{\delta}_t)_{t\geq 0}$ but where $(0,1)^2$ and $(\delta,1)\times(0,1)$ from \eqref{sdesingletraj} and \eqref{defflowybysdedelta} are replaced by respectively $(0,\rho)\times(0,1)$ and $(\delta,\rho)\times(0,1)$. The key observation is that, on the event $\{N((0,T]\times(\rho,1)\times(0,1))=0\}$, the processes $Y$ and $Y^\delta$ coincide with their truncated versions. More precisely, for any measurable function $f:[0,1]^2\times(0,1)^2\to\mathbb{R}_+$, 
\begin{align}
& \mathbb{E} \left [ \mathds{1}_{N((0,T]\times(\rho,1)\times (0,1))=0} \int_{(0,T] \times (\delta,1) \times (0,1)} f(Y_{s-},Y^{\delta}_{s-},r,u) N(\dd s, \dd r, \dd u) \right ] \label{indcompoform} \\
= & \mathbb{E} \left [ \mathds{1}_{N((0,T]\times(\rho,1)\times (0,1))=0} \int_{(0,T] \times (\delta,\rho] \times (0,1)} f(\tilde Y_{s-},\tilde Y^{\delta}_{s-},r,u) N(\dd s, \dd r, \dd u) \right ] \nonumber \\
= & e^{-T \int_{(\rho,1)} r^{-2} \Lambda(\dd r)} \int_0^T \mathbb{E} \left [ \int_{(\delta,\rho] \times (0,1)} f(\tilde Y_{s},\tilde Y^{\delta}_{s},r,u) r^{-2} \Lambda(\dd r) du \right ] ds \nonumber \\
\leq & \int_0^T \mathbb{E} \left [ \mathds{1}_{N((0,s]\times(\rho,1)\times (0,1))=0} \int_{(\delta,\rho] \times (0,1)} f(Y_{s},Y^{\delta}_{s},r,u) r^{-2} \Lambda(\dd r) du \right ] ds, \nonumber
\end{align}
where we have used that $\int_{(0,T] \times (\delta,\rho] \times (0,1)} f(\tilde Y_{0,s-}(a),\tilde Y^{\delta}_{0,s-}(b),r,u) N(\dd s, \dd r, \dd u)$ is a measurable function of $N((0,T] \times (\delta,\rho] \times (0,1) \cap \cdot)$, which is independent of $N((0,T]\times(\rho,1)\times (0,1) \cap \cdot)$, and the compensation formula. 
Using \eqref{smallneutral} and \eqref{largeneutral} from Lemma \ref{estimmryu} we get that, 
\begin{align}
\int_{(0,\delta] \times (0,1)} \left | \mr_{r,u}(Y_{s}) - Y_{s} \right | r^{-2} \Lambda(\dd r) du & \leq \int_{(0,\delta]} \frac{\Lambda(\dd r)}{(1-r)r} \label{evodiff2traj}, \\
\int_{(\delta,\rho] \times (0,1)} \left | \mr_{r,u}(Y_{s})-\mr_{r,u}(Y^{\delta}_{s}) - Y_{s}+Y^{\delta}_{s} \right | r^{-2} \Lambda(\dd r) du & \leq 4 \left ( \int_{(\delta,\rho]} \frac{\Lambda(\dd r)}{(1-r)^{2}r} \right ) |Y_{s} - Y^{\delta}_{s}|. \label{evodiff2trajbis}
\end{align}
Multiplying each term in \eqref{sdediff2} by $\mathds{1}_{N((0,T]\times(\rho,1)\times (0,1))=0}$, taking the expectation and using \eqref{indcompoform} (and the compensation formula for the term $\int_{(0,T] \times (0,\delta] \times (0,1)}\ldots$) and \eqref{evodiff2traj}-\eqref{evodiff2trajbis}, we get that the left-hand side of \eqref{boundflow} is smaller than 
\begin{align*}
T\int_{(0,\delta]} \frac{\Lambda(\dd r)}{(1-r)r} + 4 \left ( \int_{(0,\rho]} \frac{\Lambda(\dd r)}{(1-r)^{2}r} \right ) \times \int_0^T \mathbb{E} \left [ \left ( \sup_{t \in [0,s]} |Y_{t} - Y^{\delta}_{t}| \right ) \mathds{1}_{N((0,s]\times(\rho,1)\times (0,1))=0} \right ] ds. 
\end{align*}
We then get \eqref{boundflow} using Gronwall's lemma. 
\end{proof}

Recall that $Q:=[0,1] \cap \mathbb{Q}$. By Lemma \ref{existofsol} a flow $(Y_{0,t}(y), y \in Q, t \geq 0)$ can be defined which satisfies \eqref{defflowybysde} (with "for all $y\in [0,1]$" replaced by "for all $y\in Q$"), is c\`ad-l\`ag in $t$, is non-decreasing in $y$, and $Y_{0,t}(0)=0$, $Y_{0,t}(1)=1$. In order to extend this flow to $[0,1] \times [0,\infty)$, we need the following lemma that builds on Lemma \ref{boundflowlem} and shows that the flow of $Y_{0,\cdot}(\cdot)$ can be approximated by $Y^{\delta}_{0,\cdot}(\cdot)$. 
\begin{lemma} \label{cvyepstoyentpsdesauts0}
There is a decreasing sequence $(\delta_n)_{n\geq 1}$ in $(0,1/2)$ such that for any $T>0$ we have almost surely that for all $y \in Q$, $d_T(Y_{0,\cdot}(y), Y^{\delta_n}_{0,\cdot}(y))\rightarrow 0$ as $n \rightarrow \infty$. 
\end{lemma}
\begin{proof}
By \eqref{integassumption} and the definition of $K(\delta)$ in Lemma \ref{boundflowlem} we have $K(\delta)\rightarrow0$ as $\delta \rightarrow 0$. We can thus choose a decreasing sequence $(\delta_n)_{n\geq 1}$ such that for any $n\geq 1$ we have $K(\delta_n) \leq 2^{-n}$. Fix $T>0$ and $y \in Q$. Applying Lemma \ref{boundflowlem} we get that for any $\rho \in [1/2,1) \cap \mathbb{Q}$, $n\geq 1$, 
\begin{align*}
\mathbb{E} \left [ d_T(Y_{0,\cdot}(y), Y^{\delta_n}_{0,\cdot}(y)) \mathds{1}_{N((0,T]\times(\rho,1)\times (0,1))=0} \right ] \leq 2^{-n} C_T(\rho). 
\end{align*}
Combining with Markov inequality we get that for any $\rho \in [1/2,1) \cap \mathbb{Q}$, $\epsilon \in (0,1) \cap \mathbb{Q}$, and $n\geq 1$, 
\begin{align*}
& \mathbb{P} \left (d_T(Y_{0,\cdot}(y), Y^{\delta_n}_{0,\cdot}(y)) > \epsilon, N((0,T]\times(\rho,1)\times (0,1))=0 \right ) \leq 2^{-n} C_T(\rho)/\epsilon. 
\end{align*}
By the Borel-Cantelli lemma we get that, on $\{N((0,T]\times(\rho,1)\times (0,1))=0\}$, we have almost surely $d_T(Y_{0,\cdot}(y), Y^{\delta_n}_{0,\cdot}(y)) \leq \epsilon$ for all large $n$. Since this is true for all $\rho \in [1/2,1) \cap \mathbb{Q}$ and $\epsilon \in (0,1) \cap \mathbb{Q}$, and since $\bigcup_{\rho \in [1/2,1)\cap\mathbb{Q}} \{N((0,T]\times(\rho,1)\times (0,1))=0\}$ has probability one, we get that $d_T(Y_{0,\cdot}(y), Y^{\delta_n}_{0,\cdot}(y))$ converges almost surely to $0$ as $n$ goes to infinity. Since $Q$ is countable the result follows. 
\end{proof}

\begin{proof}[Proof of Proposition \ref{propexistence}]
If two such flows exist, they almost surely coincide on all $(y,t)\in Q \times [0,\infty)$ by Lemma \ref{existofsol} and then on all $(y,t)\in [0,1] \times [0,\infty)$ by continuity with respect to $y$. This proves uniqueness. 
We now prove existence. We consider the flow $(Y_{0,t}(y), y \in Q, t \geq 0)$ as defined before Lemma \ref{cvyepstoyentpsdesauts0}. By that lemma, there is a decreasing sequence $(\delta_n)_{n\geq 1}$ in $(0,1)$ such that for any $T>0$ we have almost surely that for all $y \in Q$, $d_T(Y_{0,\cdot}(y), Y^{\delta_n}_{0,\cdot}(y))$ converges to $0$ as $n$ goes to infinity. Combining with Lemma \ref{propydelta} we get that, almost surely, 
\begin{align}
\forall T \in (0,\infty) \cap \mathbb{Q}, \forall a,b \in Q, \ d_T(Y_{0,\cdot}(a), Y_{0,\cdot}(b)) \leq e^{S_T} |a-b|. \label{estcontydyad}
\end{align}
Since $(D([0,\infty)),d_\infty)$ is a complete metric space and \eqref{estcontydyad} shows that $y \mapsto Y_{0,\cdot}(y)$ is almost surely uniformly continuous from $(Q, |\cdot|)$ to $(D([0,\infty)),d_\infty)$, there is a probability one event on which this map extends uniquely to a continuous map from $[0,1]$ to $(D([0,\infty)),d_\infty)$. For any fixed realization of this event, we can thus define $Y_{0,\cdot}(y)$ for all $y\in [0,1]$ by extension and obtain a flow satisfying (ii),(iii). To show that it satisfies (i) we consider $y \in [0,1]$ and $(y_n)_{n\geq 1}$ in $Q$ converging to $y$. Then, for each $n\geq 1$, $Y_{0,\cdot}(y_n)$ satisfies \eqref{sdesingletraj}, i.e. 
\begin{align}
Y_{0,t}(y_n) = y_n + \int_{(0,t] \times (0,1)^2} \left ( \mr_{r,u}(Y_{0,s-}(y_n))-Y_{0,s-}(y_n) \right ) N(ds,dr,du), \ t\geq 0. \label{eqflowapprox}
\end{align} 
Since $Y_{0,\cdot}(y_n)$ converges to $Y_{0,\cdot}(y)$ in $(D([0,\infty)),d_{\infty})$, the left-hand side of \eqref{eqflowapprox} converges to $Y_{0,t}(y)$ while the integrand in the right-hand side converges to $\mr_{r,u}(Y_{0,s-}(y))-Y_{0,s-}(y)$. By \eqref{smallneutral} from Lemma \ref{estimmryu}, the absolute value of the integrand is bounded by $r/(1-r)$, which is integrable with respect to $N$ almost surely by \eqref{integassumption}. Therefore, dominated convergence applies and the right-hand side of \eqref{eqflowapprox} converges to the right-hand side of \eqref{defflowybysde}. Thus the flow indeed satisfies (i). 
\end{proof}

\begin{remark} \label{contfrom01tocadlag}
The above proof shows that, for the flow $(Y_{0,t}(y), y \in [0,1], t \geq 0)$ from Proposition \ref{propexistence}, we have almost surely that $y \mapsto Y_{0,\cdot}(y)$ is continuous from $[0,1]$ to $(D([0,\infty)),d_{\infty})$. 
\end{remark}

The following proposition allows to identify (in law) the flow $(Y_{0,t}(y), y \in [0,1], t \geq 0)$ with the flow of inverses of the $\Lambda$-process (see \eqref{defflowxbysde}). 
\begin{prop} \label{martpb}
For any $p\geq 1$ and $y_1,\cdots,y_p \in [0,1]$ with $y_1 \leq \cdots \leq y_p$, $(Y_{0,t}(y_1),\cdots,Y_{0,t}(y_p))_{t\geq 0}$ is solution to the martingale problem from \cite[Thm. 5]{BERTOIN2005307} and, under the assumption \eqref{integassumption}, this martingale problem is well posed. 
\end{prop}

\begin{proof}
By \eqref{defflowybysde}, the $\mathbb{R}^p$-valued process $(Y_{0,t}(y_1),\cdots,Y_{0,t}(y_p))_{t\geq 0}$ satisfies the SDE 
\begin{align}
Z^j_t = y_j + \int_{(0,t] \times (0,1)^2} \left ( \mr_{r,u}(Z^j_{s-})-Z^j_{s-} \right ) N(ds,dr,du), \ j \in \{ 1,\ldots,p \}, t \geq 0. \label{sdepmotion}
\end{align} 
Applying It{\^o}'s formula (see e.g. \cite[Thm. II.5.1]{ikedawatanabe}) we get that this process solves the martingale problem. 
Note that "$(Z^1_t,\ldots,Z^p_t)_{t\geq 0}$ solves the SDE \eqref{sdepmotion} with initial value $(y_1,\ldots,y_p)$" is equivalent to "for each $i \in \{1,\ldots,p\}$, $(Z^i_t)_{t\geq 0}$ solves \eqref{sdesingletraj} with initial value $y_i$. We deduce from Lemma \ref{existofsol} that there exists a pathwise unique strong solution of \eqref{sdepmotion} with initial value $(y_1,\ldots,y_p)$. Moreover it satisfies $(Z^1_t,\ldots,Z^p_t)\in [0,1]^p$ for all $t\geq 0$. By \cite[Thm. 2.3]{Kurtz2011}, every solution to the martingale problem is a weak solution to the SDE \eqref{sdepmotion}. Since pathwise uniqueness implies weak uniqueness (see e.g. \cite[Thm. 1]{BLP15}), the martingale problem is well posed. This completes the proof. 
\end{proof}

\subsection{Shifted flows and composition property} \label{shiftedflows}

In this subsection we define the flows \\ $(Y_{s,t}(y), y \in [0,1], t \geq s)$ started at jumping times $s \in J_N$ and establish the composition property \eqref{compprop} (Proposition \ref{propcomposition}). These shifted flows are used throughout the paper, in particular in the proof of Theorem \ref{intsto}. Our argument is similar to \cite[Thm. 4.5]{10.1214/10-AOP629}. 

If $\tau$ is a stopping time with respect to the filtration $(\mathcal{F}_{t})_{t\geq 0}$ we consider stochastic flows $(Y_{\tau,t}(y), y \in [0,1], t \geq \tau)$ that satisfies 
\begin{align}
Y_{\tau,t}(y) = y + \int_{(\tau,t] \times (0,1)^2} \left ( \mr_{r,u}(Y_{\tau,s-}(y))-Y_{\tau,s-}(y) \right ) N(ds,dr,du), \label{defflowybysdetspsauts}
\end{align} 
almost surely for all $y \in [0,1]$ and $t \geq \tau$. This flow is well-defined almost surely by Proposition \ref{propexistence} applied to the shifted measure $N(\tau+ds,dr,du)$. Moreover the flow $(Y_{\tau,\tau+t}(y), y \in [0,1], t \geq 0)$ is equal in law to $(Y_{0,t}(y), y \in [0,1], t \geq 0)$ from Proposition \ref{propexistence} and independent of $\mathcal{F}_{\tau}$ by the strong Markov property applied to the Poisson random measure $N(ds,dr,du)$. 

For $\eta\in(0,1/2)$, let $(s^{\eta}_i,r^{\eta}_i,u^{\eta}_i)_{i\geq 1}$ be the enumeration of $\{(s,r,u) \in N, r>\eta\}$ such that $s^{\eta}_1<s^{\eta}_2<\ldots$ and for convenience we set $s^{\eta}_0:=0$. Note that for any $j\geq 1$, $s^{\eta}_j$ is a stopping time. We can thus define the countable collection of flows $\{(Y_{s^{\eta}_j,t}(y), y \in [0,1], t \geq s^{\eta}_j), \eta\in(0,1/2) \cap \mathbb{Q}, j \geq 1\}$ on the same probability space. 

To establish the composition property, we fix $\eta \in (0,1/2)\cap\mathbb{Q}$, $0 \leq i < j$, and define $Z^{\eta,i,j}_{s^{\eta}_i,t}(\cdot):=Y_{s^{\eta}_i,t}(\cdot)$ for $t \in [s^{\eta}_i,s^{\eta}_j)$ and $Z^{\eta,i,j}_{s^{\eta}_i,t}(\cdot):=Y_{s^{\eta}_j,t}(Y_{s^{\eta}_i,s^{\eta}_j}(\cdot))$ for $t \in [s^{\eta}_j,\infty)$. It is not difficult to see that this flow satisfies \eqref{defflowybysdetspsauts} with $\tau=s^{\eta}_i$ and properties (ii) and (iii) from Proposition \ref{propexistence} (for $t \geq s^{\eta}_i$ instead of $t\geq 0$). By uniqueness from Proposition \ref{propexistence} we get that $Z^{\eta,i,j}_{s^{\eta}_i,\cdot}(\cdot)=Y_{s^{\eta}_i,\cdot}(\cdot)$. This yields the composition property $Y_{s^{\eta}_i,t}(\cdot)=Y_{s^{\eta}_j,t}(Y_{s^{\eta}_i,s^{\eta}_j}(\cdot))$ for all $t \geq s^{\eta}_j$. 

Since any pair $s_1 < s_2$ in $J_N \cup \{0\}$ satisfies $s_1 = s^\eta_i$ and $s_2 = s^\eta_j$ for some $\eta\in(0,1/2) \cap \mathbb{Q}$ and $i<j$, the composition property extends to all pairs in $J_N \cup \{0\}$. Thus, we have proved the following proposition.

\begin{prop} \label{propcomposition}
One can define a countable family of stochastic flows $\{ (Y_{s,t}(y), y \in [0,1], t \geq s), s \in J_N \cup \{0\}\}$ such that, almost surely, each flow $(Y_{s,t}(y), y \in [0,1], t \geq s)$ from this family satisfies the following properties: 
\begin{enumerate}[(i)]
\item \eqref{defflowybysdetspsauts} holds with $\tau=s$, for all $y \in [0,1]$ and $t \geq s$;
\item for every $y \in [0,1]$, the trajectory $t \mapsto Y_{s,t}(y)$ is c\`ad-l\`ag; 
\item for every $t \geq s$, the map $y \mapsto Y_{s,t}(y)$ is non-decreasing and continuous, and $Y_{s,t}(0)=0$, $Y_{s,t}(1)=1$. 
\end{enumerate}
Moreover, almost surely, for any $s_1,s_2 \in J_N \cup \{0\}$ with $s_1<s_2$, \eqref{compprop} holds true. 
\end{prop}

\subsection{No continuous mergers: Proof of Proposition \ref{nonzerodiff}} \label{sectionnocontmergers}

In this subsection we prove Proposition \ref{nonzerodiff} which states that distinct trajectories of the flow can only merge at jumping times of $N$. The proof proceeds in several steps. Lemma \ref{ntxfinite} shows that, on any bounded time interval, each trajectory is involved in only finitely many mergers caused by jumps of $N$. Lemma \ref{finiteexpectoflog} then shows that, between two consecutive mergers caused by jumps, distinct trajectories cannot merge. Together, these two lemmas imply Lemma \ref{nonzerodiffrationals}, which establishes the result for trajectories starting from $Q:=[0,1] \cap \mathbb{Q}$. The full result, is then deduced. For $y \in [0,1]$ and $0 \leq t_1 \leq t_2$, let 
\begin{align*}
N_{t_1,t_2}(y) := \{ (s,r,u) \in N \ \text{s.t.} \ s \in (t_1,t_2], \ Y_{0,s-}(y) \in I_{r,u} \}
\end{align*}
denote the set of jumps of $N$ that cause a merger for the trajectory $Y_{0,\cdot}(y)$ on $(t_1,t_2]$.
\begin{lemma} \label{ntxfinite}
Almost surely, $N_{t_1,t_2}(y)$ is finite for all $y \in Q$ and $0 \leq t_1 \leq t_2 < \infty$. Moreover, almost surely, for almost every $z \in [0,1]$, $N_{t_1,t_2}(z)$ is finite for all $0 \leq t_1 \leq t_2 < \infty$.
\end{lemma}

\begin{proof}
First note that for any $y\in [0,1]$, and $r \leq 1/2$, 
\begin{align}
\int_{(0,1)} \mathds{1}_{y \in I_{r,u}} du = \left | [\frac{y-r}{1-r}, \frac{y}{1-r} ] \cap [0,1] \right | \leq \frac{r}{1-r} \leq 2r. \label{majoprobayinfusion}
\end{align}
Fix $y\in [0,1]$ and let $\tilde N_{t_1,t_2}(y):=\{ (s,r,u) \in N_{t_1,t_2}(y) \ \text{s.t.} \ r \leq 1/2 \}$. By the compensation formula and \eqref{majoprobayinfusion} we get 
\begin{align}
& \mathbb{E} \left [ \sharp \tilde N_{t_1,t_2}(y) \right ] = \mathbb{E} \left [ \sum_{(s,r,u) \in N, s \in (t_1,t_2], r \leq 1/2} \mathds{1}_{Y_{0,s-}(y) \in I_{r,u}} \right ] \nonumber \\
= & \int_{t_1}^{t_2} \mathbb{E} \left [ \int_{(0,1/2]} \left ( \int_{(0,1)} \mathds{1}_{Y_{0,s}(y) \in I_{r,u}} du \right ) r^{-2} \Lambda(dr) \right ] ds \leq 2(t_2-t_1)\int_{(0,1/2]}r^{-1} \Lambda(dr)<\infty, \label{finitentt}
\end{align}
where the finiteness comes from \eqref{integassumption}. The combination of \eqref{finitentt} with the monotonicity of $\tilde N_{t_1,t_2}(y)$ with respect to the time interval $(t_1,t_2]$ show that $\tilde N_{t_1,t_2}(y)$ is almost surely finite for all $t_1,t_2 \geq 0$. Clearly $N_{t_1,t_2}(y) \setminus \tilde N_{t_1,t_2}(y) \subset \{ (s,r,u) \in N \ \text{s.t.} \ s \in (t_1,t_2], r > 1/2 \}$, which is also almost surely finite for all $t_1,t_2 \geq 0$. The first statement is thus proved for a fixed $y\in [0,1]$ and, since $Q$ is countable, the statement follows. Then, using Fubini's theorem and \eqref{finitentt}, 
\begin{align*}
\mathbb{E} \left [ \int_{[0,1]} \sharp \tilde N_{t_1,t_2}(z) dz \right ] = \int_{[0,1]} \mathbb{E} \left [ \sharp \tilde N_{t_1,t_2}(z) \right ] dz \leq 2(t_2-t_1)\int_{(0,1/2)}r^{-1} \Lambda(dr)<\infty. 
\end{align*}
We thus get that $\int_{[0,1]} \sharp \tilde N_{t_1,t_2}(z) dz$ is almost surely finite so, for almost every $z\in[0,1]$, $\tilde N_{t_1,t_2}(z)$ is finite. By the monotonicity of $\tilde N_{t_1,t_2}(y)$ with respect to the time interval $(t_1,t_2]$ we get that the second statement of the lemma holds for $\tilde N_{t_1,t_2}(z)$. Then, $N_{t_1,t_2}(z) \setminus \tilde N_{t_1,t_2}(z) \subset \{ (s,r,u) \in N \ \text{s.t.} \ s \in (t_1,t_2], r > 1/2 \}$, which does not depend on $z$ and is also almost surely finite for all $t_1,t_2 \geq 0$. This concludes the proof of the second statement of the lemma for $N_{t_1,t_2}(z)$. 
\end{proof}

By Lemma \ref{ntxfinite}, on a probability one event, $N_{0,\infty}(y)$ is discrete for all $y \in Q$. On this event, for any $y_1, y_2 \in Q$, we define $(T_k(y_1,y_2))_{k \geq 1}$ as the increasing enumeration of the time components of $N_{0,\infty}(y_1) \cup N_{0,\infty}(y_2) \cup \{ (s,r,u) \in N \ \text{s.t.} \ r \geq 1/2 \}$, and set $T_0(y_1,y_2) := 0$ for convenience. This sequence thus enumerates all times at which at least one of the two trajectories $Y_{0,\cdot}(y_1)$, $Y_{0,\cdot}(y_2)$ is potentially involved in a merger caused by a jump of $N$. The following two lemmas show that trajectories starting from distinct points in $Q$ cannot merge outside $\{ T_k(y_1,y_2), k \geq 1 \}$. The first is a real-analytic estimate needed in the proof of the second.
\begin{lemma} \label{estimfm}
For any $M>0$ let $f_{M}:[0,1] \rightarrow \mathbb{R}$ be defined by $f_{M}(x):=-\log(x\vee e^{-M})$. Then there is a constant $C$ independent of $M$ such that for any $a \in (0,1)$ and $r \in (0,a \wedge (1/2))$, 
\begin{align}
\frac{a-r}{1-r} \left (f_{M} \left ( \frac{a-r}{1-r} \right )-f_{M}(a) \right ) \leq Cr. \label{majoincrementofdelta2}
\end{align}
\end{lemma}

\begin{proof}
Let $a \in (0,1)$ and $r \in (0,a \wedge (1/2))$. Distinguishing the three cases $a<e^{-M}$, $\frac{a-r}{1-r} \leq e^{-M} \leq a$, and $e^{-M}<\frac{a-r}{1-r}$ we get that in any case 
\begin{align}
\frac{a-r}{1-r} \left (f_{M} \left ( \frac{a-r}{1-r} \right )-f_{M}(a) \right ) \leq -\frac{a-r}{1-r} \log \left (\frac{a-r}{a(1-r)} \right ) \leq -2(a-r) \log \left (1-\frac{r}{a} \right ) =:u(a,r). \label{majoincrementofdelta1}
\end{align}
Let $C_1:=\sup_{x \in (0,1/2]} -\log(1-x)/x$ and $C_2:=\sup_{x \in (0,1]} -x\log(x)$. If $r \in (0,a/2]$ we have $u(a,r) \leq 2C_1r$. If $r \in (a/2,a)$ we have $u(a,r) \leq 2(a-r) \times C_2/(1-\frac{r}{a}) = 2 C_2 a \leq 4 C_2 r$. Setting $C:=\max \{ 2C_1, 4C_2 \}$ we get that, in both cases $r \in (a/2,a]$ and $r \in (a/2,a)$ we have $u(a,r) \leq Cr$. Combining with \eqref{majoincrementofdelta1} we get \eqref{majoincrementofdelta2}. 
\end{proof}

\begin{lemma} \label{finiteexpectoflog}
For any $y_1, y_2 \in Q$ with $y_1 \neq y_2$ and $k \geq 0$ we have 
\begin{align*}
& \mathbb{P} ( \exists t \in (T_k(y_1,y_2),T_{k+1}(y_1,y_2)) \ \text{s.t.} \ Y_{0,t}(y_1)=Y_{0,t}(y_2) | Y_{0,T_k(y_1,y_2)}(y_1)\neq Y_{0,T_k(y_1,y_2)}(y_2))=0, \\
& \mathbb{P} ( \exists t \in (T_k(y_1,y_2),T_{k+1}(y_1,y_2)] \ \text{s.t.} \ Y_{0,t-}(y_1)=Y_{0,t-}(y_2) | Y_{0,T_k(y_1,y_2)}(y_1)\neq Y_{0,T_k(y_1,y_2)}(y_2))=0. 
\end{align*} 
\end{lemma}

\begin{proof}
Fix $y_1,y_2 \in Q$ such that $y_1 \geq y_2$ and $k \geq 1$. The idea is to show that $\Delta_t := Y_{0,t}(y_1) - Y_{0,t}(y_2)$ cannot reach $0$ between two consecutive times $T_k(y_1,y_2)$, by controlling $f_M(\Delta_t) := -\log(\Delta_t \vee e^{-M})$ on this interval. By definition of the times $T_j(y_1,y_2)$, there is no jump $(s,r,u) \in N$ with $s \in (T_k(y_1,y_2),T_{k+1}(y_1,y_2))$ and $Y_{0,s-}(y_i)\in I_{r,u}$ for $i=1$ or $2$. Therefore, distinguishing cases as in Lemma \ref{estimmryu} for each jump $(s,r,u) \in N$ with $s \in (T_k(y_1,y_2),T_{k+1}(y_1,y_2))$, taking $z_1=Y_{0,s-}(y_1)$ and $z_2=Y_{0,s-}(y_2)$, we are always in cases (i) or (iv). Using this and \eqref{defflowybysde} we get that, almost surely, for any $t \in [T_k(y_1,y_2),T_{k+1}(y_1,y_2))$, $\Delta_t - \Delta_{T_k(y_1,y_2)}$ equals 
\begin{align*}
& \int_{(T_k(y_1,y_2),t] \times (0,1/2) \times (0,1)} \left ( \mr_{r,u}(Y_{0,s-}(y_1))-\mr_{r,u}(Y_{0,s-}(y_2))-Y_{0,s-}(y_1)+Y_{0,s-}(y_2) \right ) N(ds,dr,du) \\
= & \int_{(T_k(y_1,y_2),t] \times (0,1/2) \times (0,1)} \left ( \frac{r\Delta_{s-}}{1-r} \mathds{1}_{u \in (0,\frac{Y_{0,s-}(y_2)-r}{1-r}]\cup(\frac{Y_{0,s-}(y_1)}{1-r},1)} \right . \\
& \qquad \qquad \qquad \qquad \qquad \qquad \qquad \qquad \qquad \left . - \frac{r(1-\Delta_{s-})}{1-r} \mathds{1}_{u \in (\frac{Y_{0,s-}(y_2)}{1-r}, \frac{Y_{0,s-}(y_1)-r}{1-r}]} \right ) N(ds,dr,du). 
\end{align*} 
Fix $M, T \in (0,\infty) \cap \mathbb{Q}$ and set $f_{M}(x):=-\log(x\vee e^{-M})$. Since $(\Delta_t)_{t \geq 0}$ is given by an uncompensated Poisson stochastic integral and $f_M$ is Lipschitz, It{\^o}'s formula from \cite[Thm. II.5.1]{ikedawatanabe} extends to $f_M$. It gives that, almost surely, for any $t \in [T_k(y_1,y_2),T_{k+1}(y_1,y_2))$, 
\begin{align*}
& f_{M}(\Delta_t) - f_{M}(\Delta_{T_k(y_1,y_2)}) \\
= & \int_{(T_k(y_1,y_2),t] \times (0,1/2) \times (0,1)} \left ( f_{M} \left (\frac{\Delta_{s-}}{1-r} \right )-f_{M}(\Delta_{s-}) \right ) \mathds{1}_{u \in (0,\frac{Y_{0,s-}(y_2)-r}{1-r}]\cup(\frac{Y_{0,s-}(y_1)}{1-r},1)} N(ds,dr,du) \\
+ & \int_{(T_k(y_1,y_2),t] \times (0,1/2) \times (0,1)} \left ( f_{M} \left (\frac{\Delta_{s-}-r}{1-r} \right )-f_{M}(\Delta_{s-}) \right ) \mathds{1}_{u \in (\frac{Y_{0,s-}(y_2)}{1-r}, \frac{Y_{0,s-}(y_1)-r}{1-r}]} N(ds,dr,du). 
\end{align*} 
For any $a \in [0,1]$ and $r \in (0,1/2)$ we have $\frac{a-r}{1-r} \leq a \leq \frac{a}{1-r}$ so, since $f_{M}$ is non-increasing, we see that the integrand of the first integral is non-positive while the integrand of the second integral is non-negative. We thus get that almost surely, 
\begin{align}
\sup_{t \in [T_k(y_1,y_2),T_{k+1}(y_1,y_2) \wedge T)} f_{M}(\Delta_t) \leq f_{M}(\Delta_{T_k(y_1,y_2)}) + B_M(T), \label{majofmdeltat}
\end{align} 
where 
\begin{align*}
B_M(T) := \int_{(0,T] \times (0,1/2) \times (0,1)} \left ( f_{M} \left (\frac{\Delta_{s-}-r}{1-r} \right )-f_{M}(\Delta_{s-}) \right ) \mathds{1}_{u \in (\frac{Y_{0,s-}(y_2)}{1-r}, \frac{Y_{0,s-}(y_1)-r}{1-r}]} N(ds,dr,du). 
\end{align*} 
By the compensation formula and Fubini's theorem we get that $\mathbb{E}[B_M(T)]$ equals 
\begin{align}
& \int_{(0,T]} \int_{(0,1/2)} \mathbb{E} \left [ \int_{(0,1)} \left ( f_{M} \left (\frac{\Delta_{s-}-r}{1-r} \right ) -f_{M}(\Delta_{s-}) \right ) \mathds{1}_{u \in (\frac{Y_{0,s-}(y_2)}{1-r}, \frac{Y_{0,s-}(y_1)-r}{1-r}]} du \right ] r^{-2} \Lambda(dr) ds \nonumber \\
= & \int_{(0,T]} \int_{(0,1/2)} \mathbb{E} \left [ \frac{\Delta_{s-}-r}{1-r} \left ( f_{M} \left (\frac{\Delta_{s-}-r}{1-r} \right ) -f_{M}(\Delta_{s-}) \right ) \mathds{1}_{r<\Delta_{s-}} \right ] r^{-2} \Lambda(dr) ds \nonumber \\
\leq & CT \int_{(0,1/2)} r^{-1} \Lambda(dr) <\infty, \label{espbmt}
\end{align}
where the last inequality comes from Lemma \ref{estimfm} and the finiteness from \eqref{integassumption}. Let us denote by $\mathcal{E}_k$ the event $\{Y_{0,T_k(y_1,y_2)}(y_1)\neq Y_{0,T_k(y_1,y_2)}(y_2)\}$ and 
\begin{align*}
\mathcal{A}_k(T) := & \{ \exists t \in (T_k(y_1,y_2),T_{k+1}(y_1,y_2)\wedge T) \ \text{s.t.} \ Y_{0,t}(y_1)=Y_{0,t}(y_2) \} \\
\cup & \{ \exists t \in (T_k(y_1,y_2),T_{k+1}(y_1,y_2)\wedge T] \ \text{s.t.} \ Y_{0,t-}(y_1)=Y_{0,t-}(y_2) \}. 
\end{align*} 
Using the definitions of $\Delta_t$ and $f_{M}(\cdot)$, \eqref{majofmdeltat}, Markov inequality, and \eqref{espbmt} we get that $\mathbb{P} ( \mathcal{A}_k(T) | \mathcal{E}_k)$ is smaller than 
\begin{align*}
& \mathbb{P} \left ( \sup_{t \in [T_k(y_1,y_2),T_{k+1}(y_1,y_2) \wedge T)} f_{M}(\Delta_t) \geq M | \mathcal{E}_k \right ) \leq \mathds{1}_{f_{M}(\Delta_{T_k(y_1,y_2)}) \geq M/2} + \frac{\mathbb{P} ( B_M(T) \geq M/2 )}{\mathbb{P}(\mathcal{E}_k)} \\
\leq & \mathds{1}_{\Delta_{T_k(y_1,y_2)} \leq e^{-M/2}} + \frac{2\mathbb{E}[B_M(T)]}{M \times \mathbb{P}(\mathcal{E}_k)} \leq \mathds{1}_{\Delta_{T_k(y_1,y_2)} \leq e^{-M/2}} + \frac{2CT \int_{(0,1/2)} r^{-1} \Lambda(dr)}{M \times \mathbb{P}(\mathcal{E}_k)}. 
\end{align*}
Letting $M \to \infty$, we get that $\mathbb{P} ( \mathcal{A}_k(T) | \mathcal{E}_k) \to 0$, since $\mathds{1}_{\Delta_{T_k(y_1,y_2)} \leq e^{-M/2}} \to 0$ almost surely on $\mathcal{E}_k$ and the second term converges to $0$ as $M \to \infty$. Letting $T \to \infty$ then gives the result.
\end{proof}

The following lemma is a consequence of Lemmas \ref{ntxfinite} and \ref{finiteexpectoflog}, and establishes Proposition \ref{nonzerodiff} for trajectories starting from $Q$.
\begin{lemma} \label{nonzerodiffrationals}
We have 
\begin{align*}
& \mathbb{P} \left ( \forall t>0, \forall y_1 \neq y_2 \in Q, \ Y_{0,t}(y_1)\neq Y_{0,t}(y_2) \ \text{or} \ \exists (s,r,u) \in N \ \text{s.t.} \ s\leq t, \ y_1,y_2 \in Y_{0,s-}^{-1}(I_{r,u}) \right ) = 1. 
\end{align*}
The same property holds with $Y_{0,t}(\cdot)$ replaced by $Y_{0,t-}(\cdot)$ and $s\leq t$ replaced by $s<t$. 
\end{lemma}

\begin{proof}
Let $y_1, y_2 \in Q$ with $y_1 \neq y_2$ and $T(y_1, y_2) := \inf \{ t \geq 0, \ Y_{0,t}(y_1)=Y_{0,t}(y_2) \ \text{or} \ Y_{0,t-}(y_1)=Y_{0,t-}(y_2) \}$. By Lemma \ref{ntxfinite} the sequence $(T_k(y_1,y_2))_{k \geq 1}$ is well-defined almost surely and by Lemma \ref{finiteexpectoflog} we have almost surely $T(y_1, y_2) \in \{ T_k(y_1,y_2), k \geq 1 \} \cup \{ \infty \}$ and $Y_{0,T(y_1, y_2)-}(y_1)\neq Y_{0,T(y_1, y_2)-}(y_2)$ on $\{T(y_1, y_2)<\infty\}$. On $\{T(y_1, y_2)<\infty\}$, let us denote by $(s,r,u) \in N$ the jump at time $T(y_1, y_2)$ (then $s=T(y_1, y_2)$). Since $Y_{0,s-}(y_1)\neq Y_{0,s-}(y_2)$ and $Y_{0,s}(y_1)=Y_{0,s}(y_2)$, we see from \eqref{defflowybysde} that we have necessarily $Y_{0,s-}(y_1), Y_{0,s-}(y_2) \in I_{r,u}$. Since $Q$ is countable the result follows. 
\end{proof}

We can now conclude the proof of Proposition \ref{nonzerodiff}. 
\begin{proof}[Proof of Proposition \ref{nonzerodiff}]
Assume we are on the probability one event where the claims from Lemmas \ref{nonzerodiffrationals} and \ref{ntxfinite} hold true. By contradiction, we assume that there are $t>0$ and $y_1,y_2 \in [0,1]$ with $y_1<y_2$ such that $Y_{0,t}(y_1)=Y_{0,t}(y_2)$ (resp. $Y_{0,t-}(y_1)=Y_{0,t-}(y_2)$) and such that there is no jump $(s,r,u) \in N$ with $s\leq t$ (resp. $s<t$) for which both $y_1$ and $y_2$ belong to $Y_{0,s-}^{-1}(I_{r,u})$. Then, $Y_{0,t}([y_1,y_2])$ (resp. $Y_{0,t-}([y_1,y_2])$) is a singleton. 

We now construct by induction an infinite sequence of distinct jumps $(s_j, r_j, u_j)_{j \geq 1}$ in $N$ and a rational $q$ such that $q\in Y_{0,s_j-}^{-1}(I_{r_j,u_j})$ for all $j \geq 1$, contradicting the finiteness of $N_{0,t}(q)$ given by Lemma 
\ref{ntxfinite}. Fix $n_1> m_0:=3/(y_2-y_1)$ and choose $q_1, \tilde q_1 \in Q$ such that $q_1 \in (y_1,y_1+1/n_1)$ and $\tilde q_1 \in (y_2-1/n_1,y_2)$. By Lemma \ref{nonzerodiffrationals}, there is $(s_1,r_1,u_1) \in N$ such that $s_1\leq t$ (resp. $s_1<t$) and $q_1, \tilde q_1 \in Y_{0,s_1-}^{-1}(I_{r_1,u_1})$. We can choose a rational $q \in (y_1+1/m_0,y_2-1/m_0) \subset [q_1, \tilde q_1] \subset Y_{0,s_1-}^{-1}(I_{r_1,u_1})$. By assumption, it is not possible that both $y_1$ and $y_2$ belong to the interval $Y_{0,s_1-}^{-1}(I_{r_1,u_1})$ so there is $n_2>n_1$ such that $y_1+1/n_2< \inf Y_{0,s_1-}^{-1}(I_{r_1,u_1})$ or $y_2-1/n_2>\sup Y_{0,s_1-}^{-1}(I_{r_1,u_1})$. We can choose $q_2, \tilde q_2 \in Q$ such that $q_2 \in (y_1,y_1+1/n_2)$ and $\tilde q_2 \in (y_2-1/n_2,y_2)$ and proceed as before. Iterating this procedure (but with always the same $q$) we obtain the sought sequence $(s_j,r_j,u_j)_{j\geq 1}$. This contradicts the assumption that we are on the probability one event given by Lemma \ref{ntxfinite}. The result follows. 
\end{proof}

\begin{remark} \label{nonzerodiffrmk1}
Proposition \ref{nonzerodiff} and \eqref{collage} yield that, almost surely, if $Y_{0,t_2}(y_1)=Y_{0,t_2}(y_2)$ (resp. $Y_{0,t_2-}(y_1)=Y_{0,t_2-}(y_2)$) and $Y_{0,t_1}(y_1)\neq Y_{0,t_1}(y_2)$ for some $t_2>t_1\geq 0$ and $y_1,y_2 \in [0,1]$, then there is $(s,r,u) \in N$ such that $s\in (t_1,t_2]$ (resp. $s\in (t_1,t_2)$) and $y_1,y_2 \in Y_{0,s-}^{-1}(I_{r,u})$. 
\end{remark}

\section{From the flow to the $\Lambda$-coalescent: auxiliary results} \label{coaltrajy}

In this appendix we assume that \eqref{integassumption} holds true, take as given the results established in Appendix \ref{propflows}, and establish two sets of auxiliary results used in the construction of the $\Lambda$-coalescent from the flow $Y$. Section \ref{coaltrajysec} proves Lemma \ref{lawofpartbyblg1and20}, which identifies the partition process $(\pi^Y_t)_{t \geq 0}$ as a $\Lambda$-coalescent, and Lemma \ref{probspeevt}, which provides the lower bound on $\mathbb{P}(\mathcal{E}(t,t+w,k,\epsilon,\alpha))$ used in Section \ref{longtimelowerboundspe}. Section \ref{poisrep1ststepappend} establishes the auxiliary lemmas used in 
the proofs of Proposition \ref{poissrepstep1} and Theorem \ref{represwktviamut}.

\subsection{The partition process $(\pi^Y_t)_{t\geq 0}$ is a $\Lambda$-coalescent} \label{coaltrajysec}

In this subsection, we prove Lemmas \ref{lawofpartbyblg1and20} and \ref{probspeevt}. Both proofs rely on the structure of the process $(\pi^N_t)_{t\geq 0}$, a process of random partitions of $\mathbb{N}$ where two individuals $i$ and $j$ are declared equivalent at time $t$ if and only if there exists a jump of $N$ before time $t$ that explicitly caused their ancestral lines to merge. Formally, let $(U_i)_{i\geq 1}$ be as in the definition of $(\pi^Y_t)_{t\geq 0}$ (see Section \ref{model2}), we define $(\pi^N_t)_{t\geq 0}$ by the equivalence relation $i \sim_{\pi^N_t} j \Leftrightarrow \exists (s,r,u) \in N$ such that $s \in (0,t]$ and $U_i, U_j \in Y_{0,s-}^{-1}(I_{r,u})$, where $I_{r,u}$ is defined in \eqref{defiru}. For two partitions $P_1$ and $P_2$ of $\mathbb{N}$, we say that $P_2$ is a coagulation of $P_1$ and note $P_1 \leq P_2$ if each block of $P_2$ is a union of blocks from $P_1$. 
\begin{lemma} \label{nocontmerging}
We have almost surely that $\pi^Y_t=\pi^N_t$ for all $t\geq 0$. 
\end{lemma}
\begin{proof}
The property \eqref{collage} yields that we have almost surely $\pi^N_{t} \leq \pi^Y_{t}$ for all $t \geq 0$, and Proposition \ref{nonzerodiff} yields that we have almost surely $\pi^Y_{t} \leq \pi^N_{t}$ for all $t\geq 0$. 
\end{proof}
By Corollary \ref{calczonecollage}, almost surely, for all $t \geq 0$ there are infinitely many $i \geq 1$ such that $U_i \in C_t^c$, which implies that $\pi^Y_t$ (and therefore $\pi^N_t$) has infinitely many singleton blocks. We note that Lemma \ref{lawofpartbyblg1and20} is not used in the proof of Corollary \ref{calczonecollage} (which belongs to Section \ref{calcmeas}), so there is no circularity in using the fact of the infiniteness of blocks in the following proof. 
\begin{proof}[Proof of Lemma \ref{lawofpartbyblg1and20}]
Thanks to Lemma  \ref{nocontmerging}, we only need to prove the statement for $(\pi^N_t)_{t\geq 0}$. 
For any $t\geq 0$ we denote by $A^1_t, A^2_t,...$ the blocks of $\pi^N_t$ ordered by their lowest elements. We see from Lemma  \ref{nocontmerging} that $j_1 \sim_{\pi^N_t} j_2 \Leftrightarrow Y_{0,t}(U_{j_1})=Y_{0,t}(U_{j_2})$. For $i \geq 1$ we can thus set $U_i(t):=Y_{0,t}(U_j)$ where $j$ is any element of $A^i_t$. For any $(s,r,u)\in N$ we set $Z_{s,i}:=\mathds{1}_{U_i(s-)\in I_{r,u}}$. We see from the definition of $(\pi^N_t)_{t\geq 0}$ that merging events in that coalescent process only occur at jumping times $s \in J_N$ and that, for any $s \in J_N$, the non-empty blocks $A^i_{s-}$ involved in the merging are exactly those for which $Z_{s,i}=1$. We show below that the random set $\{ (s,r,(Z_{s,i})_{i\geq 1}) \}$ defines a Poisson random measure on $(0,\infty) \times (0,1) \times (\{0,1\}^{\mathbb{N}})$ with intensity measure $m(ds,dr,dz):=ds \otimes (\Ber(r)^{\otimes \mathbb{N}}(dz)) r^{-2} \Lambda(dr)$. By \cite[Cor. 3]{pitman1999} it will follow that $(\pi^N_t)_{t\geq 0}$ is a $\Lambda$-coalescent. 

Let us fix some $\eta\in(0,1/2)$ and let $(s^{\eta}_i,r^{\eta}_i,u^{\eta}_i)_{i\geq 1}$ be the enumeration of $\{(s,r,u) \in N, r>\eta\}$ such that $s^{\eta}_1<s^{\eta}_2<\ldots$ and, for convenience, we set $s^{\eta}_0:=0$. Let $\tilde Y^\eta_{0,\cdot}(\cdot)$ be the flow which retains only jumps of $N$ with $r$-component smaller or equal to $\eta$. By the discussion after Proposition \ref{propexistence}, for each $t$, the generalized inverse of $\tilde Y^\eta_{0,t}(\cdot)$ is a bridge. The law of $Y_{0,s^\eta_1-}(\cdot) = \tilde Y^{\eta}_{0,s^\eta_1}(\cdot)$ 
is the mixture $\int_0^\infty \mathcal{L}(\tilde Y^\eta_{0,t}(\cdot)) \Lambda((\eta,1)) e^{-\Lambda((\eta,1))t} dt$, where $\mathcal{L}$ denotes 
the law. By the definition of a bridge (see for example \cite[Sec. 2.1]{BLGI}), any mixture of laws of bridges is again the law of a bridge. We thus get that the generalized inverse of $Y_{0,s^\eta_1-}(\cdot)$ is a bridge $B$, even after conditioning with respect to $(s^{\eta}_1,r^{\eta}_1,u^{\eta}_1)$. We see that $B$, $\pi^N_{s^{\eta}_1-}$ and $(U_j(s^{\eta}_1-))_{j\geq 1}$ are as the bridge, the partition and the sequence considered in \cite[Lem. 2]{BLGI}. By that lemma we get that, conditionally on $(s^{\eta}_1,r^{\eta}_1,u^{\eta}_1)$, $(U_j(s^{\eta}_1-))_{j\geq 1} \sim \mathcal{U}([0,1])^{\otimes \mathbb{N}}$ so $(Z_{s^{\eta}_1,j})_{j \geq 1}\sim\Ber(r^{\eta}_1)^{\otimes \mathbb{N}}$. 
Then, conditionally on $(s^{\eta}_1,r^{\eta}_1)$, the generalized inverse of $\mr_{r^{\eta}_1,u^{\eta}_1}(\cdot)$ is a bridge $\tilde B$ independent of $(U_j(s^{\eta}_1-))_{j\geq 1}$. Let us define a partition $\pi$ of $\mathbb{N}$ by the equivalence relation $i \sim_{\pi} j \Leftrightarrow \mr_{r^{\eta}_1,u^{\eta}_1}(U_i(s^{\eta}_1-))=\mr_{r^{\eta}_1,u^{\eta}_1}(U_j(s^{\eta}_1-))$. We note from \eqref{collage} that $\{ j \geq 1, Z_{s^{\eta}_1,j}=1\}$ is a block of the partition $\pi$ and that, for all $j \geq 1$ such that $Z_{s^{\eta}_1,j}=0$, $\{j\}$ is a block of the partition $\pi$. In particular, $(Z_{s^{\eta}_1,j})_{j \geq 1}$ is a function of the partition $\pi$. We see that $\tilde B$, $\pi$ and $(U_j(s^{\eta}_1))_{j\geq 1}$ are as the bridge, the partition and the sequence considered in \cite[Lem. 2]{BLGI}. By that lemma we get that, conditionally on $(s^{\eta}_1,r^{\eta}_1)$,  the sequences $(U_j(s^{\eta}_1))_{j\geq 1}$ and $(Z_{s^{\eta}_1,j})_{j \geq 1}$ are independent and that $(U_j(s^{\eta}_1))_{j\geq 1} \sim \mathcal{U}([0,1])^{\otimes \mathbb{N}}$. Then, iterating the above arguments applied to $Y_{s^{\eta}_{i},s^{\eta}_{i+1}-}(\cdot)$ and $(U_j(s^{\eta}_i))_{j\geq 1}$ instead of $Y_{0,s^{\eta}_1-}(\cdot)$ and $(U_j)_{j\geq 1}$, and then to $\mr_{r^{\eta}_{i+1},u^{\eta}_{i+1}}(\cdot)$ and $(U_j(s^{\eta}_{i+1}-))_{j\geq 1}$ instead of $\mr_{r^{\eta}_1,u^{\eta}_1}(\cdot)$ and $(U_j(s^{\eta}_{1}-))_{j\geq 1}$, we get that, conditionally on $(s^{\eta}_i,r^{\eta}_i)_{i \geq 1}$, the sequences $(Z_{s^{\eta}_i,j})_{j \geq 1}$ are mutually independent, with $(Z_{s^{\eta}_i,j})_{j \geq 1}\sim\Ber(r^{\eta}_i)^{\otimes \mathbb{N}}$ for each $i \geq 1$. Combined with the fact that $(s^\eta_i, r^\eta_i)_{i \geq 1}$ are the atoms of a Poisson random measure with intensity $ds \otimes r^{-2} \Lambda(dr)$ restricted to $(0,\infty) \times (\eta,1)$, this shows that $\{(s,r,(Z_{s,i})_{i\geq 1})\} \cap \mathcal{D}_\eta$ is a Poisson random measure on $\mathcal{D}_{\eta}:=(0,\infty) \times (\eta,1) \times (\{0,1\}^{\mathbb{N}})$ with intensity $m(\mathcal{D}_\eta \cap \cdot)$. Since this holds for any choice of $\eta$, we conclude that  $\{ (s,r,(Z_{s,i})_{i\geq 1}) \}$ defines a Poisson random measure on $(0,\infty) \times (0,1) \times (\{0,1\}^{\mathbb{N}})$ with intensity measure $m(ds,dr,dz)$. 
\end{proof}

\begin{remark} \label{poispint}
As a byproduct of the above proof, the $\Lambda$-coalescent process $(\pi^N_t)_{t\geq 0}$ only has transitions at times $s \in J_N$ and, if $p$ is the $r$-component of the jump at time $s$, each block takes part in the merging event independently with probability $p$.
\end{remark}

\begin{proof}[Proof of Lemma \ref{probspeevt}]
Let $w>0$, and let $(U_i)_{i\geq 1}$ and $(\pi^N_t)_{t\geq 0}$ be as above. We consider the event $A$ where: 
\begin{enumerate}[(1)]
\item there are exactly $k$ jumps $(s_i,r_i,u_i)\in N$ with $r_i>\epsilon$ and $s_i \in (0,w]$, ordered as $s_1<\ldots<s_k$;
\item for each $i \in \{1,\ldots,k\}$, the block of $(\pi^N_t)_{t\geq 0}$ containing $i$ takes part in the merger event at time $s_i$ but is not involved in any other merger on $(0,w]\setminus\{s_i\}$.
\end{enumerate}
By properties of poisson point processes and Remark \ref{poispint}, we have $\mathbb{P}(A)>0$. On $A$, the blocks of $(\pi^N_t)_{t\geq 0}$ containing $1,\ldots,k$ are all distinct at time $w$. By Lemma  \ref{nocontmerging}, $\pi^Y_w = \pi^N_w$ almost surely, so the $Y_{0,w}(U_i)$ are distinct for $i \in \{1,\ldots,k\}$. Moreover, by \eqref{collage}, $Y_{0,w}(U_i) = Y_{s_i,w}(u_i)$ for each $i$. Therefore $A \subset \mathcal{E}(0,w,k,\epsilon,\infty)$, which shows that $\mathbb{P}(\mathcal{E}(0,w,k,\epsilon,\infty)) \geq \mathbb{P}(A) > 0$.

Since $\mathbb{P}(S_w > \alpha) \to 0$ as $\alpha \to \infty$, we can choose $\alpha > 0$ large enough so that $\mathbb{P}(S_w > \alpha) < \mathbb{P}(A)$, which gives 
\begin{align*}
\mathbb{P}(\mathcal{E}(0,w,k,\epsilon,\alpha)) \geq \mathbb{P}(\mathcal{E}(0,w,k,\epsilon,\infty)) - \mathbb{P}(S_w > \alpha) \geq \mathbb{P}(A) - \mathbb{P}(S_w>\alpha) > 0.
\end{align*}
Finally, since $\mathcal{E}(t,t+w,k,\epsilon,\alpha)$ is independent of $\mathcal{F}_t$ by definition \eqref{defspeevinter}, and since $N$ has stationary increments, we have for all $t \geq 0$,
\begin{align*}
\mathbb{P}(\mathcal{E}(t,t+w,k,\epsilon,\alpha)|\mathcal{F}_t) = 
\mathbb{P}(\mathcal{E}(0,w,k,\epsilon,\alpha)) =: c(w,k,\epsilon,\alpha) > 0,
\end{align*}
which concludes the proof.
\end{proof}

\subsection{Auxiliary lemmas used in Sections \ref{poisrep} and \ref{itoform}} \label{poisrep1ststepappend}

In this subsection, we establish four auxiliary lemmas used in Sections \ref{poisrep} and \ref{itoform}. Lemmas \ref{avoidboundaries} and \ref{avoidboundaries2} show that, almost surely, significant trajectories of the flow $Y$ never land on the boundary of the intervals $I_{r,u}$ (associated to atoms of $N$), and that the preimage of the interior of $I_{r,u}$ under $Y_{0,s-}(\cdot)$ coincides with the interior of the preimage of $I_{r,u}$. Lemma \ref{diffcountable} is a consequence of Lemmas \ref{avoidboundaries} and \ref{avoidboundaries2}. Lemma \ref{piecewisecte} describes the structure of an interval-valued process $(J_t(y))_{t \geq 0}$ of intervals on which the flow $Y$ is constant. Recall that $Q:=[0,1] \cap \mathbb{Q}$. 

\begin{lemma} \label{avoidboundaries}
Almost surely, for all $(s_1,r_1,u_1), (s_2,r_2,u_2) \in N$ with $s_1<s_2$ we have $Y_{s_1,s_2-}(u_1) \notin I_{r_2,u_2}\setminus I_{r_2,u_2}^{\mathrm{o}}$, and for all $(s,r,u) \in N$ and $y \in Q$, $Y_{0,s-}(y) \notin I_{r,u}\setminus I_{r,u}^{\mathrm{o}}$. 
\end{lemma}

\begin{proof}
The idea is to condition $N$ with respect to its first two coordinates, so that hitting a boundary reduces to a null probability event by independence. Formally, let $\tilde N$ be a Poisson point process on $(0,\infty)\times (0,1)$ with intensity measure $dt \otimes r^{-2} \Lambda(dr)$ and let $(V_i)_{i \geq 1} \sim \mathcal{U}([0,1])^{\otimes \mathbb{N}}$. We set a deterministic total order on $(0,\infty)\times (0,1)$ such that, almost surely, the elements of $\tilde N$ can be enumerated in a sequence that respects that order. We denote by $n(s,r)$ the position of the element $(s,r)\in \tilde N$. Note that $\{ (s,r,V_{n(s,r)}), (s,r) \in \tilde N \}$ is equal in law to $N$ so, in this proof, we assume that $N$ is built in this way. We work conditionally on $\tilde N$ and pick $(s_1,r_1), (s_2,r_2) \in \tilde N$ with $s_1<s_2$. Then 
\begin{align*} 
& \left \{ Y_{s_1,s_2-}(V_{n(s_1,r_1)}) \in I_{r_2,V_{n(s_2,r_2)}}\setminus I_{r_2,V_{n(s_2,r_2)}}^{\mathrm{o}} \right \} \\
& \qquad \qquad \qquad = \left \{ Y_{s_1,s_2-}(V_{n(s_1,r_1)}) \in \{(1-r_2)V_{n(s_2,r_2)}, (1-r_2)V_{n(s_2,r_2)}+r_2\} \right \}. 
\end{align*}
Since $V_{n(s_2,r_2)}$ is independent of $(Y_{s_1,s_2-}(V_{n(s_1,r_1)}),r_2)$, the above event has null probability. Intersecting over all pairs $(s_1,r_1),(s_2,r_2) \in \tilde N$ with $s_1 < s_2$ preserves the null probability, and integrating over the law of $\tilde N$ yields the first statement.

For the second statement we work conditionally on $\tilde N$, fix $(s,r) \in \tilde N$, and set $u:=V_{n(s,r)}$. The condition $Y_{0,s-}(y) \in I_{r,u} \setminus I_{r,u}^{\mathrm{o}}$ for $y \in Q$ is equivalent to $u \in \{ Y_{0,s-}(y)/(1-r), y \in Q \} \cup \{ (Y_{0,s-}(y)-r)/(1-r), y \in Q \}$. Since $u = V_{n(s,r)}$ is independent of $(Y_{0,s-}(y))_{y \in Q}$ and $Q$ is countable, this event has null probability. Intersecting over all $(s,r) \in \tilde N$ and integrating over $\tilde N$ yields the second statement.
\end{proof}

\begin{lemma} \label{avoidboundaries2}
Almost surely, for all $(s,r,u) \in N$ we have $Y_{0,s-}^{-1}(I_{r,u}^{\mathrm{o}})=(Y_{0,s-}^{-1}(I_{r,u}))^{\mathrm{o}}$. 
\end{lemma}

\begin{proof}
Let us assume that we are on the probability one event provided by Lemma \ref{avoidboundaries}, and assume that there is $(s,r,u) \in N$ such that $Y_{0,s-}^{-1}(I_{r,u}^{\mathrm{o}})\neq (Y_{0,s-}^{-1}(I_{r,u}))^{\mathrm{o}}$. By the continuity of $Y_{0,s-}(\cdot)$ (Proposition \ref{propexistence}(iii)) and the fact that $I_{r,u}^{\mathrm{o}}$ is open, we have $Y_{0,s-}^{-1}(I_{r,u}^{\mathrm{o}}) \subset (Y_{0,s-}^{-1}(I_{r,u}))^{\mathrm{o}}$. Therefore, there exists a nonempty open interval $J \subset (Y_{0,s-}^{-1}(I_{r,u}))^{\mathrm{o}} \setminus Y_{0,s-}^{-1}(I_{r,u}^{\mathrm{o}})$. We have $Y_{0,s-}(J) \subset I_{r,u} \setminus I_{r,u}^{\mathrm{o}}$. Then, for $y \in J \cap \mathbb{Q}$ we have $Y_{0,s-}(y) \in I_{r,u} \setminus I_{r,u}^{\mathrm{o}}$, which contradicts the assumption that we are one the probability one event provided by Lemma \ref{avoidboundaries}. 
\end{proof}

\begin{lemma} \label{diffcountable}
Almost surely, the set $\mathcal{S} := \cup_{(s,r,u) \in N} (Y_{0,s-}^{-1}(I_{r,u}) \setminus Y_{0,s-}^{-1}(I_{r,u}^{\mathrm{o}}))$ is countable and $\mathbb{Q} \cap \mathcal{S} = \emptyset$. 
\end{lemma}

\begin{proof}
Let us assume that we are on the probability one event provided by Lemmas \ref{avoidboundaries} and \ref{avoidboundaries2}. Since each set $Y_{0,s-}^{-1}(I_{r,u})$ is a closed interval, Lemma \ref{avoidboundaries2} shows that each set $Y_{0,s-}^{-1}(I_{r,u}) \setminus Y_{0,s-}^{-1}(I_{r,u}^{\mathrm{o}})$ has exactly two elements, yielding the first claim. Let $y \in Q$, then $y \in \mathcal{S} \Rightarrow \exists (s,r,u) \in N \ \text{s.t.} \ Y_{0,s-}(y) \in I_{r,u} \setminus I_{r,u}^{\mathrm{o}}$. 
By Lemma \ref{avoidboundaries}, the latter does not occur. This proves the second claim. 
\end{proof}

For any $t \geq 0$ and $y \in [0,1]$ we define the interval 
\begin{align*}
J_{t}(y):=Y_{0,t}^{-1}(\{Y_{0,t}(y)\})=\{ z \in [0,1] \ \text{s.t.} \ Y_{0,t}(z)=Y_{0,t}(y)\}, 
\end{align*}
which represents the set of all individuals that share the same ancestor as $y$ at time $t$. By Proposition \ref{propexistence}(iii), $Y_{0,t}(\cdot)$ is non-decreasing and continuous, so $J_t(y)$ is a closed interval for all $t \geq 0$ and 
$y \in [0,1]$, possibly reduced to the singleton $\{y\}$. We note that for each $y \in C_t$, $J_{t}(y)^{\mathrm{o}}$ is the open connected component of $C_t$ containing $y$ so, in particular, $J_{t}(y)^{\mathrm{o}} \in \{ \mathcal{O}_k(t), k \geq 1 \}$ and $|J_{s}(y)|=|J_{s}(y)^{\mathrm{o}}| \in \{ W_k(s), k\geq 1\}$. 
\begin{lemma} \label{piecewisecte}
Almost surely, for every $y \in Q$, the interval-valued process $(J_{t}(y))_{t \geq 0}$ is non-decreasing, piecewise constant and right-continuous, increase times $s$ of $(J_{t}(y))_{t \geq 0}$ are exactly the time components of jumps $(s,r,u) \in N$ such that $y \in Y^{-1}_{0,s-}(I_{r,u}^{\mathrm{o}})$. This claim is also true when "for every $y \in Q$" is replaced by "for almost every $y \in [0,1]$". 
\end{lemma}

\begin{proof}
We assume that we are on the probability one events from Proposition \ref{nonzerodiff}, \eqref{collage}, Remark \ref{nonzerodiffrmk1}, Lemma \ref{ntxfinite} and Lemma \ref{diffcountable}. Let $\mathcal{U}\subset [0,1]$ be the set of measure one from Lemma \ref{ntxfinite}. 

\textit{Step 1: non-decreasing.} Let $t_2>t_1 \geq 0$ and $y \in [0,1]$. If $z \in J_{t_1}(y)$ then $Y_{0,t_1}(z)=Y_{0,t_1}(y)$ so, by Proposition \ref{nonzerodiff}, either $z=y$, in which case $Y_{0,t_2}(z)=Y_{0,t_2}(y)$, or $\exists (s,r,u) \in N$ such that $s\leq t_1$ and $y,z \in Y_{0,s-}^{-1}(I_{r,u})$. In this case, we get from \eqref{collage} that $Y_{0,t_2}(z)=Y_{0,t_2}(y)$. Therefore $z \in J_{t_2}(y)$. This proves that $(J_{t}(y))_{t \geq 0}$ is non-decreasing for all $y \in [0,1]$. 

\textit{Step 2: piecewise constant with increase times determined by $N$.} Let us show that, for $y \in Q$ (resp. $y \in \mathcal{U} \cap (\mathcal{S})^c$, where $\mathcal{S}$ is defined in Lemma \ref{diffcountable}), $J_{t_1}(y) \neq J_{t_2}(y)$ implies that $N_{t_1,t_2}(y)$ is non-empty. If $J_{t_1}(y) \neq J_{t_2}(y)$ for such a choice of $y$ and some $t_2>t_1\geq 0$, then let $z \in J_{t_2}(y) \setminus J_{t_1}(y)$. By definition of $J_{\cdot}(y)$ we have $Y_{0,t_1}(z)\neq Y_{0,t_1}(y)$ and $Y_{0,t_2}(z)=Y_{0,t_2}(y)$. By Remark \ref{nonzerodiffrmk1} we get that there is $(s,r,u) \in N$ with $s \in (t_1,t_2]$ and $y \in Y_{0,s-}^{-1}(I_{r,u})$ (so $(s,r,u) \in N_{t_1,t_2}(y)$) and, by Lemma \ref{diffcountable}, we even have $y \in Y^{-1}_{0,s-}(I_{r,u}^{\mathrm{o}})$. Thus $N_{t_1,t_2}(y)$ is non-empty. By Lemma \ref{ntxfinite}, the sets $N_{t_1,t_2}(y)$ are finite for all $0\leq t_1<t_2$ so let $T_1<T_2<\dots$ be the ordered sequence of time components of elements $(s,r,u) \in N_{0,\infty}(y)$ (and set $T_0:=0$ for convenience). We thus get that $(J_{t}(y))_{t \geq 0}$ is constant on intervals $[T_i,T_{i+1})$, concluding the proof. 
\end{proof}

\section{Proof of Remark \ref{propnlambd}} \label{proofestnlbd}

We prove the two bounds stated in Remark \ref{propnlambd}. It has been justified after \eqref{defnlambda} that $N(\Lambda)>2$. For $k \geq 3$, using the definition of $\lambda_k(\Lambda)$ and the strict Bernoulli inequality $(1-r)^{k-1} > 1-(k-1)r$, which holds for all $r \in (0,1)$,
\begin{align}
\lambda_k(\Lambda) & = \int_{(0,1)} [1-(1-r)^{k-1} (1+(k-1)r)] r^{-2} \Lambda(dr) \nonumber \\
& < \int_{(0,1)} [1-(1-(k-1)r) (1+(k-1)r)] r^{-2} \Lambda(dr) = (k-1)^2 \lambda_2(\Lambda), \label{majodk}
\end{align}
where the last equality uses $\lambda_2(\Lambda) = \Lambda((0,1))$. We now show that $N(\Lambda) > M(\Lambda)$, where 
$M(\Lambda):=\sqrt{\phi_S(1)/\lambda_2(\Lambda)}$. Set $K := \lfloor 1+M(\Lambda) \rfloor$. Then $(K-1)^2 \lambda_2(\Lambda)\leq \phi_S(1)$ so, by \eqref{majodk}, $\lambda_K(\Lambda)<\phi_S(1)$. By the definition of $N(\Lambda)$ in \eqref{defnlambda} the latter implies $K < N(\Lambda)$. Since $K>M(\Lambda)$ we get $N(\Lambda) > M(\Lambda)$, which concludes the proof. 

\subsection*{Acknowledgment}
This work was supported by Beijing Natural Science Foundation, project number IS24067. 

\bibliographystyle{plain}
\bibliography{thbiblio}

\begin{thebibliography}{10}

\bibitem{romain2013}
Romain Abraham and Jean-Fran\c cois Delmas.
\newblock A construction of a $\beta$-coalescent via the pruning of binary
  trees.
\newblock {\em Journal of Applied Probability}, 50(3):772--790, 2013.

\bibitem{romain2015}
Romain Abraham and Jean-Fran\c cois Delmas.
\newblock $\beta$-coalescents and stable {G}alton-{W}atson trees.
\newblock {\em ALEA, Lat. Am. J. Probab. Math. Stat.}, 12:451--476, 2015.

\bibitem{dembozeitouni}
Ofer~Zeitouni Amir~Dembo.
\newblock {\em Large Deviations Techniques and Applications}.
\newblock Stochastic Modelling and Applied Probability. Springer Berlin,
  Heidelberg, 2009.

\bibitem{Apple09}
David Applebaum.
\newblock {\em L\'{e}vy processes and stochastic calculus}, volume 116 of {\em
  Cambridge Studies in Advanced Mathematics}.
\newblock Cambridge University Press, Cambridge, second edition, 2009.

\bibitem{BLP15}
M\'{a}ty\'{a}s Barczy, Zenghu Li, and Gyula Pap.
\newblock Yamada--{W}atanabe results for stochastic differential equations with
  jumps.
\newblock {\em Int. J. Stoch. Anal.}, pages Art. ID 460472, 23, 2015.

\bibitem{10.1214/EJP.v13-494}
Anne-Laure Basdevant and Christina Goldschmidt.
\newblock Asymptotics of the allele frequency spectrum associated with the
  {B}olthausen-{S}znitman coalescent.
\newblock {\em Electronic Journal of Probability}, 13(none):486 -- 512, 2008.

\bibitem{ref17surveyb}
Julien Berestycki, Nathana{\"e}l Berestycki, and Vlada Limic.
\newblock The {$\Lambda$}-coalescent speed of coming down from infinity.
\newblock {\em The Annals of Probability}, 38(1):207 -- 233, 2010.

\bibitem{10.1214/13-AIHP546}
Julien Berestycki, Nathana{\"e}l Berestycki, and Vlada Limic.
\newblock Asymptotic sampling formulae for {$\Lambda$}-coalescents.
\newblock {\em Annales de l'Institut Henri Poincar\'e, Probabilit\'es et
  Statistiques}, 50(3):715 -- 731, 2014.

\bibitem{ref18surveyb}
Julien Berestycki, Nathana{\"e}l Berestycki, and Vlada Limic.
\newblock A small-time coupling between {$\Lambda$}-coalescents and branching
  processes.
\newblock {\em The Annals of Applied Probability}, 24(2):449 -- 475, 2014.

\bibitem{ref20surveyb}
Julien Berestycki, Nathana{\"e}l Berestycki, and Jason Schweinsberg.
\newblock Beta-coalescents and continuous stable random trees.
\newblock {\em The Annals of Probability}, 35(5):1835 -- 1887, 2007.

\bibitem{ref19surveyb}
Julien Berestycki, Nathana{\"e}l Berestycki, and Jason Schweinsberg.
\newblock Small-time behavior of beta coalescents.
\newblock {\em Annales de l'Institut Henri Poincar\'e, Probabilit\'es et
  Statistiques}, 44(2):214 -- 238, 2008.

\bibitem{Berestycki2009}
Nathana{\"e}l Berestycki.
\newblock Recent progress in coalescent theory, volume 16 of ensaios
  matem{\'a}ticos [mathematical surveys].
\newblock {\em Sociedade Brasileira de Matem{\'a}tica, Rio de Janeiro}, 51,
  2009.

\bibitem{lecturebertoin2010}
Jean Bertoin.
\newblock Exchangeable coalescents.
\newblock {\em Nachdiplom Lectures, ETH Z\"urich}, 2010.

\bibitem{ref29surveyb}
Jean Bertoin and Jean-Fran\c cois {Le Gall}.
\newblock The {B}olthausen-{S}znitman coalescent and the genealogy of
  continuous-state branching processes.
\newblock {\em Probability Theory and Related Fields}, 117(2):249--266, 2000.

\bibitem{BLGI}
Jean Bertoin and Jean-Fran\c cois {Le Gall}.
\newblock Stochastic flows associated to coalescent processes.
\newblock {\em Probability Theory and Related Fields}, 126(2):261--288, 2003.

\bibitem{BERTOIN2005307}
Jean Bertoin and Jean-Fran\c cois {Le Gall}.
\newblock Stochastic flows associated to coalescent processes ii: Stochastic
  differential equations.
\newblock {\em Annales de l'Institut Henri Poincare (B) Probability and
  Statistics}, 41(3):307--333, 2005.

\bibitem{BLGIII}
Jean Bertoin and Jean-Fran\c cois {Le Gall}.
\newblock {Stochastic flows associated to coalescent processes. III. Limit
  theorems}.
\newblock {\em Illinois Journal of Mathematics}, 50(1-4):147 -- 181, 2006.

\bibitem{ref36surveyb}
Matthias Birkner, Jochen Blath, Marcella Capaldo, Alison Etheridge, Martin
  M{\"o}hle, Jason Schweinsberg, and Anton Wakolbinger.
\newblock Alpha-stable branching and beta-coalescents.
\newblock {\em Electronic Journal of Probability}, 10(none):303 -- 325, 2005.

\bibitem{10.1214/22-EJP739}
Fernando Cordero, Adri{\'a}n~Gonz{\'a}lez Casanova, Jason Schweinsberg, and
  Maite Wilke-Berenguer.
\newblock {$\Lambda$}-coalescents arising in a population with dormancy.
\newblock {\em Electronic Journal of Probability}, 27(none):1 -- 34, 2022.

\bibitem{cordhumvech2022}
Fernando Cordero, Sebastian Hummel, and Gr{\'e}goire V{\'e}chambre.
\newblock {$\Lambda$-Wright-Fisher processes with general selection and
  opposing environmental effects: Fixation and coexistence}.
\newblock {\em Ann. Appl. Prob.}, 35(1):393 -- 457, 2025.

\bibitem{10.1214/10-AOP629}
Donald~A. Dawson and Zenghu Li.
\newblock {Stochastic equations, flows and measure-valued processes}.
\newblock {\em The Annals of Probability}, 40(2):813 -- 857, 2012.

\bibitem{10.1214/07-AAP476}
Jean-Fran{\c{c}}ois Delmas, Jean-St{\'e}phane Dhersin, and Arno Siri-Jegousse.
\newblock Asymptotic results on the length of coalescent trees.
\newblock {\em The Annals of Applied Probability}, 18(3):997 -- 1025, 2008.

\bibitem{10.1214/19-EJP354}
Christina~S. Diehl and G{\"o}tz Kersting.
\newblock External branch lengths of {$\Lambda$}-coalescents without a dust
  component.
\newblock {\em Electronic Journal of Probability}, 24(none):1 -- 36, 2019.

\bibitem{10.1214/19-AAP1462}
Christina~S. Diehl and G{\"o}tz Kersting.
\newblock Tree lengths for general {$\Lambda $}-coalescents and the asymptotic
  site frequency spectrum around the {B}olthausen-{S}znitman coalescent.
\newblock {\em The Annals of Applied Probability}, 29(5):2700 -- 2743, 2019.

\bibitem{donnellykurtz1999}
Peter Donnelly and Thomas~G. Kurtz.
\newblock Particle representations for measure-valued population models.
\newblock {\em The Annals of Probability}, 27(1):166 -- 205, 1999.

\bibitem{DRMOTA20071404}
Michael Drmota, Alex Iksanov, Martin Moehle, and Uwe Roesler.
\newblock Asymptotic results concerning the total branch length of the
  {B}olthausen-{S}znitman coalescent.
\newblock {\em Stochastic Processes and their Applications},
  117(10):1404--1421, 2007.

\bibitem{D08}
Richard Durrett.
\newblock {\em Probability {M}odels for {DNA} {S}equence {E}volution}.
\newblock Springer, New York, 2nd edition, 2008.

\bibitem{DURRETT20051628}
Rick Durrett and Jason Schweinsberg.
\newblock A coalescent model for the effect of advantageous mutations on the
  genealogy of a population.
\newblock {\em Stochastic Processes and their Applications},
  115(10):1628--1657, 2005.

\bibitem{fractal}
Gerald Edgar.
\newblock {\em Measure, {T}opology, and {F}ractal {G}eometry}.
\newblock Undergraduate Texts in Mathematics. Springer New York, NY, second
  edition, 2008.

\bibitem{foucartmamallein2019}
Cl{\'e}ment Foucart, Chunhua Ma, and Bastien Mallein.
\newblock Coalescences in continuous-state branching processes.
\newblock {\em Electronic Journal of Probability}, 24(none):1 -- 52, 2019.

\bibitem{foucartmohle2022}
Cl\'ement Foucart and Martin M\"ohle.
\newblock Asymptotic behaviour of ancestral lineages in subcritical
  continuous-state branching populations.
\newblock {\em Stochastic Processes and their Applications}, 150:510--531,
  2022.

\bibitem{10.1214/20-AAP1641}
F{\'e}lix Foutel-Rodier, Amaury Lambert, and Emmanuel Schertzer.
\newblock {Exchangeable coalescents, ultrametric spaces, nested
  interval-partitions: A unifying approach}.
\newblock {\em The Annals of Applied Probability}, 31(5):2046 -- 2090, 2021.

\bibitem{mohle2009}
Fabian Freund and Martin M\"ohle.
\newblock On the number of allelic types for samples taken from exchangeable
  coalescents with mutation.
\newblock {\em Advances in Applied Probability}, 41(4):1082--1101, 2009.

\bibitem{mohle2016}
Florian Gaise and Martin M\"ohle.
\newblock On the block counting process and the fixation line of exchangeable
  coalescents.
\newblock {\em ALEA, Lat. Am. J. Probab. Math. Stat.}, 13:809--833, 2016.

\bibitem{ref27ofsurvey}
Alexander Gnedin, Alexander Iksanov, and Alexander Marynych.
\newblock On {$\Lambda$}-coalescents with dust component.
\newblock {\em Journal of Applied Probability}, 48(4):1133--1151, 2011.

\bibitem{surveylambdacoal}
Alexander Gnedin, Alexander Iksanov, and Alexander Marynych.
\newblock {$\Lambda$}-coalescents: a survey.
\newblock {\em Journal of Applied Probability}, 51A:23--40, 2014.

\bibitem{betacoalgimm}
Alexander Gnedin, Alexander Iksanov, Alexander Marynych, and Martin M\"ohle.
\newblock On asymptotics of the beta coalescents.
\newblock {\em Advances in Applied Probability}, 46(2):496--515, 2014.

\bibitem{10.1214/EJP.v10-265}
Christina Goldschmidt and James Martin.
\newblock Random recursive trees and the {B}olthausen-{S}znitman coalesent.
\newblock {\em Electronic Journal of Probability}, 10(none):718 -- 745, 2005.

\bibitem{treevalued2013}
Andreas Greven, Peter Pfaffelhuber, and Anita Winter.
\newblock Tree-valued resampling dynamics martingale problems and applications.
\newblock {\em Probability Theory and Related Fields}, 155(3):789 -- 838, 2013.

\bibitem{Griffiths2014}
Robert~C. Griffiths.
\newblock The {$\Lambda$}-{F}leming-{V}iot process and a connection with
  {W}right-{F}isher diffusion.
\newblock {\em Advances in Applied Probability}, 46(4):1009--1035, 2014.

\bibitem{10.1214/18-EJP166}
Stephan Gufler.
\newblock {Pathwise construction of tree-valued Fleming-Viot processes}.
\newblock {\em Electronic Journal of Probability}, 23(none):1 -- 58, 2018.

\bibitem{10.3150/10-BEJ312}
B{\'e}n{\'e}dicte Haas and Gr{\'e}gory Miermont.
\newblock Self-similar scaling limits of non-increasing {M}arkov chains.
\newblock {\em Bernoulli}, 17(4):1217 -- 1247, 2011.

\bibitem{HENARD20132054}
Olivier H\'enard.
\newblock Change of measure in the lookdown particle system.
\newblock {\em Stochastic Processes and their Applications}, 123(6):2054--2083,
  2013.

\bibitem{10.1214/14-AAP1077}
Olivier H\'enard.
\newblock {The fixation line in the ${\Lambda}$-coalescent}.
\newblock {\em The Annals of Applied Probability}, 25(5):3007 -- 3032, 2015.

\bibitem{Huillet2014}
Thierry Huillet.
\newblock Pareto genealogies arising from a {P}oisson branching evolution model
  with selection.
\newblock {\em Journal of Mathematical Biology}, 68(3):727 -- 761, 2014.

\bibitem{ikedawatanabe}
N.~Ikeda and S.~Watanabe.
\newblock {\em Stochastic Differential Equations and Diffusion Processes}.
\newblock Cambridge university press, 1981.

\bibitem{10.1214/ECP.v12-1253}
Alex Iksanov and Martin M{\"o}hle.
\newblock A probabilistic proof of a weak limit law for the number of cuts
  needed to isolate the root of a random recursive tree.
\newblock {\em Electronic Communications in Probability}, 12(none):28 -- 35,
  2007.

\bibitem{10.1214/11-AAP827}
G{\"o}tz Kersting.
\newblock The asymptotic distribution of the length of {B}eta-coalescent trees.
\newblock {\em The Annals of Applied Probability}, 22(5):2086 -- 2107, 2012.

\bibitem{10.1214/EJP.v19-3332}
G{\"o}tz Kersting, Jason Schweinsberg, and Anton Wakolbinger.
\newblock {The evolving beta coalescent}.
\newblock {\em Electronic Journal of Probability}, 19(none):1 -- 27, 2014.

\bibitem{kerstingschweinsbergwakolbinger2018}
G{\"o}tz Kersting, Jason Schweinsberg, and Anton Wakolbinger.
\newblock The size of the last merger and time reversal in
  {$\Lambda$}-coalescents.
\newblock {\em Annales de l'Institut Henri Poincar\'e, Probabilit\'es et
  Statistiques}, 54(3):1527 -- 1555, 2018.

\bibitem{kerstingwakolbinger2018}
G\"otz Kersting and Anton Wakolbinger.
\newblock The time to absorption in {$\Lambda$}-coalescents.
\newblock {\em Advances in Applied Probability}, 50(A):177--190, 2018.

\bibitem{kerstingwakolbinger2020}
G\"otz Kersting and Anton Wakolbinger.
\newblock Probabilistic aspects of {$\Lambda$}-coalescents in equilibrium and
  in evolution.
\newblock {\em Probabilistic Structures in Evolution, E. Baake, A. Wakolbinger,
  eds. EMS}, 2020.

\bibitem{Kurtz2011}
Thomas~G. Kurtz.
\newblock Equivalence of stochastic equations and martingale problems.
\newblock In {\em Stochastic analysis 2010}, pages 113--130. Springer,
  Heidelberg, 2011.

\bibitem{labbe2014a}
Cyril Labb{\'e}.
\newblock {From flows of $\Lambda$-Fleming-Viot processes to lookdown processes
  via flows of partitions}.
\newblock {\em Electronic Journal of Probability}, 19(none):1 -- 49, 2014.

\bibitem{10.3150/17-BEJ971}
Amaury Lambert and Emmanuel Schertzer.
\newblock Recovering the {B}rownian coalescent point process from the {K}ingman
  coalescent by conditional sampling.
\newblock {\em Bernoulli}, 25(1):148 -- 173, 2019.

\bibitem{10.1214/13-AOP902}
Vlada Limic and Anna Talarczyk.
\newblock Second-order asymptotics for the block counting process in a class of
  regularly varying ${\Lambda}$-coalescents.
\newblock {\em The Annals of Probability}, 43(3):1419 -- 1455, 2015.

\bibitem{bj/1141136647}
Martin M\"ohle.
\newblock On sampling distributions for coalescent processes with simultaneous
  multiple collisions.
\newblock {\em Bernoulli}, 12(1):35 -- 53, 2006.

\bibitem{mohle2010}
Martin M\"ohle.
\newblock Asymptotic results for coalescent processes without proper
  frequencies and applications to the two-parameter {P}oisson-{D}irichlet
  coalescent.
\newblock {\em Stochastic Processes and their Applications},
  120(11):2159--2173, 2010.

\bibitem{mohle2021}
Martin M\"ohle.
\newblock The rate of convergence of the block counting process of exchangeable
  coalescents with dust.
\newblock {\em ALEA, Lat. Am. J. Probab. Math. Stat.}, 18:1195--1220, 2021.

\bibitem{10.1214/aop/1015345761}
Martin M{\"o}hle and Serik Sagitov.
\newblock A classification of coalescent processes for haploid exchangeable
  population models.
\newblock {\em The Annals of Probability}, 29(4):1547 -- 1562, 2001.

\bibitem{pitman1999}
Jim Pitman.
\newblock Coalescents with multiple collisions.
\newblock {\em The Annals of Probability}, 27(4):1870 -- 1902, 1999.

\bibitem{lecturepitman2006}
Jim Pitman.
\newblock {\em Combinatorial Stochastic Processes, {E}cole d'{E}t\'e de
  Probabilit\'es de {S}aint-{F}lour XXXII - 2002}.
\newblock Lecture Notes in Mathematics. Springer Berlin, 2006.

\bibitem{sagitov1999}
Serik Sagitov.
\newblock The general coalescent with asynchronous mergers of ancestral lines.
\newblock {\em Journal of Applied Probability}, 36(4):1116--1125, 1999.

\bibitem{KenIti1999}
K.-I. Sato.
\newblock {\em L{\'e}vy processes and infinitely divisible distributions}.
\newblock Cambridge university press, 1999.

\bibitem{10.1214/EJP.v5-68}
Jason Schweinsberg.
\newblock Coalescents with simultaneous multiple collisions.
\newblock {\em Electronic Journal of Probability}, 5(none):1 -- 50, 2000.

\bibitem{schweinsberg2000}
Jason Schweinsberg.
\newblock A necessary and sufficient condition for the {$\Lambda$}-coalescent
  to come down from infinity.
\newblock {\em Electronic Communications in Probability}, 5(none):1 -- 11,
  2000.

\bibitem{schweinsberg2003}
Jason Schweinsberg.
\newblock Coalescent processes obtained from supercritical {G}alton-{W}atson
  processes.
\newblock {\em Stochastic Processes and their Applications}, 106(1):107--139,
  2003.

\bibitem{10.1214/EJP.v17-2378}
Jason Schweinsberg.
\newblock {Dynamics of the evolving Bolthausen-Sznitman coalecent}.
\newblock {\em Electronic Journal of Probability}, 17(none):1 -- 50, 2012.

\bibitem{vechldp}
G.~V\'echambre.
\newblock Large deviation principle for the absorption time of the
  beta-coalescent via integral functionals.
\newblock {\em arXiv e-prints}, 2025.

\end{thebibliography}

\end{document}